%% file: HeiU21-ArXiv.tex
\documentclass[11pt,reqno]{amsart}

\usepackage[english]{babel}
\usepackage[utf8]{inputenc}
\usepackage[T1]{fontenc}
\usepackage{lmodern} \normalfont 
\DeclareFontShape{T1}{lmr}{bx}{sc} { <-> ssub * cmr/bx/sc }{}
\usepackage{microtype}	
\usepackage{xspace}

\usepackage{amssymb}
\usepackage{amsmath}
\usepackage{amsthm}
\usepackage{mathtools}
\usepackage{mathrsfs}
\usepackage{accents} 
\usepackage{etoolbox}
\usepackage{siunitx}
\usepackage{dsfont} 
\usepackage{graphicx}
\usepackage{color}
\usepackage[dvipsnames]{xcolor}
\usepackage{tikz}
\usetikzlibrary{calc,er,positioning}
\usepackage{pgfplots}
\pgfplotsset{compat=newest}
\usepackage[margin=10pt,font=small]{caption}
\usepackage{subcaption}

\usepackage{booktabs}
\usepackage{paralist}
\usepackage{soul}

\numberwithin{equation}{section}

\usepackage{algorithm}
\usepackage{algorithmicx}
\usepackage{algpseudocode}

\usepackage{cite}
\usepackage{multirow} 
\usepackage{enumitem} 
\setlist[enumerate]{label=(\roman*)}

\usepackage[colorlinks,citecolor=black,urlcolor=black]{hyperref}
\usepackage[nameinlink,noabbrev]{cleveref}

\theoremstyle{plain}
\newtheorem{Theorem}{Theorem}[section]
\newtheorem{Proposition}[Theorem]{Proposition}
\newtheorem{Lemma}[Theorem]{Lemma}
\newtheorem{Corollary}[Theorem]{Corollary}
\newtheorem{Remark}[Theorem]{Remark}
\newtheorem{Definition}[Theorem]{Definition}
\newtheorem{Assumption}[Theorem]{Assumption}
\newtheorem{Example}[Theorem]{Example}

\newcommand{\ddt}{\tfrac{\text{\normalfont d}}{\text{\normalfont d}t}}

\newcommand{\N}{\mathbb{N}}
\newcommand{\R}{\mathbb{R}}

\newcommand{\state}{x}
\newcommand{\stateDim}{n}

\newcommand{\dmdState}{\tilde{x}}

\newcommand{\timeStep}{h}
\newcommand{\nrSnapshots}{m}

\newcommand{\Ad}{A_\timeStep}

\newcommand{\dmdAg}{A_{\mathrm{DMD}}}  
\newcommand{\tdmdAg}{\tilde A_{\mathsf{DMD}}}  

\newcommand{\dmdX}{X}

\newcommand{\dmdY}{Z}
\newcommand{\frobnorm}[1]{\left\|#1\right\|_\mathrm{F}}

\newcommand{\pseudo}[1]{#1^\dagger}

\newcommand{\RKmatrix}{\mathcal{A}}
\newcommand{\RKvector}{\beta}

\DeclareMathOperator{\rank}{rank}

\DeclareMathOperator{\spann}{span}

\newcommand{\GL}[2]{\mathrm{GL}_{#1}(#2)}

\newcommand{\DMD}{\textsf{DMD}\xspace}
\newcommand{\IVP}{\textsf{IVP}\xspace}
\newcommand{\ODE}{\textsf{ODE}\xspace}
\newcommand{\RKM}{\textsf{RKM}\xspace}
\newcommand{\SVD}{\textsf{SVD}\xspace}

\title[Identification of LTI systems with DMD]{Identification of linear time-invariant systems with Dynamic Mode Decomposition}
\author{Jan Heiland${}^{\dagger,\ddagger}$ \and Benjamin Unger${}^\star$}
\address{${}^{\dagger}$  Max Planck Institute for Dynamics of Complex Technical Systems, Magdeburg, Germany}
\email{heiland@mpi-magdeburg.mpg.de}
\address{${}^{\ddagger}$ Faculty of Mathematics, Otto von Guericke University Magdeburg, Germany}
\email{jan.heiland@ovgu.de}
\address{${}^{\star}$ Stuttgart Center for Simulation Science (SC SimTech), University of Stuttgart, Universit\"{a}tsstr.~32, 70569 Stuttgart, Germany}
\email{benjamin.unger@simtech.uni-stuttgart.de}
\date{\today}

\begin{document}

\begin{abstract}
	Dynamic mode decomposition (\DMD) is a popular data-driven framework to extract linear dynamics from complex high-dimensional systems. In this work, we study the system identification properties of \DMD. We first show that \DMD is invariant under linear transformations in the image of the data matrix.  If, in addition, the data is constructed from a linear time-invariant system, then we prove that \DMD can recover the original dynamics under mild conditions. If the linear dynamics are discretized with a Runge-Kutta method, then we further classify the error of the \DMD approximation and detail that for one-stage Runge-Kutta methods even the continuous dynamics can be recovered with \DMD. A numerical example illustrates the theoretical findings.
\end{abstract}

\maketitle

\section{Introduction}

Dynamical systems play a fundamental role in many modern modeling approaches of physical and chemical phenomena. The need for high fidelity models often results in large-scale dynamical systems, which are computationally demanding to solve, analyze, and optimize. Thus the last three decades have seen significant efforts to replace the so-called full-order model, which is considered the \emph{truth model}, with a computationally cheaper surrogate model \cite{BenCOW17,BauBF14,QuaMN16,Ant05,BenGW15,HesRS16,AntBG20}. Often, the surrogate model is constructed by projecting the dynamical system onto a low-dimensional manifold, thus requiring a state-space description of the differential equation. 

If a mathematical model is not available or not suited for modification,
data-driven methods like the \emph{Loewner framework} \cite{MayA07,BeaG12},
\emph{vector fitting} \cite{GusS99,DrmGB15a,DrmGB15b}, \emph{operator inference}
\cite{PehW16}, or \emph{dynamic mode decomposition}
(\DMD) \cite{KutBBP16} may be used to create a low-dimensional realization
directly from measurement or simulation data of the system. Suppose the dynamical system that creates the data is linear. In that case, the Loewner framework and vector fitting are -- under some technical assumptions -- able to recover the original dynamical system and hence serve as system identification tools. Despite the popularity of \DMD, a similar analysis seems to be missing, and this paper aims to close this gap. 

Since \DMD creates a discrete, linear time-invariant dynamical system from data, we are interested answering the following questions:
\begin{enumerate}
	\item What is the impact of transformations of the data on the resulting \DMD approximation?
	\item Assume that the data used to generate the \DMD approximation is obtained from a linear differential equation. Can we estimate the error between the continuous dynamics and the \DMD approximation?
	\item Are there situations in which we are even able to recover the original dynamical system from the \DMD approximation?
\end{enumerate}

It is essential to know, how the data for the construction of the \DMD model is generated to answer these questions. Assuming exact measurements of the solution may be valid from a theoretical perspective only. Instead, we take the view of a numerical analyst and assume that the data is obtained via time integration of the dynamics with a general \emph{Runge-Kutta method} (\RKM) with known order of convergence. Thus we can summarize the questions graphically as in \Cref{fig:problemSetup}. Hereby the dashed line represents the questions that we aim to answer in this paper.

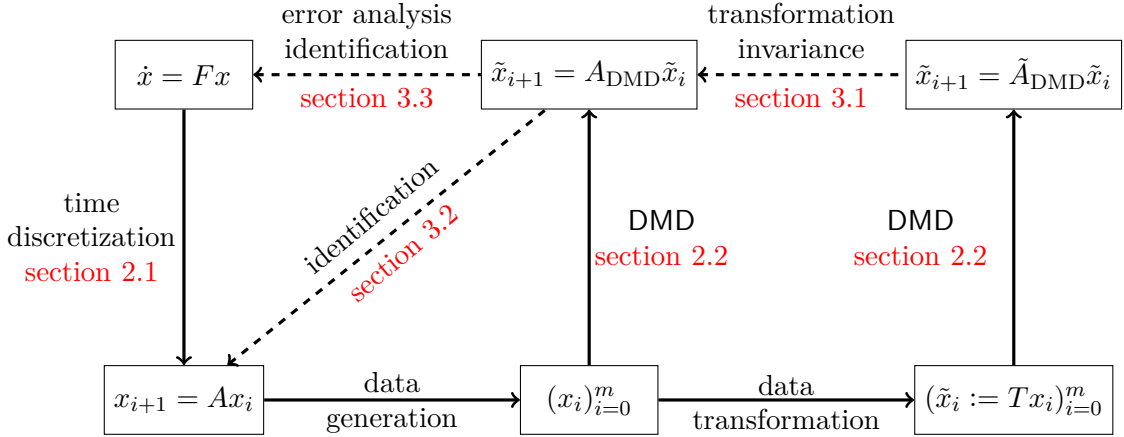
\begin{figure}[ht]
	\centering
	\begin{tikzpicture}[auto,node distance=3.4cm]
		\node[entity] (dynSys) {$\dot{\state} = F\state$};
		\node[entity] (discSys) [below = of dynSys]	{$\state_{i+1} = A\state_{i}$};
		\node[entity] (data) [right = of discSys] {$(\state_i)_{i=0}^{\nrSnapshots}$};
		\node[entity] (transformedData) [right = of data] {$(\tilde{\state}_i\vcentcolon=T\state_i)_{i=0}^{\nrSnapshots}$};
		\node[entity] (dmdSys) [above = of data]	{$\dmdState_{i+1} = A_{\mathrm{DMD}}\dmdState_i$};
		\node[entity] (transformedDmdSys) [above = of transformedData]	{$\dmdState_{i+1} = \tilde{A}_{\mathrm{DMD}}\dmdState_i$};
		
		\path[->,very thick] (dynSys) edge node[align=center,xshift=-2.5cm] {time\\discretization\\\cref{subsec:RKM}} (discSys);
      	\draw[->, very thick] (discSys) -- node[align=center,yshift=-1.5em] {data\\generation} (data);
      	\path[->,very thick] (data) edge node[align=center,xshift=2cm] {\DMD\\\cref{subsec:DMD}} (dmdSys);
      	\path[->,very thick] (data) edge node[align=center,yshift=-1.4em] {data\\transformation} (transformedData);
      	\path[->,very thick] (transformedData) edge node[align=center,xshift=-.2cm] {\DMD\\\cref{subsec:DMD}} (transformedDmdSys);
      	
      	\path[<-,dashed,very thick] (dynSys) edge node[align=center,yshift=-1.5em] {error analysis\\identification\\[.5em]\cref{subsec:continuousTime}} (dmdSys);
      	\path[<-,dashed,very thick] (dmdSys) edge node[align=center,yshift=-1.5em] {transformation\\invariance\\[.5em]\cref{subsec:dataTransformation}} (transformedDmdSys);
      	\path[<-,dashed,very thick] (discSys) edge node[align=center,xshift=3.5em,yshift=1em,rotate=40] {identification\\[.5em]\cref{subsec:discreteTime}} (dmdSys);
	\end{tikzpicture}
	\caption{Problem setup}
	\label{fig:problemSetup}
\end{figure}

\noindent Our main results are the following:
\begin{itemize}
	\item We show in \Cref{thm:invariance} that \DMD is invariant in the image of the data under linear transformations of the data.
	\item \Cref{thm:exactDMD} details that \DMD is able to identify to discrete-time dynamics, i.e., for every initial value in the image of the data, the \DMD approximation exactly recovers the discrete-time dynamics. 
	\item In Theorem~\ref{thm:DMDapproximationError} we show that if the \DMD approximation is constructed with data that is obtained via a \RKM, then the approximation error of \DMD with respect to the ordinary differential equation is in the order of the error of the \RKM. If a one-stage \RKM is used and the data is sufficiently rich, then the continous-time dynamics, i.e., the matrix $F$ in \Cref{fig:problemSetup} can be recovered, cf.~\Cref{lem:exactRKDMD}.
\end{itemize}

To render the manuscript self-contained, we recall important definitions and results for \RKM and \DMD in the upcoming \cref{subsec:RKM,subsec:DMD}, respectively, before we present our analysis in \cref{sec:errorAnalysis}. We conclude with a numerical example to confirm the theoretical findings.

\subsection*{Notation}
As is standard, $\mathbb{N}$ and $\mathbb{R}$ denote the positive integers and the real numbers, respectively. For any $n,m\in\mathbb{N}$, we denote with $\mathbb{R}^{n\times m}$ the set of $n\times m$ matrices with real entries. The set of nonsingular matrices of size $n\times n$ is denoted with $\GL{n}{\R}$. Let $A=[a_{ij}]\in\mathbb{R}^{n\times m}$, $B\in\mathbb{R}^{p\times q}$, and $x_i\in\mathbb{R}^{n}$ ($i=1,\ldots,k$). The transpose and the Moore-Penrose pseudoinverse of $A$ are denoted with $A^T$ and $\pseudo{A}$, respectively. The Kronecker product $\otimes$ is defined as 

\begin{displaymath}
	A\otimes B \vcentcolon= \left[\begin{smallmatrix} a_{11}B & \cdots & a_{1m}B\\
		\vdots & & \vdots\\
		a_{n1}B & \cdots & a_{nm}B
	\end{smallmatrix}\right]\in\mathbb{R}^{np\times mq}.
\end{displaymath}

We will use $\spann\{x_1,\ldots,x_k\}$ to denote the linear span of the vectors
$x_1,\ldots,x_k$ and also casually write $\spann \{X\}=\spann\{x_1,\ldots,x_k\}$
for the column space of the matrix $X$ with $\{x_1,\ldots,x_k\}$ as its columns.
For $A\in \mathbb R^{n\times n}$ and a vector $x_0\in \mathbb R^{n}$, we denote
the reachable space as $\mathcal C(x_0, A)=\spann\{x_0, Ax_0, \dotsc,
A^{n-1}x_0\}$. 
For a continuously differentiable function $\state:\mathbb{I}\to\mathbb{R}^n$ from the interval $\mathbb{I}\subseteq\mathbb{R}$ to the vector space $\mathbb{R}^n$ we use the notation $\dot{\state} \vcentcolon= \ddt \state$ to denote the derivative with respect to the independent variable $t$, which we refer to as the time.

\section{Preliminaries} 
As outlined in the introduction, \DMD creates a finite-dimensional linear model to approximate the
original dynamics. Thus, in view of possibly exact system identification, we need to assume that the data that is fed to the \DMD algorithm is obtained from a linear \ODE, which in the sequel is denoted by
\begin{subequations}
	\label{eq:IVP}
	\begin{equation}
		\label{eq:ODE}
		\dot{\state}(t) = F\state(t)
	\end{equation}
  with $F\in\mathbb{R}^{\stateDim\times\stateDim}$. To fix a solution of \eqref{eq:ODE}, we prescribe the initial condition
	\begin{equation}
		\label{eq:IC}
		\state(0) = \state_0\in\mathbb{R}^{\stateDim},
	\end{equation}
\end{subequations}
and denote the solution of the \emph{initial value problem} (\IVP) as $\state(t;\state_0) \vcentcolon= \exp(F t)\state_0$.

\begin{Remark}
  While a \DMD approximation, despite its linearity, may well reproduce
  trajectories of nonlinear systems (see, e.g., \cite{Mez05}), the question of \DMD being able to recover the full dynamics has to focus on linear systems. 
  Here, the key observation is that a \DMD approximation is a finite-dimensional linear map. In contrast, the encoding of nonlinear systems via a linear operator necessarily needs an infinite-dimensional mapping.
\end{Remark}

\subsection{Runge-Kutta methods}
\label{subsec:RKM}
To solve the \IVP~\eqref{eq:IVP} numerically, we employ a \RKM, which is a common one-step method to approximate ordinary and differential-algebraic equations \cite{HaiNW08,KunM06}. More precisely, given a step size $\timeStep>0$, the solution of the \IVP~\eqref{eq:IVP} is approximated via the sequence $\state_i \approx \state(t_0 + i\timeStep)$ given by
\begin{subequations}
	\label{eq:RKmethod}
	\begin{equation}
		\state_{i+1} = \state_i + \timeStep\sum_{j=1}^s \beta_j k_j,
	\end{equation}
	with the so-called \emph{internal stages} $k_j\in\mathbb{R}^\stateDim$ (implicitly) defined via
	\begin{equation}
		\label{eq:internalStages}k_j = F\state_i + \timeStep \sum_{\ell=1}^s \alpha_{j,\ell} F k_\ell \qquad\text{for}\ j=1,\ldots,s.
	\end{equation}
\end{subequations}
Using the matrix notation $\RKmatrix = [\alpha_{j,\ell}]\in\mathbb{R}^{s\times s}$ and $\RKvector = [\beta_{j}]\in\mathbb{R}^s$ the $s$-stage \RKM defined via \eqref{eq:RKmethod} is conveniently summarized with the pair $(\RKmatrix,\RKvector)$.  Note that we restrict our presentation to linear time-invariant dynamics and hence do not require the full Butcher tableau.

Since the \ODE~\eqref{eq:ODE} is linear, we can rewrite the internal stages as
\begin{equation}
	\label{eq:internalStages2}
	\begin{bmatrix}
		I_s - \timeStep \alpha_{1,1}F & -\timeStep \alpha_{1,2}F & \ldots & -\timeStep\alpha_{1,s}F\\
		-\timeStep\alpha_{2,1}F & I_s - \timeStep\alpha_{2,2}F & \ldots & -\timeStep\alpha_{2,s}F\\
		\vdots & \ddots & \ddots & \vdots\\
		-\timeStep\alpha_{s,1}F & \cdots & -\timeStep\alpha_{s,s-1}F & I_s - \timeStep\alpha_{s,s}F
	\end{bmatrix}
	\begin{bmatrix}
		k_1\\k_2\\\vdots\\k_s
	\end{bmatrix} = \begin{bmatrix}
		F\state_i\\
		F\state_i\\
		\vdots\\
		F\state_i
	\end{bmatrix}
\end{equation}
Setting $k\vcentcolon= \begin{bmatrix}
	k_1^T & \ldots & k_s^T
\end{bmatrix}^T\in\mathbb{R}^{sn}$ and $e \vcentcolon= \begin{bmatrix} 1 & \ldots 1\end{bmatrix}^T\in\mathbb{R}^s$, the linear system in \eqref{eq:internalStages2} can be written as
\begin{equation}
	(I_s\otimes I_{\stateDim} - \timeStep\RKmatrix\otimes F)k = (e\otimes F)\state_i,
\end{equation}
where $\otimes$ denotes the Kronecker product. If $\timeStep$ is small enough, the matrix $(I_s\otimes I_{\stateDim} - \timeStep\RKmatrix\otimes F)$ is invertible and thus we obtain the discrete linear system
\begin{align*}
	\state_{i+1} &= \state_i + \timeStep\sum_{j=1}^s \beta_jk_j = \state_i + \timeStep (\RKvector^T\otimes I_{\stateDim})k\\
	&= \state_i + \timeStep (\RKvector^T\otimes I_{\stateDim}) \left(I_s\otimes I_{\stateDim} - \timeStep\RKmatrix\otimes F\right)^{-1}(e\otimes F)\state_i = A_{\timeStep}\state_i,
\end{align*}
with 
\begin{equation}
	\label{eq:discretizationRungeKutta}
	A_{\timeStep} \vcentcolon= I_{\stateDim} + \timeStep (\RKvector^T\otimes I_{\stateDim}) \left(I_s\otimes I_{\stateDim} - \timeStep\RKmatrix\otimes F\right)^{-1}(e\otimes F).
\end{equation}

\begin{Example}
	The explicit Euler method is given as $(\RKmatrix,\RKvector) = (0,1)$ and according to \eqref{eq:discretizationRungeKutta} we obtain the well-known formula $A_{\timeStep} = I_{\stateDim} + \timeStep F$. For the implicit Euler method $(\RKmatrix,\RKvector) = (1,1)$ the discrete system matrix is given by 
	\begin{displaymath}
		A_{\timeStep} = I_{\stateDim} + \timeStep(I_{\stateDim} - \timeStep F)^{-1}F = (I_{\stateDim} - \timeStep F)^{-1}(I_{\stateDim} - \timeStep F + \timeStep F) = (I_{\stateDim}-\timeStep F)^{-1}.
	\end{displaymath}
\end{Example}

To guarantee that the representation~\eqref{eq:discretizationRungeKutta} is valid, we will make the following assumption throughout the manuscript.

\begin{Assumption}
	\label{ass:suitableStepsize}
	For any $s$-stage \RKM $(\RKmatrix,\RKvector)$ and any dynamical system matrix $F\in\mathbb{R}^{\stateDim\times\stateDim}$ we assume that the step size $\timeStep$ is chosen such that the matrix $I_{s\stateDim} - \timeStep \RKmatrix\otimes F$ is nonsingular.
\end{Assumption}

\begin{Remark}
	Using \Cref{ass:suitableStepsize}, the matrix $I_{s\stateDim} - \timeStep \RKmatrix\otimes F$ is nonsingular and thus there exists a polynomial $p = \sum_{k=0}^{s\stateDim-1} p_kt^k \in\mathbb{R}[t]$ of degree at most $s\stateDim-1$ depending on the step size $\timeStep$ such that
	\begin{align*}
		\left(I_{s\stateDim} - \timeStep \RKmatrix\otimes F\right)^{-1} &= p(I_{s\stateDim} - \timeStep \RKmatrix\otimes F) = \sum_{k=0}^{s\stateDim-1} p_k \left(I_{s\stateDim} - \timeStep \RKmatrix\otimes F\right)^k\\
		&= \sum_{k=0}^{s\stateDim-1}p_k\sum_{\rho=0}^k \binom{k}{\rho} h^{\rho} (\RKmatrix^\rho\otimes F^\rho),
	\end{align*}
	where the last equality follows from the binomial theorem. Consequently, we have
	\begin{align}
    \label{eq:ah-as-polynomial}
		A_{\timeStep} = I_{\stateDim} + \sum_{k=0}^{sn-1}p_k\sum_{\rho=0}^k \binom{k}{\rho} h^{\rho+1} \left(\beta^T\RKmatrix^\rho e\right) F^{\rho+1}.
	\end{align}
	Rearranging the terms together with the Cayley-Hamilton theorem implies the existence of a polynomial $\tilde{p}\in\mathbb{R}[t]$ of degree at most $\stateDim$ such that $A_{\timeStep} = \tilde{p}(F)$. As a direct consequence, we see that any eigenvector of $F$ is an eigenvector of $A_{\timeStep}$ and thus $A_{\timeStep}$ is diagonalizable if $F$ is diagonalizable.
\end{Remark}

Having computed the matrix $A_{\timeStep}$, the question that remains to be answered is the quality of the approximation $\|\state(i\timeStep;\state_0)-\state_i\|$, which yields the following well-known definition (cf.~\cite{HaiNW08}).

\begin{Definition}
	A \RKM $(\RKmatrix,\RKvector)$ has \emph{order $p$} if there exists a constant $C\geq 0$ (independent of $h$) such that
	\begin{equation}
		\label{eq:locerror-estimate}
		\|\state(\timeStep;\state_0)-\state_1\| \leq C\timeStep^{p+1}
	\end{equation}
	holds, where $\state_1 = A_{\timeStep}\state_0$ with $A_{\timeStep}$ defined as in \eqref{eq:discretizationRungeKutta}.
\end{Definition}

For one-step methods, it is well-known that the local errors -- as estimated in \eqref{eq:locerror-estimate} for the initial time step -- basically sum in the global error, such that the following estimate holds 

\begin{displaymath}
	\|\state(N\timeStep;x_0)-x_N\| \leq C\timeStep^p;
\end{displaymath}
see, e.g., \cite[Thm.~II.3.6]{HaiNW08}.

\subsection{Dynamic Mode Decomposition}
\label{subsec:DMD}
For $i=0,\ldots,\nrSnapshots$, assume data points $\state_i\in\R^{\stateDim}$ available.  The idea of \DMD is to determine a linear time-invariant relation between the data, i.e., finding a matrix $A_{\mathrm{DMD}}\in\R^{\stateDim\times \stateDim}$, such that the data approximately satisfies
\begin{displaymath}
	\state_{i+1} \approx A_{\mathrm{DMD}} \state_i\qquad \text{for $i=0,1,\ldots,\nrSnapshots-1$}.
\end{displaymath}
Following \cite{TuRLBK14} we introduce
\begin{equation}
	\label{eq:DMDdata}
	\dmdX \vcentcolon= \begin{bmatrix}
		\state_0 & \ldots & \state_{\nrSnapshots-1}
	\end{bmatrix}\in\mathbb{R}^{\stateDim\times\nrSnapshots}\qquad\text{and}\qquad
	\dmdY \vcentcolon= \begin{bmatrix}
		\state_1 & \ldots & \state_{\nrSnapshots}
	\end{bmatrix}\in\mathbb{R}^{\stateDim\times\nrSnapshots}.
\end{equation}
Then, the \DMD approximation matrix is defined as the minimum-norm solution of
\begin{equation}
	\label{eq:DMDminFrob}
	\min_{M\in\mathbb{R}^{\stateDim\times\stateDim}} \|\dmdY - M\dmdX\|_F,
\end{equation}
where $\|\cdot\|_F$ denotes the Frobenius norm. It is easy to show that the minimum-norm solution is given by $A_{\mathrm{DMD}} = \dmdY\pseudo{\dmdX}$ \cite{KutBBP16}, where $\pseudo{\dmdX}$ denotes the Moore-Penrose pseudoinverse of $\dmdX$. This motivates the following definition.

\begin{Definition}
	\label{def:exactDMD}
	Consider the data $\state_i\in\R^{\stateDim}$ for $i=0,1,\ldots,\nrSnapshots$ and associated data matrices $\dmdX$ and $\dmdY$ defined in~\eqref{eq:DMDdata}. Then the matrix $A_{\mathrm{DMD}} \vcentcolon= \dmdY\pseudo{\dmdX}$ is called the \emph{\DMD matrix} for $(\state_i)_{i=0}^{\nrSnapshots}$.  If the eigendecomposition of $A_{\mathrm{DMD}}$ exists, then the eigenvalues and eigenvectors of $A_{\mathrm{DMD}}$ are called \emph{\DMD eigenvalues} and \emph{\DMD modes} of $(state_i)_{i=0}^{\nrSnapshots}$, respectively.
\end{Definition}

The Moore-Penrose pseudoinverse, and thus also the \DMD matrix, can be computed
via the \emph{singular value decomposition} (\SVD); see, e.g., \cite[Ch.
5.5.4]{GolV96}: Let
\begin{displaymath}
	\begin{bmatrix}
		U & \bar{U}
	\end{bmatrix}\begin{bmatrix}
		\Sigma & 0\\
		0 & 0
	\end{bmatrix}\begin{bmatrix}
		V^\top\\
		\bar{V}^\top
	\end{bmatrix} = \dmdX
\end{displaymath}
denote the \SVD of $\dmdX$, with $r \vcentcolon= \rank(\dmdX)$,  $U\in\R^{\stateDim\times r}$, $\Sigma\in\R^{r\times r}$ and $\rank(\Sigma) = r$, and $V\in\R^{\nrSnapshots\times r}$. 
Then
\begin{equation}
  \label{eq:moore-penrose-def-svd}
  \pseudo X = 
  \begin{bmatrix}
		V & \bar{V}
	\end{bmatrix}
  \begin{bmatrix}
    \Sigma^{-1} & 0\\
		0 & 0
	\end{bmatrix}
  \begin{bmatrix}
		U^\top\\
		\bar{U}^\top
	\end{bmatrix} = V\Sigma^{-1}U^\top
\end{equation}
and, thus,
\begin{equation}
	\label{eq:DMDrealization}
	A_{\mathrm{DMD}} = \dmdY V\Sigma^{-1}U^T.
\end{equation}
For later reference, we call $U\Sigma V^\top = \dmdX$ the \emph{trimmed \SVD} of $\dmdX$.

\section{System identification and error analysis}
\label{sec:errorAnalysis}

In this section we present our main results. Before discussing system identification for discrete-time (cf.~\cref{subsec:discreteTime}) and continuous-time (cf.~\cref{subsec:continuousTime}) dynamical systems via \DMD, we study the impact of transformations of the data on \DMD in \cref{subsec:dataTransformation}.

\subsection{Data scaling and invariance of the DMD approximation}
\label{subsec:dataTransformation}

Scaling and more general transformation of data is often used to improve the performance of the methods that work on the data. Since \DMD is inherently related to the Moore-Penrose inverse, we first study the impact of a nonsingular matrix $T\in\GL{\stateDim}{\R}$ on the generalized inverse. To this purpose, consider a matrix $X\in\R^{\stateDim\times\nrSnapshots}$ with $r\vcentcolon= \rank(X)$. Let $X = U\Sigma V^\top$ denote the trimmed \SVD of $X$ with $U\in\R^{\stateDim\times r}$, $\Sigma\in\GL{r}{\R}$ and $V\in\R^{\nrSnapshots\times r}$. Let $TU = QR$ denote the QR-decomposition of $TU$ with $Q\in\R^{\stateDim\times \stateDim}$ and $R\in\R^{\stateDim\times r}$. We immediately obtain $\rank(RS) = r$. Let $R\Sigma = \widehat{U}\widehat{\Sigma}\widehat{V}^\top$ denote the trimmed \SVD of $R\Sigma$ with $\widehat{U}\in\R^{\stateDim\times r}$, $\widehat{\Sigma}\in\GL{r}{\R}$, and $\widehat{V}\in\R^{r\times r}$.  We immediately infer
\begin{equation}
	\label{eqn:transformedRightSingularVector}
	\widehat{V}\widehat{V}^\top = I_r. 
\end{equation}
It is easy to see that the matrices $U_T \vcentcolon= Q\widehat{U}\in\R^{\stateDim\times r}$, and $V_T \vcentcolon= V\widehat{V}\in\R^{\nrSnapshots\times r}$ satisfy $U_T^\top U_T = I_r = V_T^\top V_T$. The trimmed \SVD of $TX$ is thus given by
\begin{displaymath}
	TX = TU\Sigma V^\top = QR\Sigma V^\top = Q\widehat{U}\widehat{\Sigma}\widehat{V}^\top V^\top = U_T\widehat{\Sigma}V_T^\top.
\end{displaymath}
We thus obtain
\begin{equation*}
	\pseudo{(TX)}TX = V_T V_T\top = V\widehat{V}\widehat{V}^\top V^\top = VV^\top = \pseudo{X}\pseudo{X}
\end{equation*}
where we have used the identity~\eqref{eqn:transformedRightSingularVector}. We have thus shown the following result.

\begin{Proposition}
	\label{prop:MPItransformation}
	Let $X\in\R^{\stateDim\times\nrSnapshots}$ and $T\in\GL{\stateDim}{\R}$. Then  $\pseudo{(TX)}(TX) = \pseudo{X}{X}$.
\end{Proposition}

With these preparations, we can now show that the \DMD approximation is partially invariant to general regular transformations applied to the training data. More precisely, a data transformation only affects the part of the \DMD approximation that is not in the image of the data. 

\begin{Theorem}
	\label{thm:invariance}
	For given data $(x_i)_{i=0}^{\nrSnapshots}$ consider the matrices $\dmdX$ and $\dmdY$ as defined in~\eqref{eq:DMDdata} and the corresponding \DMD matrix $\dmdAg\in\R^{\stateDim\times\stateDim}$. Consider $T\in \GL{n}{\R}$ and let
	\begin{equation*}
		\tilde{\dmdX} \vcentcolon= T\dmdX 
		\quad\text{and}\quad
		\tilde{\dmdY} \vcentcolon= T\dmdY
	\end{equation*}
	be the matrices of the transformed data. Let $\tdmdAg \vcentcolon= \tilde{\dmdY}\pseudo{\tilde{\dmdX}}$ denote the \DMD matrix for the transformed data.
	Then the \DMD matrix is invariant under the transformation in the image of~$\dmdX$, i.e.,
	\begin{equation*}
		\dmdAg \dmdX = T^{-1}\tdmdAg T \dmdX = T^{-1}\tdmdAg \tilde{\dmdX}.
	\end{equation*}
  	Moreover, if $T$ is unitary or $\rank(X) = n$, then
  \begin{equation}
  \label{eq:DMDtransformation}
    \dmdAg = T^{-1}\tdmdAg T.
  \end{equation}
\end{Theorem}

\begin{proof}
	Using \Cref{prop:MPItransformation} we obtain
	\begin{equation*}
		T^{-1}\tdmdAg T\dmdX = T^{-1} T\dmdY \pseudo{(T\dmdX)}T\dmdX = \dmdY\pseudo{\dmdX}\dmdX = \dmdAg\dmdX.
	\end{equation*}
	If $T$ is unitary or $\rank(X) = \stateDim$, then we immediately obtain $\pseudo{(TX)} = \pseudo{X}T^{-1}$, and thus
  \begin{displaymath}
  	T^{-1}\tdmdAg T= T^{-1}T\dmdY\pseudo{T\dmdX}T = \dmdY \pseudo{\dmdX}T^{-1} T = \dmdAg,
  \end{displaymath}
  which concludes the proof.
\end{proof}

While \Cref{thm:invariance} states that \DMD is invariant under transformations
in the image of the data matrix, the invariance in the orthogonal complement of
the image of the data matrix, i.e., equality~\eqref{eq:DMDtransformation}, is in
general not satisfied. We illustrate this observation in the numerical simulations in a
later chapter and in the following analytical example.

\begin{Example}
	Consider the data vectors $\state_i \vcentcolon= [i+1,0]^\top$ for $i=0,1,2$ and $T \vcentcolon= \left[\begin{smallmatrix}1 & 0\\ 1 & 1\end{smallmatrix}\right]$. Then,
	\begin{displaymath}
		\dmdX = \begin{bmatrix}
			1 & 2\\
			0 & 0
		\end{bmatrix}, \quad \dmdY = \begin{bmatrix}
			2 & 3\\
			0 & 0
		\end{bmatrix},  \quad \pseudo{\dmdX} = \tfrac{1}{5}\begin{bmatrix}
			1 & 0\\
			2 & 0
		\end{bmatrix},\quad T\dmdX = \begin{bmatrix}
			1 & 2\\
			1 & 2
		\end{bmatrix}, \quad \pseudo{(T\dmdX)} = \tfrac{1}{10}\begin{bmatrix}
			1 & 2\\
			1 & 2
		\end{bmatrix}.
	\end{displaymath}
	We thus obtain
	\begin{equation*}
		\dmdAg = \tfrac{1}{5}\begin{bmatrix}
			8 & 0\\
			0 & 0
		\end{bmatrix},\qquad
		\tdmdAg = \tfrac{1}{5}\begin{bmatrix}
			4 & 4\\
			4 & 4
		\end{bmatrix},\qquad \text{and}\qquad
		T^{-1}\tdmdAg T = \tfrac{1}{5}\begin{bmatrix}
			8 & 4\\
			0 & 0
		\end{bmatrix},
	\end{equation*}
	confirming that \DMD is invariant under transformations in the image of the data, but not in the orthogonal complement.
\end{Example}

\begin{Remark}
	One can show that in the setting from \Cref{thm:invariance},  the matrix $\widehat{M} \vcentcolon= TA_{\mathrm{DMD}} T^{-1}$ is a minimizer (not necessarily the minimum-norm solution) of 
	\begin{equation*}
		\min_{M\in\mathbb{R}^{\stateDim\times\stateDim}} \frobnorm{\widehat{\dmdY}-M\widehat{\dmdX}}.
	\end{equation*}
\end{Remark}

\subsection{Discrete-time dynamics}
\label{subsec:discreteTime}

In this subsection, we focus on the identification of discrete-time dynamics, which are exemplified by the discrete-time system
\begin{equation}
	\label{eqn:discreteSystem}
	\state_{i+1} = A\state_i
\end{equation}
with initial value $\state_0\in\R^{\stateDim}$ and system matrix $A\in\R^{\stateDim\times\stateDim}$.  The question that we want to answer is to what extend \DMD is able to recover the matrix $A$ solely from data.

\begin{Proposition}
	\label{prop:discreteRealization}
	Consider data $(\state_i)_{i=0}^{\nrSnapshots}$ generated by~\eqref{eqn:discreteSystem}, associated data matrices $\dmdX, \dmdY$ as defined in~\eqref{eq:DMDdata}, and the corresponding \DMD matrix $A_{\mathrm{DMD}}$.  Moreover let $U\Sigma V^\top = \dmdX$ with $U\in\R^{\stateDim\times r}$, $\Sigma\in\GL{r}{\R}$, $V\in\R^{\nrSnapshots\times r}$, and $r \vcentcolon= \rank(\dmdX)$ denote the trimmed \SVD of $\dmdX$ defined in \eqref{eq:DMDdata}. Then
	\begin{equation}
		\label{eq:dmdA-skinnySVD}
		A_{\mathrm{DMD}} = A UU^\top.
	\end{equation}
\end{Proposition}

\begin{proof}
	By assumption, we have
	$\dmdX = \begin{bmatrix}
			\state_0 & A\state_0 & \cdots & A^{\nrSnapshots-1}\state_0
		\end{bmatrix}$ and $\dmdY = A\dmdX = A U\Sigma V^\top$.  We conclude
	\begin{equation*}
		A_{\mathrm{DMD}} = \dmdY\pseudo{\dmdX} = A U\Sigma V^\top V\Sigma^{-1}U^\top = A UU^\top.\qquad\qedhere
	\end{equation*}
\end{proof}

\begin{Remark}
\label{rem:controllability-and-full-rank}
We immediately conclude that \DMD recovers the true dynamics, i.e.,
$A_{\mathrm{DMD}} = A$, whenever $\rank(\dmdX) = \stateDim$. This is the case if
and only if $(A,\state_0)$ is controllable, i.e., $\mathcal C(A, \state_0)$ has
dimension $\stateDim$, and the data set is sufficiently rich, i.e., $\nrSnapshots \geq \stateDim$. 
\end{Remark}

Our next theorem identifies the part of the dynamics that is exactly recovered in the case
that $\rank(\dmdX) < \stateDim$ that occurs for $(A, \state_0)$ not controllable
or $\nrSnapshots < \stateDim$.
\begin{Theorem}
	\label{thm:exactDMD}
	Consider the setting of \Cref{prop:discreteRealization}.  If $\spann\{U\}$ is $A_{\mathrm{DMD}}$-invariant, then the DMD approximation is exact in the image of $U$, i.e.,
	\begin{equation}
		\label{eqn:exactDMD}
		(A^i-A_{\mathrm{DMD}}^i)x = 0\qquad\text{for all $i\geq 0$ and $x\in \spann\{U\}$.}
	\end{equation}
	If in addition, $\ker(A)\cap \spann\{U\}^\perp = \{0\}$, then also the converse direction holds.
\end{Theorem}

\begin{proof}
	Let $x\in \spann\{U\}$. Since $\spann\{U\}$ is $A_{\mathrm{DMD}}$ invariant, we conclude $A_{\mathrm{DMD}}^ix\in\spann\{U\}$ for $i\geq 0$, i.e., there exist $y_i\in\R^r$ such that $A_{\mathrm{\DMD}}^ix = Uy_i$. Using \Cref{prop:discreteRealization} we conclude
	\begin{equation*}
		A_{\mathrm{DMD}}^{i+1}x = A_{\mathrm{DMD}} A_{\mathrm{DMD}}^i x = A_{\mathrm{DMD}} Uy_i = Ay_i.
	\end{equation*}
	The proof of~\eqref{eqn:exactDMD} follows via induction over $i$. For the converse direction, let $x = x_U + x_U^\perp$ with $x_U\in\spann\{U\}$ and $x_U^\perp \in \spann\{U\}^\perp$.  \Cref{prop:discreteRealization} and~\eqref{eqn:exactDMD} imply
	\begin{equation*}
		(A-A_{\mathrm{DMD}})x = Ax_U^\perp \neq 0,
	\end{equation*}
	which completes the proof.
\end{proof}

\begin{Remark}
	The proof of~\Cref{thm:exactDMD} details that $\spann\{U\}$ is $A_{\mathrm{DMD}}$-invariant if and only if $\spann\{U\}$ is $A$-invariant.  Moreover, $\spann\{U\} = \spann\{\dmdX\}$ implies that this condition can be checked easily during the data-generation process.  If we further assume that the data is generated via~\eqref{eqn:discreteSystem}, then this is the case,  whenever
	\begin{equation*}
		\rank\left(\begin{bmatrix}
			\state_0 & \cdots & \state_i
		\end{bmatrix}) = \rank(\begin{bmatrix}
			\state_0 & \cdots & \state_{i+1}
		\end{bmatrix}\right)
	\end{equation*}
	for some $i\geq 0$.
\end{Remark}

\subsection{Continuous-time dynamics and RK approximation}
\label{subsec:continuousTime}
 
Suppose now that the data $(\state_i)_{i=0}^{\nrSnapshots}$ is generated from a continuous process, i.e., via the dynamical system~\eqref{eq:IVP}. In this case, we are interested in recovering the continuous dynamics from the \DMD approximation. As a consequence of \Cref{thm:exactDMD} we immediately obtain the following results for exact sampling.

\begin{Corollary}
	\label{cor:exactDMD}
	Let $A_{\mathrm{DMD}}$ be the \DMD matrix for the sequence $\state_i = \exp(iF \timeStep)\state_0\in\mathbb{R}^{\stateDim}$ for $i=1,\ldots,\nrSnapshots$ with $m\geq n$.  Then 
	\begin{equation*}
		\state(i\timeStep;\tilde{\state}_0) = A_{\mathrm{DMD}}^i \tilde{\state}_0
	\end{equation*}
	if and only if $\tilde{\state}_0\in\spann\{\state_0,\ldots,\state_{\nrSnapshots}\}$, where $\state(t;\tilde{\state}_0)$ denotes the solution of the \IVP~\eqref{eq:IVP} with initial value $\tilde{\state}_0$.
\end{Corollary}

\begin{proof}
	The assertion follows immediately from \Cref{prop:discreteRealization} with the observation that $\exp(i F\timeStep)$ is nonsingular.
\end{proof}

We conclude that we can recover the continuous dynamics with the matrix logarithm (see \cite{Hig08} for further details), whenever $\rank(\dmdX) = \stateDim$. In practical applications, an exact evaluation of the flow map is typically not possible. Instead, a numerical time-integration method is used to approximate the continuous dynamics.  

Suppose we have used a \RKM with constant step size $\timeStep>0$ to obtain a numerical approximation $(\state_i)_{i=0}^{\nrSnapshots}\subseteq\mathbb{R}^{\stateDim}$ of the \IVP~\eqref{eq:IVP} and used this data to construct the \DMD matrix $A_{\mathrm{DMD}}\in\mathbb{R}^{\stateDim\times\stateDim}$ as in \Cref{def:exactDMD}. If we now want to use the \DMD matrix to obtain an approximation for a different initial condition, say $\state(0) = \tilde{\state}_0$, we are interested in quantifying the error
\begin{equation*}
	\|\state(i\timeStep;\tilde{\state}_0) - A_{\mathrm{DMD}}^i\tilde{\state}_0\|.
\end{equation*}

\begin{Theorem}
	\label{thm:DMDapproximationError}
	Suppose that the sequence $(\state_i)_{i=0}^\nrSnapshots$, with
  $\state_i\in\R^{\stateDim}$ for $i=0,\ldots,\nrSnapshots$, is generated from the linear \IVP~\eqref{eq:IVP} via a \RKM of order $p$ and step size $\timeStep>0$ and satisfies
	\begin{equation*}
		\spann\{\state_0,\ldots,\state_{\nrSnapshots-1}\} = \spann\{\state_0,\ldots,\state_{\nrSnapshots}\}. 
	\end{equation*}
	Let $A_{\mathrm{DMD}}\in\R^{\stateDim\times\stateDim}$ denote the associated \DMD matrix. Then there exists a constant $C\geq 0$ such that
	\begin{equation}
		\label{eq:errorEstimate}
		\|\state(ih;\tilde{\state}_0)-A_{\mathrm{DMD}}^i\tilde{\state}_0\| \leq Ch^p
	\end{equation}
	holds for any $\tilde{\state}_0\in\spann(\{\state_0, \ldots,\state_{\nrSnapshots-1}\})$.
\end{Theorem}

\begin{proof}
	Since the data $(\state_i)_{i=0}^\nrSnapshots$ is generated from a \RKM there
  exists a matrix $A_h\in\R^{\stateDim\times\stateDim}$ such that $\state_{i+1}
  = A_h\state_i$ for $i=0,\ldots,\nrSnapshots-1$. Let $\tilde{\state}_0\in
  \spann(\{\state_0, \ldots,\state_{\nrSnapshots-1}\})$. Then,
  \Cref{thm:exactDMD} implies $A_h^i\tilde{\state}_0 =
  A_{\mathrm{DMD}}^i\tilde{\state}_0$ for any $i\geq 0$. Thus the result follows
  from the classical error estimates for RKM (see, e.g.,
  \cite[Thm.~II.3.6]{HaiNW08}) and the from equality
	\begin{equation*}
		\|\state(i\timeStep;\tilde{\state}_0) - A_{\mathrm{DMD}}^i\tilde{\state}_0\| = \|\state(i\timeStep;\tilde{\state}_0) - \Ad^i\tilde{\state}_0\| \leq C\timeStep^p
	\end{equation*}
	for some $C\geq 0$, since the \RKM is of order $p$.
\end{proof}

The proof details that due to \Cref{prop:discreteRealization} we are essentially able to recover the discrete dynamics $A_h$ obtained from the \RKM via \DMD, provided that $\rank(\dmdX) = \stateDim$. As layed out in \Cref{rem:controllability-and-full-rank}, this condition is equivalent to $(A_h,\state_0)$ controllable for which controllability of $(F,\state_0)$ is a necessary condition.

The question that remains to be answered is whether it is possible to recover the continuous dynamic matrix $F$ from the discrete dynamics $A_{\mathrm{DMD}}$ (respectively $A_h$) provided that the Runge-Kutta scheme is known that was used to discretize the continuous dynamics. For any $1$-state Runge-Kutta method $(\alpha,\beta)$, i.e, $s=1$ in~\eqref{eq:RKmethod} this is indeed the case,  since then \eqref{eq:discretizationRungeKutta} simplifies to
\begin{equation*}
	A_h = I_{\stateDim} + \timeStep \beta (I_{\stateDim} - \timeStep\alpha F)^{-1}F,
\end{equation*}
which yields
\begin{equation*}
	F = -\frac{1}{\timeStep}(I_{\stateDim} - A_h)\left(\alpha A_h + (\beta-\alpha)I_{\stateDim}\right)^{-1}.
\end{equation*}
Combining \eqref{eq:reverseRKs1} with \Cref{prop:discreteRealization} yields the following result.

\begin{Lemma}
	\label{lem:exactRKDMD}
	Suppose that the sequence $(\state_i)_{i=0}^\nrSnapshots\subseteq\mathbb{R}^{\stateDim}$ is generated from the linear \IVP~\eqref{eq:IVP} via the $1$-stage Runge-Kutta method $(\alpha,\beta)$ and step size $\timeStep>0$. Let $A_{\mathrm{DMD}}\in\mathbb{R}^{\stateDim\times\stateDim}$ denote the associated \DMD matrix. If $\rank(\{\state_0,\ldots,\state_{\nrSnapshots-1}\}) = \stateDim$, then
	\begin{equation}
		\label{eq:reverseRKs1}
		F = -\frac{1}{\timeStep}(I_{\stateDim} - A_{\mathrm{DMD}})\left(\alpha A_{\mathrm{DMD}} + (\beta-\alpha)I_{\stateDim}\right)^{-1},
	\end{equation}
	provided the inverse exists. 
\end{Lemma}

If the assumption of \Cref{lem:exactRKDMD} holds, then we can recover the continuous dynamic matrix from the \DMD approximation. The corresponding formula for popular one-stage methods is presented in \Cref{tab:continuousDynamicsFromDMD}. 
\begin{table}[ht]
	\centering
	\caption{Identification of continous-time systems via \DMD with one-stage Runge-Kutta methods}
	\label{tab:continuousDynamicsFromDMD}
	\begin{tabular}{lll}
		\toprule
		\textbf{method} & $(\alpha,\beta)$ & \Cref{lem:exactRKDMD}\\\midrule
		explicit Euler & $(0,1)$ & $F = -\tfrac{1}{\timeStep}(I_{\stateDim}-A_{\mathrm{DMD}})$\\
		implicit Euler & $(1,1)$ & $F = \frac{1}{\timeStep}(I_{\stateDim}-A_{\mathrm{DMD}}^{-1})$\\
		implicit midpoint rule & $(\tfrac{1}{2},1)$ & $F = \tfrac{1}{2h}(A_{\mathrm{DMD}}-I_{\stateDim})(A_{\mathrm{DMD}}+I_{\stateDim})^{-1}$\\\bottomrule
	\end{tabular}
\end{table}
In this scenario, let us emphasize that we can compute the discrete dynamics with the \DMD approximation for any time-step.

The situation is different for $s\geq 2$, as we illustrate with the following example.

\begin{Example}
	For given $\timeStep>0$, consider $F_1 \vcentcolon= 0$ and $F_2 \vcentcolon= -\tfrac{2}{\timeStep}$.  Then, for Heun's method, i.e., $\RKmatrix  = \left[\begin{smallmatrix} 0 & 0\\1 & 0\end{smallmatrix}\right]$ and $\RKvector^\top = \left[\begin{smallmatrix} \tfrac{1}{2} & \tfrac{1}{2}\end{smallmatrix}\right]$, we obtain $A_h = p(F)$ with $p(x) = 1 + hx + \tfrac{h^2}{2}x^2$, and thus $p(F_1) = p(F_2)$.  In particular, we cannot distinguish the continuous-time dynamics in this specific scenario.
\end{Example}

\section{Numerical examples}
\providecommand\inival{\state_0}  
\providecommand\diagmat[1]{\ensuremath{\mathsf{diag}(#1)}}

To illustrate our analytical findings, we have constructed a dynamical system that exhibits some fast dynamics that is stable but not exponentially stable and has a nontrivial but exactly computable flow map. In this way, we can check the approximation both qualitatively and quantitatively. Also, the system can be scaled to arbitrary state-space dimensions. Most importantly, for our purposes, the system is designed such that for any initial value, the space not reached by the system is as least as large as the reachable space.
The complete code of our numerical examples can be found in the
\emph{supplementary material}.

With $N\in \N$, $\Delta\vcentcolon= \diagmat{0,1,\dotsc,N-1}$ we consider the continuous-time dynamics~\eqref{eq:IVP} with 
\begin{equation*}
	F \vcentcolon=
	\begin{bmatrix}
		0 & 2\Delta \\
		0 & -\frac 12 \Delta
	\end{bmatrix}\qquad\text{and}\qquad \exp(tF) = \begin{bmatrix}
		I & 4(I-\exp(-\frac{t}{2}\Delta)) \\
		0 & \exp(-\frac{t}{2} \Delta)
	\end{bmatrix}.
\end{equation*}
Starting with an initial value $\state_0\in \R^{2N}$ we can thus generate exact snapshots of the solution via $\state(t) = \exp(tF)\state_0$, as well as the controllability space
\begin{equation*}
\mathcal{C}(F,\state_0) = \spann\biggl\{
	\inival, 
	\begin{bmatrix}
		0 & 2\Delta \\
		0 & -\frac 12 \Delta
	\end{bmatrix} \inival,
	\begin{bmatrix}
		0 & 2\Delta \\
		0 & -\frac 12 \Delta
	\end{bmatrix}^2 \inival,
  \dotsc,
	\begin{bmatrix}
		0 & 2\Delta \\
		0 & -\frac 12 \Delta
  \end{bmatrix}^{2N-1} \inival
\biggr\}.
\end{equation*}
It is easy to see that $\dim(\mathcal{C}(F,\state_0))\leq N$ with equality if
and only if $\state_0$ has no zero entries. Due to~\eqref{eq:ah-as-polynomial}, we immediately infer
\begin{equation*}
	\dim(\mathcal{C}(A_h,\state_0)) \leq N
\end{equation*}
for any $A_h$ obtained by a Runge-Kutta method. We conclude that \DMD will at most be capable to
reproduce solutions that evolve in~$\mathcal{C}(F,\state_0)$. Indeed, as outlined in \mbox{\Cref{prop:discreteRealization}}, all components of another initial value $\tilde{\state}_0$ that are in the orthogonal complement of
$\mathcal{C}(F,\inival)$ are set to zero in the first \DMD iteration.

For our numerical experiments we set $N\vcentcolon=5$, $\state_0\vcentcolon= [1,2,\ldots,10]^\top$, and consider the time-grid $t_i \vcentcolon= i\timeStep$ for $i=0,1,\ldots,100$ with uniform stepsize $\timeStep=0.1$.  A \SVD of exactly sampled data
\begin{equation}
  \begin{bmatrix} 
    U_1 & U_2
  \end{bmatrix}
  \begin{bmatrix}
    \Sigma_1 & 0 \\ 0 & 0
  \end{bmatrix}
  V^T =
  \begin{bmatrix}
    \inival & \state(\timeStep;\state_0) & \state(2\timeStep;\state_0) & \cdots &  \state(10;\state_0) 
  \end{bmatrix}
\end{equation}
of the matrix of snapshots of the solution $\state(t;\state_0)$, reveals that the solution space
is indeed of dimension $N=5$ and defines the bases $U_1 \in \R^{10,5}$
and $U_2\in\R^{10,5}$ of $\mathcal{C}(F,\state_0)$ and its orthogonal
complement, respectively.

For our numerical experiment, depicted in \Cref{fig:inivcspace}, we choose the initial values
\begin{equation*}
	\tilde{\state}_0 \vcentcolon= U_1e\in \spann(U_1)\qquad\text{and}\qquad 
	\widehat{\state}_0 \vcentcolon= U_2e\in \spann(U_2) = \spann(U_1)^\perp,
\end{equation*}
with $e=[1,1,1,1,1]^\top$.
The exact solution for both initial values is presented in \Cref{fig:exactSolutionInReachableSpace,fig:exactSolutionOutsideReachableSpace}, respectively.  Our simulations confirm the following:
\begin{itemize}
	\item As predicted by \Cref{thm:exactDMD}, the \DMD approximation for the initial value $\tilde{\state}_0$, depicted in \Cref{fig:DMDInReachableSpace}, exactly recovers the exact solution, while the \DMD approximation for the initial value $\widehat{\state}_0$ (cf.~\Cref{fig:DMDOutsideReachableSpace}) is identically zero.  
	\item If we first transform the data with the matrix
\begin{equation*}
	T = \begin{bmatrix}
		1 & 1\\
		& \ddots & \ddots\\
		& & \ddots & 1\\
		& & \phantom{\ddots} & 1
	\end{bmatrix}\in \GL{2N}{\R},
\end{equation*}
then compute the \DMD approximation, and then transform the results back, the \DMD approximation for $\tilde{\state}_0$ remains unchanged, see \Cref{fig:TDMDInReachableSpace}, confirming~\eqref{eq:DMDtransformation} from~\Cref{thm:invariance}. In contrast, the prediction of the dynamics for $\widehat{\state}_0$ changes (see~\Cref{fig:TDMDOutsideReachableSpace}), highlighting that \DMD is not invariant under state-space transformations in the orthogonal complement of the data.
\end{itemize}

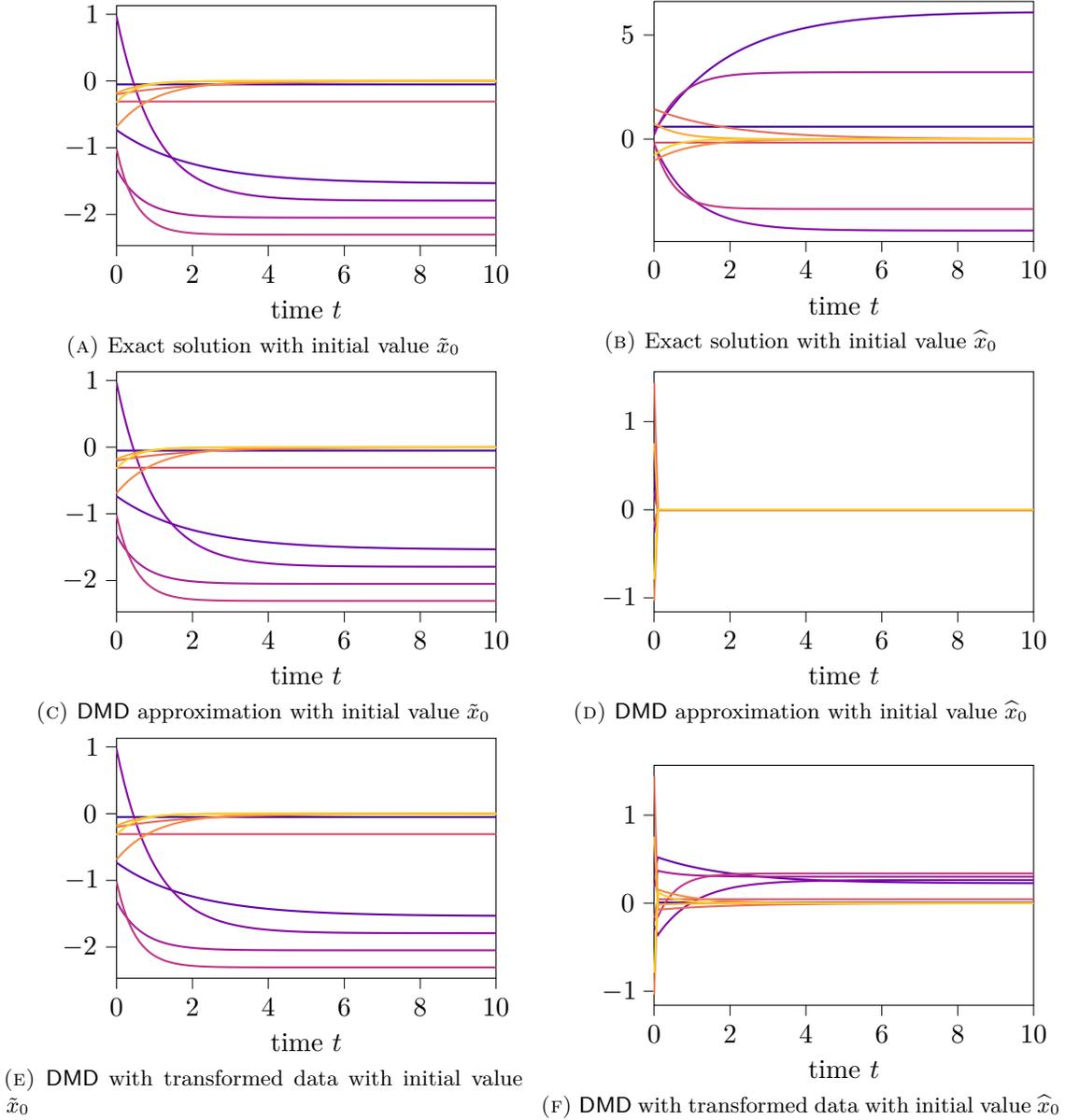
\begin{figure}[ht]
  \centering
  \newlength\figureheight
  \setlength\figureheight{5cm}
  \pgfplotsset{yticklabel style={text width=3em, align=right}}

  \begin{subfigure}{0.49\linewidth}
    \input{txz-exact.tex}\\[-1.5em]
     \caption{Exact solution with initial value $\tilde{\state}_0$}
    \label{fig:exactSolutionInReachableSpace}
  \end{subfigure}\hfill
  \begin{subfigure}{0.49\linewidth}
    \input{hxz-exact.tex}\\[-1.5em]
    \caption{Exact solution with initial value $\widehat{\state}_0$}
    \label{fig:exactSolutionOutsideReachableSpace}
  \end{subfigure}\\

  \begin{subfigure}[b]{0.49\linewidth}
    \input{txz-dmd.tex}\\[-1.5em]
    \caption{\DMD approximation with initial value $\tilde{\state}_0$}
    \label{fig:DMDInReachableSpace}
  \end{subfigure}\hfill
  \begin{subfigure}[b]{0.49\linewidth}
    \input{hxz-dmd.tex}\\[-1.5em]
    \caption{\DMD approximation with initial value $\widehat{\state}_0$}
    \label{fig:DMDOutsideReachableSpace}
  \end{subfigure}

  \begin{subfigure}[b]{0.49\linewidth}
    \input{txz-tdmd.tex}\\[-1.5em]
    \label{fig:TDMDInReachableSpace}
    \caption{\DMD with transformed data with initial value $\tilde{\state}_0$}
  \end{subfigure}\hfill
  \begin{subfigure}[b]{0.49\linewidth}
    \input{hxz-tdmd.tex}\\[-1.5em]
    \label{fig:TDMDOutsideReachableSpace}
    \caption{\DMD with transformed data with initial value $\widehat{\state}_0$}
  \end{subfigure}

	 \caption{Comparison of the exact solution, \DMD approximation, and \DMD approximation based on transformed data, for initial values inside, the reachable subspace, i.e., $\tilde{\state}_0\in\mathcal{C}(F,\state_0)$, and
   outside the reachable subspace, i.e., $\widehat{\state}_0\in\mathcal{C}(F,\state_0)^\perp$.}
	\label{fig:inivcspace}
\end{figure}

\subsection*{Acknowledgments} We thank Dr.~Robert Altmann for inviting us to the Sion workshop, where we started this work.
 B. Unger acknowledges funding from the DFG under Germany's Excellence Strategy -- EXC 2075 -- 390740016 and is thankful for support by the Stuttgart Center for Simulation Science (SimTech).

\bibliographystyle{siam}
\bibliography{literature}

\end{document}

%% file: txz-exact.tex
\begin{tikzpicture}

\definecolor{color0}{rgb}{0.254627,0.013882,0.615419}
\definecolor{color1}{rgb}{0.399411,0.000859,0.656133}
\definecolor{color2}{rgb}{0.534952,0.031217,0.650165}
\definecolor{color3}{rgb}{0.650746,0.125309,0.595617}
\definecolor{color4}{rgb}{0.752312,0.227133,0.513149}
\definecolor{color5}{rgb}{0.836801,0.329105,0.430905}
\definecolor{color6}{rgb}{0.907365,0.434524,0.35297}
\definecolor{color7}{rgb}{0.959424,0.543431,0.278701}
\definecolor{color8}{rgb}{0.990681,0.669558,0.201642}
\definecolor{color9}{rgb}{0.988648,0.809579,0.145357}

\begin{axis}[
height=\figureheight,
tick align=outside,
tick pos=left,
width=.95\linewidth,
x grid style={white!69.0196078431373!black},
xlabel={time \(\displaystyle t\)},
xmin=0, xmax=10,
xtick style={color=black},
y grid style={white!69.0196078431373!black},
ymin=-2.46869853866662, ymax=1.13059137490514,
ytick style={color=black}
]
\addplot [thick, color0]
table {%
0 -0.0512243212074243
0.1 -0.0512243212074243
0.2 -0.0512243212074243
0.3 -0.0512243212074243
0.4 -0.0512243212074243
0.5 -0.0512243212074243
0.6 -0.0512243212074243
0.7 -0.0512243212074243
0.8 -0.0512243212074243
0.9 -0.0512243212074243
1 -0.0512243212074243
1.1 -0.0512243212074243
1.2 -0.0512243212074243
1.3 -0.0512243212074243
1.4 -0.0512243212074243
1.5 -0.0512243212074243
1.6 -0.0512243212074243
1.7 -0.0512243212074243
1.8 -0.0512243212074243
1.9 -0.0512243212074243
2 -0.0512243212074243
2.1 -0.0512243212074243
2.2 -0.0512243212074243
2.3 -0.0512243212074243
2.4 -0.0512243212074243
2.5 -0.0512243212074243
2.6 -0.0512243212074243
2.7 -0.0512243212074243
2.8 -0.0512243212074243
2.9 -0.0512243212074243
3 -0.0512243212074243
3.1 -0.0512243212074243
3.2 -0.0512243212074243
3.3 -0.0512243212074243
3.4 -0.0512243212074243
3.5 -0.0512243212074243
3.6 -0.0512243212074243
3.7 -0.0512243212074243
3.8 -0.0512243212074243
3.9 -0.0512243212074243
4 -0.0512243212074243
4.1 -0.0512243212074243
4.2 -0.0512243212074243
4.3 -0.0512243212074243
4.4 -0.0512243212074243
4.5 -0.0512243212074243
4.6 -0.0512243212074243
4.7 -0.0512243212074243
4.8 -0.0512243212074243
4.9 -0.0512243212074243
5 -0.0512243212074243
5.1 -0.0512243212074243
5.2 -0.0512243212074243
5.3 -0.0512243212074243
5.4 -0.0512243212074243
5.5 -0.0512243212074243
5.6 -0.0512243212074243
5.7 -0.0512243212074243
5.8 -0.0512243212074243
5.9 -0.0512243212074243
6 -0.0512243212074243
6.1 -0.0512243212074243
6.2 -0.0512243212074243
6.3 -0.0512243212074243
6.4 -0.0512243212074243
6.5 -0.0512243212074243
6.6 -0.0512243212074243
6.7 -0.0512243212074243
6.8 -0.0512243212074243
6.9 -0.0512243212074243
7 -0.0512243212074243
7.1 -0.0512243212074243
7.2 -0.0512243212074243
7.3 -0.0512243212074243
7.4 -0.0512243212074243
7.5 -0.0512243212074243
7.6 -0.0512243212074243
7.7 -0.0512243212074243
7.8 -0.0512243212074243
7.9 -0.0512243212074243
8 -0.0512243212074243
8.1 -0.0512243212074243
8.2 -0.0512243212074243
8.3 -0.0512243212074243
8.4 -0.0512243212074243
8.5 -0.0512243212074243
8.6 -0.0512243212074243
8.7 -0.0512243212074243
8.8 -0.0512243212074243
8.9 -0.0512243212074243
9 -0.0512243212074243
9.1 -0.0512243212074243
9.2 -0.0512243212074243
9.3 -0.0512243212074243
9.4 -0.0512243212074243
9.5 -0.0512243212074243
9.6 -0.0512243212074243
9.7 -0.0512243212074243
9.8 -0.0512243212074243
9.9 -0.0512243212074243
10 -0.0512243212074243
};
\addplot [thick, color1]
table {%
0 -0.734977255192507
0.1 -0.774079180223295
0.2 -0.811274081867201
0.3 -0.846654966752294
0.4 -0.880310305519867
0.5 -0.912324254047122
0.6 -0.942776863880699
0.7 -0.971744282407236
0.8 -0.999298943261506
0.9 -1.02550974744823
1 -1.05044223563046
1.1 -1.07415875201542
1.2 -1.09671860024744
1.3 -1.11817819169801
1.4 -1.13859118652356
1.5 -1.1580086278438
1.6 -1.17647906937614
1.7 -1.19404869684521
1.8 -1.21076144347131
1.9 -1.22665909982628
2 -1.24178141833172
2.1 -1.25616621266078
2.2 -1.26984945229196
2.3 -1.28286535245164
2.4 -1.29524645966989
2.5 -1.30702373316379
2.6 -1.31822662225158
2.7 -1.3288831399913
2.8 -1.33901993322804
2.9 -1.3486623492249
3 -1.3578344990444
3.1 -1.36655931783863
3.2 -1.37485862219914
3.3 -1.38275316470974
3.4 -1.3902626858388
3.5 -1.39740596330067
3.6 -1.40420085900977
3.7 -1.41066436374469
3.8 -1.41681263963394
3.9 -1.42266106056974
4 -1.42822425065074
4.1 -1.43351612074988
4.2 -1.43854990329881
4.3 -1.4433381853759
4.4 -1.44789294018043
4.5 -1.45222555697189
4.6 -1.45634686954901
4.7 -1.46026718333994
4.8 -1.46399630117114
4.9 -1.46754354777961
5 -1.47091779312954
5.1 -1.47412747459189
5.2 -1.47718061804214
5.3 -1.48008485792925
5.4 -1.48284745636567
5.5 -1.48547532128648
5.6 -1.48797502372276
5.7 -1.49035281423265
5.8 -1.49261463853095
5.9 -1.49476615235655
6 -1.49681273561468
6.1 -1.4987595058295
6.2 -1.50061133094058
6.3 -1.50237284147527
6.4 -1.50404844212744
6.5 -1.50564232277149
6.6 -1.50715846893925
6.7 -1.50860067178588
6.8 -1.50997253756968
6.9 -1.5112774966697
7 -1.51251881216341
7.1 -1.51369958798612
7.2 -1.51482277669241
7.3 -1.51589118683911
7.4 -1.51690749000808
7.5 -1.51787422748662
7.6 -1.51879381662198
7.7 -1.51966855686598
7.8 -1.52050063552486
7.9 -1.5212921332287
8 -1.52204502913401
8.1 -1.52276120587273
8.2 -1.52344245425974
8.3 -1.52409047777086
8.4 -1.5247068968024
8.5 -1.52529325272303
8.6 -1.52585101172796
8.7 -1.52638156850523
8.8 -1.52688624972314
8.9 -1.52736631734761
9 -1.52782297179775
9.1 -1.52825735494756
9.2 -1.52867055298116
9.3 -1.52906359910887
9.4 -1.52943747615073
9.5 -1.5297931189941
9.6 -1.53013141693132
9.7 -1.53045321588345
9.8 -1.53075932051549
9.9 -1.53105049624847
10 -1.53132747117338
};
\addplot [thick, color2]
table {%
0 0.966987287924606
0.1 0.704353927588371
0.2 0.466713435931625
0.3 0.251687427040139
0.4 0.0571238483441892
0.5 -0.11892455784689
0.6 -0.278219743154173
0.7 -0.422355987333174
0.8 -0.552775854361501
0.9 -0.670784630104006
1 -0.777563386052439
1.1 -0.874180799885912
1.2 -0.961603851156302
1.3 -1.04070749914463
1.4 -1.11228343974761
1.5 -1.17704802903631
1.6 -1.23564945278845
1.7 -1.28867421374957
1.8 -1.33665300154961
1.9 -1.38006600402309
2 -1.41934771309038
2.1 -1.45489127329887
2.2 -1.48705241654572
2.3 -1.51615302236229
2.4 -1.54248433939263
2.5 -1.56630990030786
2.6 -1.58786815932965
2.7 -1.60737487876027
2.8 -1.62502528840423
2.9 -1.64099603949375
3 -1.65544697267369
3.1 -1.66852271774043
3.2 -1.68035414114551
3.3 -1.69105965575106
3.4 -1.7007464059455
3.5 -1.70951133998058
3.6 -1.71744218026215
3.7 -1.72461830130538
3.8 -1.73111152414164
3.9 -1.73698683512754
4 -1.74230303635018
4.1 -1.74711333413824
4.2 -1.75146587156876
4.3 -1.75540421029931
4.4 -1.7589677665476
4.5 -1.76219220558233
4.6 -1.76510979867314
4.7 -1.7677497460723
4.8 -1.7701384692607
4.9 -1.77229987538291
5 -1.77425559651785
5.1 -1.77602520617998
5.2 -1.7776264152176
5.3 -1.77907524906894
5.4 -1.78038620815014
5.5 -1.78157241298033
5.6 -1.78264573549614
5.7 -1.78361691787007
5.8 -1.78449568002173
5.9 -1.78529081689811
6 -1.78601028649632
6.1 -1.78666128950992
6.2 -1.78725034139588
6.3 -1.78778333758346
6.4 -1.78826561247765
6.5 -1.78870199284769
6.6 -1.78909684613501
6.7 -1.789454124164
6.8 -1.78977740269328
6.9 -1.79006991720302
7 -1.79033459527675
7.1 -1.79057408590159
7.2 -1.79079078598022
7.3 -1.79098686431985
7.4 -1.79116428333842
7.5 -1.79132481870509
7.6 -1.79147007711177
7.7 -1.79160151235342
7.8 -1.79172043987811
7.9 -1.79182804995249
8 -1.79192541957434
8.1 -1.79201352325158
8.2 -1.7920932427554
8.3 -1.79216537594541
8.4 -1.79223064475482
8.5 -1.7922897024158
8.6 -1.79234313999727
8.7 -1.79239149232052
8.8 -1.79243524331184
8.9 -1.79247483084587
9 -1.79251065112794
9.1 -1.79254306265949
9.2 -1.79257238982601
9.3 -1.79259892614364
9.4 -1.79262293719677
9.5 -1.79264466329609
9.6 -1.7926643218837
9.7 -1.79268210970935
9.8 -1.79269820479959
9.9 -1.79271276823949
10 -1.79272594578484
};
\addplot [thick, color3]
table {%
0 -1.31689588267591
0.1 -1.41886836462996
0.2 -1.50663689322366
0.3 -1.58217996586336
0.4 -1.647200491048
0.5 -1.70316417570577
0.6 -1.75133256548085
0.7 -1.79279148277181
0.8 -1.82847550357809
0.9 -1.85918902491697
1 -1.88562439771744
1.1 -1.90837753394658
1.2 -1.92796133978769
1.3 -1.94481727768389
1.4 -1.95932531788127
1.5 -1.97181250380146
1.6 -1.98256032432606
1.7 -1.99181105918078
1.8 -1.99977324045802
1.9 -2.00662635339309
2 -2.01252488235964
2.1 -2.01760179329033
2.2 -2.02197153102397
2.3 -2.0257325991462
2.4 -2.02896978047888
2.5 -2.03175604827306
2.6 -2.03415421118796
2.7 -2.03621832913758
2.8 -2.0379949319211
2.9 -2.03952406810782
3 -2.04084020782077
3.1 -2.04197301976979
3.2 -2.04294804005011
3.3 -2.04378724778255
3.4 -2.04450956057174
3.5 -2.04513126095087
3.6 -2.04566636342613
3.7 -2.0461269303948
3.8 -2.04652334405841
3.9 -2.04686454046064
4 -2.04715821092556
4.1 -2.04741097543717
4.2 -2.04762853186846
4.3 -2.0478157844242
4.4 -2.04797695419253
4.5 -2.04811567429769
4.6 -2.04823507179869
4.7 -2.04833783818017
4.8 -2.04842629002441
4.9 -2.04850242123229
5 -2.04856794797016
5.1 -2.04862434735611
5.2 -2.04867289075747
5.3 -2.04871467245021
5.4 -2.04875063428643
5.5 -2.04878158692571
5.6 -2.04880822810923
5.7 -2.04883115838839
5.8 -2.04885089466256
5.9 -2.04886788183117
6 -2.04888250282268
6.1 -2.0488950872267
6.2 -2.04890591872362
6.3 -2.04891524147942
6.4 -2.04892326564969
6.5 -2.04893017211705
6.6 -2.0489361165686
6.7 -2.04894123300546
6.8 -2.04894563676347
6.9 -2.04894942711312
7 -2.0489526894973
7.1 -2.04895549745739
7.2 -2.04895791429103
7.3 -2.04895999447902
7.4 -2.04896178491342
7.5 -2.04896332595459
7.6 -2.04896465234101
7.7 -2.04896579397239
7.8 -2.04896677658362
7.9 -2.04896762232495
8 -2.04896835026125
8.1 -2.04896897680183
8.2 -2.04896951607031
8.3 -2.04896998022299
8.4 -2.04897037972291
8.5 -2.04897072357567
8.6 -2.04897101953248
8.7 -2.04897127426488
8.8 -2.04897149351508
8.9 -2.04897168222547
9 -2.04897184465002
9.1 -2.04897198445012
9.2 -2.04897210477718
9.3 -2.04897220834364
9.4 -2.04897229748412
9.5 -2.04897237420804
9.6 -2.04897244024494
9.7 -2.04897249708341
9.8 -2.04897254600475
9.9 -2.04897258811173
10 -2.04897262435354
};
\addplot [thick, color4]
table {%
0 -1.02005922093539
0.1 -1.25299658996195
0.2 -1.44370957752507
0.3 -1.59985216545437
0.4 -1.72769090405728
0.5 -1.83235641078617
0.6 -1.91804927993161
0.7 -1.98820866722046
0.8 -2.04565031521096
0.9 -2.09267955892826
1 -2.13118384705361
1.1 -2.16270849186721
1.2 -2.18851868805597
1.3 -2.20965028941868
1.4 -2.22695138131611
1.5 -2.24111631731437
1.6 -2.25271358603153
1.7 -2.26220862658197
1.8 -2.26998250828234
1.9 -2.27634722430123
2 -2.28155821304049
2.1 -2.28582460977528
2.2 -2.28931763998687
2.3 -2.29217749124254
2.4 -2.29451893941478
2.5 -2.29643595504014
2.6 -2.29800547468674
2.7 -2.29929048868898
2.8 -2.30034256917075
2.9 -2.30120393981588
3 -2.30190917045286
3.1 -2.30248656446336
3.2 -2.3029592946964
3.3 -2.3033463334761
3.4 -2.30366321402767
3.5 -2.3039226538803
3.6 -2.30413506526622
3.7 -2.30430897300017
3.8 -2.30445135661016
3.9 -2.30456793045039
4 -2.30466337303839
4.1 -2.30474151482034
4.2 -2.30480549190033
4.3 -2.3048578719032
4.4 -2.3049007570224
4.5 -2.30493586838834
4.6 -2.30496461514342
4.7 -2.30498815099585
4.8 -2.30500742052203
4.9 -2.30502319707572
5 -2.3050361138254
5.1 -2.30504668916559
5.2 -2.30505534752183
5.3 -2.30506243638436
5.4 -2.30506824025411
5.5 -2.30507299206076
5.6 -2.305076882511
5.7 -2.30508006774226
5.8 -2.30508267558904
5.9 -2.30508481071341
6 -2.30508655880538
6.1 -2.30508799002204
6.2 -2.30508916180314
6.3 -2.30509012117635
6.4 -2.30509090664471
6.5 -2.30509154973181
6.6 -2.305092076247
6.7 -2.30509250732117
6.8 -2.30509286025485
6.9 -2.30509314921251
7 -2.30509338579104
7.1 -2.30509357948515
7.2 -2.30509373806847
7.3 -2.30509386790552
7.4 -2.3050939742071
7.5 -2.30509406123948
7.6 -2.30509413249556
7.7 -2.3050941908351
7.8 -2.30509423859948
7.9 -2.30509427770565
8 -2.30509430972307
8.1 -2.30509433593672
8.2 -2.30509435739864
8.3 -2.30509437497017
8.4 -2.30509438935653
8.5 -2.30509440113508
8.6 -2.30509441077854
8.7 -2.30509441867394
8.8 -2.30509442513815
8.9 -2.30509443043059
9 -2.30509443476368
9.1 -2.30509443831131
9.2 -2.30509444121586
9.3 -2.30509444359391
9.4 -2.30509444554089
9.5 -2.30509444713495
9.6 -2.30509444844005
9.7 -2.30509444950857
9.8 -2.30509445038341
9.9 -2.30509445109966
10 -2.30509445168608
};
\addplot [thick, color5]
table {%
0 -0.30734592724463
0.1 -0.30734592724463
0.2 -0.30734592724463
0.3 -0.30734592724463
0.4 -0.30734592724463
0.5 -0.30734592724463
0.6 -0.30734592724463
0.7 -0.30734592724463
0.8 -0.30734592724463
0.9 -0.30734592724463
1 -0.30734592724463
1.1 -0.30734592724463
1.2 -0.30734592724463
1.3 -0.30734592724463
1.4 -0.30734592724463
1.5 -0.30734592724463
1.6 -0.30734592724463
1.7 -0.30734592724463
1.8 -0.30734592724463
1.9 -0.30734592724463
2 -0.30734592724463
2.1 -0.30734592724463
2.2 -0.30734592724463
2.3 -0.30734592724463
2.4 -0.30734592724463
2.5 -0.30734592724463
2.6 -0.30734592724463
2.7 -0.30734592724463
2.8 -0.30734592724463
2.9 -0.30734592724463
3 -0.30734592724463
3.1 -0.30734592724463
3.2 -0.30734592724463
3.3 -0.30734592724463
3.4 -0.30734592724463
3.5 -0.30734592724463
3.6 -0.30734592724463
3.7 -0.30734592724463
3.8 -0.30734592724463
3.9 -0.30734592724463
4 -0.30734592724463
4.1 -0.30734592724463
4.2 -0.30734592724463
4.3 -0.30734592724463
4.4 -0.30734592724463
4.5 -0.30734592724463
4.6 -0.30734592724463
4.7 -0.30734592724463
4.8 -0.30734592724463
4.9 -0.30734592724463
5 -0.30734592724463
5.1 -0.30734592724463
5.2 -0.30734592724463
5.3 -0.30734592724463
5.4 -0.30734592724463
5.5 -0.30734592724463
5.6 -0.30734592724463
5.7 -0.30734592724463
5.8 -0.30734592724463
5.9 -0.30734592724463
6 -0.30734592724463
6.1 -0.30734592724463
6.2 -0.30734592724463
6.3 -0.30734592724463
6.4 -0.30734592724463
6.5 -0.30734592724463
6.6 -0.30734592724463
6.7 -0.30734592724463
6.8 -0.30734592724463
6.9 -0.30734592724463
7 -0.30734592724463
7.1 -0.30734592724463
7.2 -0.30734592724463
7.3 -0.30734592724463
7.4 -0.30734592724463
7.5 -0.30734592724463
7.6 -0.30734592724463
7.7 -0.30734592724463
7.8 -0.30734592724463
7.9 -0.30734592724463
8 -0.30734592724463
8.1 -0.30734592724463
8.2 -0.30734592724463
8.3 -0.30734592724463
8.4 -0.30734592724463
8.5 -0.30734592724463
8.6 -0.30734592724463
8.7 -0.30734592724463
8.8 -0.30734592724463
8.9 -0.30734592724463
9 -0.30734592724463
9.1 -0.30734592724463
9.2 -0.30734592724463
9.3 -0.30734592724463
9.4 -0.30734592724463
9.5 -0.30734592724463
9.6 -0.30734592724463
9.7 -0.30734592724463
9.8 -0.30734592724463
9.9 -0.30734592724463
10 -0.30734592724463
};
\addplot [thick, color6]
table {%
0 -0.200438095257661
0.1 -0.190662613999964
0.2 -0.181363888588987
0.3 -0.172518667367714
0.4 -0.164104832675821
0.5 -0.156101345544007
0.6 -0.148488193085613
0.7 -0.141246338453978
0.8 -0.134357673240411
0.9 -0.127804972193731
1 -0.121571850148173
1.1 -0.115642721051933
1.2 -0.110002758993927
1.3 -0.104637861131284
1.4 -0.0995346124248969
1.5 -0.0946802520948363
1.6 -0.0900626417117537
1.7 -0.0856702348444854
1.8 -0.0814920481879609
1.9 -0.0775176340992185
2 -0.0737370544728565
2.1 -0.0701408558905931
2.2 -0.0667200459827964
2.3 -0.0634660709428766
2.4 -0.060370794138314
2.5 -0.0574264757648395
2.6 -0.0546257534928925
2.7 -0.051961624057962
2.8 -0.0494274257487776
2.9 -0.0470168217495615
3 -0.0447237842946881
3.1 -0.0425425795961302
3.2 -0.0404677535060027
3.3 -0.0384941178783517
3.4 -0.0366167375960872
3.5 -0.0348309182306197
3.6 -0.0331321943033438
3.7 -0.0315163181196155
3.8 -0.0299792491473033
3.9 -0.0285171439133528
4 -0.0271263463931027
4.1 -0.0258033788683181
4.2 -0.0245449332310841
4.3 -0.0233478627118126
4.4 -0.0222091740106791
4.5 -0.0211260198128145
4.6 -0.0200956916685342
4.7 -0.0191156132208036
4.8 -0.0181833337630033
4.9 -0.017296522110886
5 -0.016452960773402
5.1 -0.015650540407816
5.2 -0.014887254545252
5.3 -0.0141611945734757
5.4 -0.0134705449643699
5.5 -0.0128135787341686
5.6 -0.0121886531250977
5.7 -0.0115942054976256
5.8 -0.0110287494230494
5.9 -0.0104908709666498
6 -0.00997922515211758
6.1 -0.00949253259841185
6.2 -0.00902957632064157
6.3 -0.00858919868696916
6.4 -0.00817029852392796
6.5 -0.00777182836291503
6.6 -0.00739279182097399
6.7 -0.00703224110931867
6.8 -0.00668927466336746
6.9 -0.00636303488836224
7 -0.00605270601493478
7.1 -0.00575751205925842
7.2 -0.00547671488268431
7.3 -0.00520961234601029
7.4 -0.00495553655376718
7.5 -0.00471385218413221
7.6 -0.00448395490029351
7.7 -0.00426526983929336
7.8 -0.00405725017457127
7.9 -0.00385937574861285
8 -0.00367115177228502
8.1 -0.00349210758760545
8.2 -0.00332179549085251
8.3 -0.0031597896130727
8.4 -0.00300568485518648
8.5 -0.00285909587502955
8.6 -0.00271965612379672
8.7 -0.002587016929479
8.8 -0.00246084662500191
8.9 -0.00234082971888509
9 -0.00222666610634923
9.1 -0.00211807031889783
9.2 -0.00201477081049722
9.3 -0.00191650927857011
9.4 -0.00182304001810453
9.5 -0.00173412930726334
9.6 -0.00164955482295793
9.7 -0.00156910508492465
9.8 -0.00149257892691402
9.9 -0.00141978499367031
10 -0.00135054126244376
};
\addplot [thick, color7]
table {%
0 -0.689959632546218
0.1 -0.624301292462159
0.2 -0.564891169547972
0.3 -0.511134667325101
0.4 -0.462493772651113
0.5 -0.418481671103343
0.6 -0.378657874776523
0.7 -0.342623813731773
0.8 -0.310018846974691
0.9 -0.280516653039064
1 -0.253821964051956
1.1 -0.229667610593588
1.2 -0.207811847775991
1.3 -0.188035935778909
1.4 -0.170141950628164
1.5 -0.153950803305989
1.6 -0.139300447367953
1.7 -0.126044257127673
1.8 -0.114049560177664
1.9 -0.103196309559294
2 -0.0933758822924716
2.1 -0.0844899922403497
2.2 -0.0764497064286362
2.3 -0.0691745549744943
2.4 -0.062591725716908
2.5 -0.056635335488102
2.6 -0.0512457707326546
2.7 -0.0463690908749979
2.8 -0.0419564884640079
2.9 -0.0379638006916284
3 -0.0343510673966448
3.1 -0.0310821311299593
3.2 -0.0281242752786875
3.3 -0.0254478966273002
3.4 -0.0230262090786923
3.5 -0.0208349755699201
3.6 -0.0188522654995288
3.7 -0.0170582352387221
3.8 -0.0154349295296553
3.9 -0.0139661017831803
4 -0.0126370514775202
4.1 -0.0114344770305069
4.2 -0.0103463426728754
4.3 -0.00936175799023983
4.4 -0.00847086892816611
4.5 -0.00766475916948287
4.6 -0.00693536089678232
4.7 -0.00627537404699207
4.8 -0.00567819324989017
4.9 -0.00513784171933984
5 -0.0046489114356049
5.1 -0.00420650902007058
5.2 -0.00380620676066564
5.3 -0.00344399829783171
5.4 -0.00311625852753028
5.5 -0.00281970731998304
5.6 -0.00255137669103055
5.7 -0.00230858109754922
5.8 -0.00208889055963305
5.9 -0.00189010634053806
6 -0.00171023894098586
6.1 -0.0015474881875862
6.2 -0.00140022521609664
6.3 -0.00126697616920173
6.4 -0.00114640744565358
6.5 -0.00103731235314239
6.6 -0.000938599031314163
6.7 -0.000849279524065361
6.8 -0.000768459891746109
6.9 -0.000695331264311743
7 -0.000629161745879517
7.1 -0.000569289089668618
7.2 -0.000515114070011794
7.3 -0.000466094485103466
7.4 -0.000421739730461821
7.5 -0.000381605888794256
7.6 -0.000345291287123911
7.7 -0.000312432476711513
7.8 -0.000282700595538226
7.9 -0.000255798076944036
8 -0.000231455671480606
8.1 -0.000209429752172291
8.2 -0.000189499876215486
8.3 -0.000171466578712954
8.4 -0.000155149376362089
8.5 -0.000140384961117362
8.6 -0.000127025565748513
8.7 -0.000114937484936441
8.8 -0.000103999737105436
8.9 -9.41028535989017e-05
9 -8.51477830802461e-05
9.1 -7.70449001938159e-05
9.2 -6.97131085642105e-05
9.3 -6.30790291565008e-05
9.4 -5.70762658741832e-05
9.5 -5.16447410447299e-05
9.6 -4.67300941420491e-05
9.7 -4.2283137728069e-05
9.8 -3.82593651683249e-05
9.9 -3.4618505194602e-05
10 -3.13241188565481e-05
};
\addplot [thick, color8]
table {%
0 -0.183019241405419
0.1 -0.157526120916907
0.2 -0.13558398876848
0.3 -0.116698220608556
0.4 -0.100443089312396
0.5 -0.0864521681479534
0.6 -0.0744100707041838
0.7 -0.0640453413814435
0.8 -0.0551243361798743
0.9 -0.0474459558451542
1 -0.0408371126450353
1.1 -0.0351488285877505
1.2 -0.0302528771274739
1.3 -0.026038892653424
1.4 -0.0224118826040779
1.5 -0.0192900861240318
1.6 -0.0166031309928805
1.7 -0.0142904472792023
1.8 -0.0122999019598912
1.9 -0.0105866237261246
2 -0.00911199148448619
2.1 -0.00784276375181447
2.2 -0.00675032931840403
2.3 -0.00581006228784627
2.4 -0.0050007669546757
2.5 -0.00430420000613222
2.6 -0.00370465927740679
2.7 -0.00318862979000111
2.8 -0.00274447909412051
2.9 -0.00236219504744134
3 -0.00203316011920453
3.1 -0.00174995713194866
3.2 -0.00150620206187013
3.3 -0.00129640012875949
3.4 -0.00111582193146177
3.5 -0.000960396836679156
3.6 -0.000826621217863143
3.7 -0.000711479475697002
3.8 -0.000612376059795128
3.9 -0.000527076959237415
4 -0.000453659343005508
4.1 -0.000390468215104592
4.2 -0.000336079107280977
4.3 -0.00028926596834655
4.4 -0.000248973526264194
4.5 -0.000214293499974265
4.6 -0.000184444124723893
4.7 -0.000158752529354593
4.8 -0.000136639568293151
4.9 -0.000117606766325192
5 -0.000101225081857651
5.1 -8.71252353691592e-05
5.2 -7.49893850301458e-05
5.3 -6.45439618426563e-05
5.4 -5.55535027880488e-05
5.5 -4.78153429680254e-05
5.6 -4.11550470880792e-05
5.7 -3.54224772988586e-05
5.8 -3.04884087558632e-05
5.9 -2.6241616604679e-05
6 -2.25863687259355e-05
6.1 -1.94402677208901e-05
6.2 -1.67323934912087e-05
6.3 -1.4401704542566e-05
6.4 -1.23956619739036e-05
6.5 -1.06690451340076e-05
6.6 -9.18293224767929e-06
6.7 -7.90382303254845e-06
6.8 -6.80288352836654e-06
6.9 -5.85529611555572e-06
7 -5.03970007098947e-06
7.1 -4.33771004989057e-06
7.2 -3.73350163935994e-06
7.3 -3.21345464099313e-06
7.4 -2.76584604138291e-06
7.5 -2.38058574938194e-06
7.6 -2.04898914305686e-06
7.7 -1.76358129903738e-06
7.8 -1.51792849115554e-06
7.9 -1.30649315998043e-06
8 -1.12450908393993e-06
8.1 -9.67873938109537e-07
8.2 -8.33056818704809e-07
8.3 -7.17018648674513e-07
8.4 -6.17143670159669e-07
8.5 -5.31180479506663e-07
8.6 -4.57191275632672e-07
8.7 -3.93508177688987e-07
8.8 -3.386956273254e-07
8.9 -2.91518028019261e-07
9 -2.50911891987881e-07
9.1 -2.15961866813872e-07
9.2 -1.85880101370346e-07
9.3 -1.59988485908155e-07
9.4 -1.37703365957317e-07
9.5 -1.18522385460041e-07
9.6 -1.02013162550383e-07
9.7 -8.78035427074605e-08
9.8 -7.55732095666895e-08
9.9 -6.50464642780921e-08
10 -5.59860106424015e-08
};
\addplot [thick, color9]
table {%
0 -0.321258808349837
0.1 -0.263024466093197
0.2 -0.215346219202417
0.3 -0.176310572220091
0.4 -0.144350887569365
0.5 -0.118184510887142
0.6 -0.0967612936007824
0.7 -0.0792214467785683
0.8 -0.0648610347809444
0.9 -0.0531037238516198
1 -0.0434776518202819
1.1 -0.0355964906168817
1.2 -0.0291439415696928
1.3 -0.0238610412290153
1.4 -0.0195357682546565
1.5 -0.0159945342550918
1.6 -0.0130952170758029
1.7 -0.0107214569381918
1.8 -0.0087779865130989
1.9 -0.00718680750837783
2 -0.00588406032356068
2.1 -0.00481746113986511
2.2 -0.00394420358696567
2.3 -0.00322924077304928
2.4 -0.00264387872998876
2.5 -0.00216462482365056
2.6 -0.00177224491199872
2.7 -0.00145099141143933
2.8 -0.00118797129099741
2.9 -0.000972628629713328
3 -0.000796320970470399
3.1 -0.000651972467845019
3.2 -0.000533789909584863
3.3 -0.000437030214659843
3.4 -0.000357810076766285
3.5 -0.000292950113609751
3.6 -0.000239847267129992
3.7 -0.000196370333641034
3.8 -0.000160774431144098
3.9 -0.000131630971086292
4 -0.000107770324085866
4.1 -8.8234878598279e-05
4.2 -7.22406086025133e-05
4.3 -5.91456078839475e-05
4.4 -4.84243280840793e-05
4.5 -3.96464865995735e-05
4.6 -3.24597978305649e-05
4.7 -2.65758347225775e-05
4.8 -2.17584531760918e-05
4.9 -1.78143147546737e-05
5 -1.45851273346622e-05
5.1 -1.19412922864462e-05
5.2 -9.77670322640642e-06
5.3 -8.00448759517565e-06
5.4 -6.55352015680153e-06
5.5 -5.36556849328985e-06
5.6 -4.39295593320269e-06
5.7 -3.59664811942943e-06
5.8 -2.94468642337696e-06
5.9 -2.41090533298993e-06
6 -1.97388233887857e-06
6.1 -1.61607817379738e-06
6.2 -1.32313290026602e-06
6.3 -1.08328959585705e-06
6.4 -8.86922506617585e-07
6.5 -7.26150731764828e-07
6.6 -5.94521935465944e-07
6.7 -4.86753391945412e-07
6.8 -3.98519971150729e-07
6.9 -3.26280556096852e-07
7 -2.67135925407879e-07
7.1 -2.18712397383376e-07
7.2 -1.79066565817182e-07
7.3 -1.4660730428259e-07
7.4 -1.20031908642017e-07
7.5 -9.82738149558665e-08
7.6 -8.04597945266627e-08
7.7 -6.58749081653143e-08
7.8 -5.39338131711306e-08
7.9 -4.4157271473967e-08
8 -3.61529161277499e-08
8.1 -2.95995042472378e-08
8.2 -2.42340244030759e-08
8.3 -1.98411410496405e-08
8.4 -1.62445523534986e-08
8.5 -1.32999145817947e-08
8.6 -1.08890490814256e-08
8.7 -8.91519935473865e-09
8.8 -7.29914788154552e-09
8.9 -5.97603684188532e-09
9 -4.89276514397854e-09
9.1 -4.00585729096325e-09
9.2 -3.27971855655326e-09
9.3 -2.68520644369068e-09
9.4 -2.19846109381272e-09
9.5 -1.79994770694994e-09
9.6 -1.4736725416121e-09
9.7 -1.20654102978442e-09
9.8 -9.87832245934884e-10
9.9 -8.08768638628982e-10
10 -6.62163756570561e-10
};
\end{axis}

\end{tikzpicture}

%% file: hxz-exact.tex
\begin{tikzpicture}

\definecolor{color0}{rgb}{0.254627,0.013882,0.615419}
\definecolor{color1}{rgb}{0.399411,0.000859,0.656133}
\definecolor{color2}{rgb}{0.534952,0.031217,0.650165}
\definecolor{color3}{rgb}{0.650746,0.125309,0.595617}
\definecolor{color4}{rgb}{0.752312,0.227133,0.513149}
\definecolor{color5}{rgb}{0.836801,0.329105,0.430905}
\definecolor{color6}{rgb}{0.907365,0.434524,0.35297}
\definecolor{color7}{rgb}{0.959424,0.543431,0.278701}
\definecolor{color8}{rgb}{0.990681,0.669558,0.201642}
\definecolor{color9}{rgb}{0.988648,0.809579,0.145357}

\begin{axis}[
height=\figureheight,
tick align=outside,
tick pos=left,
width=.95\linewidth,
x grid style={white!69.0196078431373!black},
xlabel={time \(\displaystyle t\)},
xmin=0, xmax=10,
xtick style={color=black},
y grid style={white!69.0196078431373!black},
ymin=-4.92225568214438, ymax=6.61235300189442,
ytick style={color=black}
]
\addplot [thick, color0]
table {%
0 0.594494225663158
0.1 0.594494225663158
0.2 0.594494225663158
0.3 0.594494225663158
0.4 0.594494225663158
0.5 0.594494225663158
0.6 0.594494225663158
0.7 0.594494225663158
0.8 0.594494225663158
0.9 0.594494225663158
1 0.594494225663158
1.1 0.594494225663158
1.2 0.594494225663158
1.3 0.594494225663158
1.4 0.594494225663158
1.5 0.594494225663158
1.6 0.594494225663158
1.7 0.594494225663158
1.8 0.594494225663158
1.9 0.594494225663158
2 0.594494225663158
2.1 0.594494225663158
2.2 0.594494225663158
2.3 0.594494225663158
2.4 0.594494225663158
2.5 0.594494225663158
2.6 0.594494225663158
2.7 0.594494225663158
2.8 0.594494225663158
2.9 0.594494225663158
3 0.594494225663158
3.1 0.594494225663158
3.2 0.594494225663158
3.3 0.594494225663158
3.4 0.594494225663158
3.5 0.594494225663158
3.6 0.594494225663158
3.7 0.594494225663158
3.8 0.594494225663158
3.9 0.594494225663158
4 0.594494225663158
4.1 0.594494225663158
4.2 0.594494225663158
4.3 0.594494225663158
4.4 0.594494225663158
4.5 0.594494225663158
4.6 0.594494225663158
4.7 0.594494225663158
4.8 0.594494225663158
4.9 0.594494225663158
5 0.594494225663158
5.1 0.594494225663158
5.2 0.594494225663158
5.3 0.594494225663158
5.4 0.594494225663158
5.5 0.594494225663158
5.6 0.594494225663158
5.7 0.594494225663158
5.8 0.594494225663158
5.9 0.594494225663158
6 0.594494225663158
6.1 0.594494225663158
6.2 0.594494225663158
6.3 0.594494225663158
6.4 0.594494225663158
6.5 0.594494225663158
6.6 0.594494225663158
6.7 0.594494225663158
6.8 0.594494225663158
6.9 0.594494225663158
7 0.594494225663158
7.1 0.594494225663158
7.2 0.594494225663158
7.3 0.594494225663158
7.4 0.594494225663158
7.5 0.594494225663158
7.6 0.594494225663158
7.7 0.594494225663158
7.8 0.594494225663158
7.9 0.594494225663158
8 0.594494225663158
8.1 0.594494225663158
8.2 0.594494225663158
8.3 0.594494225663158
8.4 0.594494225663158
8.5 0.594494225663158
8.6 0.594494225663158
8.7 0.594494225663158
8.8 0.594494225663158
8.9 0.594494225663158
9 0.594494225663158
9.1 0.594494225663158
9.2 0.594494225663158
9.3 0.594494225663158
9.4 0.594494225663158
9.5 0.594494225663158
9.6 0.594494225663158
9.7 0.594494225663158
9.8 0.594494225663158
9.9 0.594494225663158
10 0.594494225663158
};
\addplot [thick, color1]
table {%
0 0.360406293128509
0.1 0.641641850399602
0.2 0.909161387691722
0.3 1.1636338431928
0.4 1.40569553059038
0.5 1.63595173018725
0.6 1.8549782024175
0.7 2.0633226275475
0.8 2.26150597516184
0.9 2.45002380685866
1 2.62934751541175
1.1 2.79992550349804
1.2 2.96218430493784
1.3 3.11652965125161
1.4 3.26334748620001
1.5 3.40300493084442
1.6 3.53585120154077
1.7 3.66221848316232
1.8 3.78242275973491
1.9 3.89676460456157
2 4.00552993181239
2.1 4.10899071145882
2.2 4.20740564934029
2.3 4.30102083406355
2.4 4.39007035235239
2.5 4.47477687438634
2.6 4.55535221059216
2.7 4.63199784128017
2.8 4.70490542045002
2.9 4.7742572550255
3 4.8402267607168
3.1 4.90297889565013
3.2 4.96267057284896
3.3 5.01945105259828
3.4 5.0734623156731
3.5 5.12483941836431
3.6 5.17371083018979
3.7 5.22019875513509
3.8 5.26441943722703
3.9 5.30648345120437
4 5.34649597901223
4.1 5.38455707281172
4.2 5.42076190516248
4.3 5.45520100700363
4.4 5.48796049402832
4.5 5.51912228201775
4.6 5.54876429167334
4.7 5.57696064345908
4.8 5.60378184294125
4.9 5.62929495708909
5 5.65356378197716
5.1 5.67664900230875
5.2 5.69860834315924
5.3 5.71949671431887
5.4 5.73936634759579
5.5 5.75826692742285
5.6 5.77624571509446
5.7 5.79334766694456
5.8 5.80961554676076
5.9 5.82509003271617
6 5.83980981908599
6.1 5.85381171300332
6.2 5.86713072649622
6.3 5.87980016403599
6.4 5.89185170581569
6.5 5.90331548696715
6.6 5.91422017291445
6.7 5.92459303105246
6.8 5.93445999892951
6.9 5.94384574910476
7 5.95277375084247
7.1 5.96126632879737
7.2 5.96934471883794
7.3 5.97702912114713
7.4 5.98433875073333
7.5 5.99129188547792
7.6 5.99790591183949
7.7 6.00419736832905
7.8 6.01018198686488
7.9 6.01587473211057
8 6.02128983889446
8.1 6.02644084780411
8.2 6.03134063904483
8.3 6.03600146464692
8.4 6.04043497910209
8.5 6.0446522685058
8.6 6.04866387827824
8.7 6.0524798395334
8.8 6.05610969416207
8.9 6.05956251869151
9 6.06284694698156
9.1 6.06597119181372
9.2 6.06894306542741
9.3 6.07176999905465
9.4 6.07445906150199
9.5 6.07701697682622
9.6 6.07945014114801
9.7 6.08176463864554
9.8 6.08396625676812
9.9 6.08606050070784
10 6.08805260716538
};
\addplot [thick, color2]
table {%
0 -0.258714306893949
0.1 -0.652633049854993
0.2 -1.00906546815184
0.3 -1.33157885722786
0.4 -1.62340103948144
0.5 -1.88745266939739
0.6 -2.12637646443873
0.7 -2.34256365425128
0.8 -2.53817791289372
0.9 -2.71517701361476
1 -2.87533242290588
1.1 -3.02024702993335
1.2 -3.15137118879177
1.3 -3.27001723413537
1.4 -3.37737261546425
1.5 -3.47451178151814
1.6 -3.5624069337205
1.7 -3.64193775629717
1.8 -3.71390022045172
1.9 -3.77901455071282
2 -3.83793243318342
2.1 -3.89124353783426
2.2 -3.93948142011918
2.3 -3.98312886097738
2.4 -4.0226226986674
2.5 -4.05835820079116
2.6 -4.09069302026505
2.7 -4.11995077483046
2.8 -4.14642428592895
2.9 -4.17037850935766
3 -4.19205318703595
3.1 -4.21166524642313
3.2 -4.2294109716014
3.3 -4.24546796775288
3.4 -4.259996938692
3.5 -4.27314329524327
3.6 -4.2850386105617
3.7 -4.29580193696115
3.8 -4.30554099742992
3.9 -4.31435326375856
4 -4.32232693207042
4.1 -4.329541805518
4.2 -4.33607009297977
4.3 -4.34197713175086
4.4 -4.34732204146074
4.5 -4.35215831576227
4.6 -4.35653435771417
4.7 -4.36049396421515
4.8 -4.36407676433793
4.9 -4.36731861595036
5 -4.37025196459302
5.1 -4.37290616820504
5.2 -4.37530779094828
5.3 -4.37748086907037
5.4 -4.37944715146755
5.5 -4.38122631735494
5.6 -4.38283617322275
5.7 -4.38429283104959
5.8 -4.38561086955659
5.9 -4.38680348011613
6 -4.38788259877555
6.1 -4.3888590257171
6.2 -4.38974253334979
6.3 -4.39054196411496
6.4 -4.39126531898442
6.5 -4.39191983753683
6.6 -4.39251207041385
6.7 -4.39304794488116
6.8 -4.39353282415056
6.9 -4.39397156105674
7 -4.39436854662612
7.1 -4.39472775402372
7.2 -4.39505277831791
7.3 -4.39534687246106
7.4 -4.3956129798462
7.5 -4.3958537637655
7.6 -4.39607163406534
7.7 -4.39626877126491
7.8 -4.39644714837957
7.9 -4.39660855066744
8 -4.39675459349686
8.1 -4.39688673851355
8.2 -4.39700630826926
8.3 -4.3971144994583
8.4 -4.39721239489444
8.5 -4.39730097434811
8.6 -4.39738112435226
8.7 -4.39745364707507
8.8 -4.39751926834833
8.9 -4.3975786449318
9 -4.39763237108627
9.1 -4.39768098452117
9.2 -4.39772497177608
9.3 -4.39776477309024
9.4 -4.39780078680858
9.5 -4.3978333733685
9.6 -4.39786285890724
9.7 -4.39788953852598
9.8 -4.39791367924331
9.9 -4.39793552266766
10 -4.39795528741534
};
\addplot [thick, color3]
table {%
0 0.188962683197963
0.1 0.610098595562681
0.2 0.972573634494038
0.3 1.28455879175724
0.4 1.5530869051399
0.5 1.78421119422274
0.6 1.9831417133819
0.7 2.15436279797657
0.8 2.30173415121935
0.9 2.42857785045197
1 2.53775323414074
1.1 2.63172135771094
1.2 2.7126004711975
1.3 2.78221376930157
1.4 2.842130490245
1.5 2.89370128988225
1.6 2.93808868848064
1.7 2.97629327650704
1.8 3.00917627015739
1.9 3.03747892508098
2 3.06183924592773
2.1 3.08280636838879
2.2 3.10085293793371
2.3 3.11638576428813
2.4 3.12975499182781
2.5 3.14126199260985
2.6 3.15116615996768
2.7 3.15969075581242
2.8 3.16702794345178
2.9 3.17334311937751
3 3.17877864166931
3.1 3.1834570390619
3.2 3.18748377301458
3.3 3.1909496150466
3.4 3.19393269292859
3.5 3.19650025185592
3.6 3.19871017030461
3.7 3.20061226474064
3.8 3.20224941259365
3.9 3.20365851880933
4 3.20487134776879
4.1 3.20591523932823
4.2 3.20681372511997
4.3 3.20758705900763
4.4 3.20825267365317
4.5 3.20882557348781
4.6 3.20931867294518
4.7 3.20974308758131
4.8 3.21010838464394
4.9 3.21042279873951
5 3.21069341745947
5.1 3.21092634115031
5.2 3.21112682042891
5.3 3.21129937454311
5.4 3.21144789324557
5.5 3.21157572447742
5.6 3.21168574983831
5.7 3.21178044954404
5.8 3.21186195833613
5.9 3.21193211360362
6 3.21199249680195
6.1 3.21204446910239
6.2 3.21208920207593
6.3 3.21212770410306
6.4 3.21216084310493
6.5 3.21218936610816
6.6 3.21221391608456
6.7 3.21223504644506
6.8 3.2122532335149
6.9 3.21226888727097
7 3.21228236058368
7.1 3.2122939571714
7.2 3.21230393844695
7.3 3.21231252941043
7.4 3.21231992372122
7.5 3.2123262880635
7.6 3.21233176590367
7.7 3.21233648072439
7.8 3.2123405388082
7.9 3.21234403163329
8 3.21234703793572
8.1 3.21234962548419
8.2 3.21235185260781
8.3 3.21235376951086
8.4 3.21235541940461
8.5 3.21235683948133
8.6 3.21235806175268
8.7 3.21235911377138
8.8 3.21236001925227
8.9 3.2123607986069
9 3.21236146940364
9.1 3.21236204676374
9.2 3.21236254370219
9.3 3.21236297142108
9.4 3.21236333956213
9.5 3.21236365642408
9.6 3.21236392914968
9.7 3.21236416388678
9.8 3.21236436592688
9.9 3.2123645398244
10 3.21236468949938
};
\addplot [thick, color4]
table {%
0 -0.197825699689978
0.1 -0.771581149363957
0.2 -1.24133238075813
0.3 -1.62593216019679
0.4 -1.94081582725023
0.5 -2.19862076910885
0.6 -2.40969360330398
0.7 -2.58250542379887
0.8 -2.72399177573342
0.9 -2.83983100320306
1 -2.93467214114524
1.1 -3.01232149743542
1.2 -3.0758954133869
1.3 -3.12794533346997
1.4 -3.17056020373723
1.5 -3.20545030856346
1.6 -3.23401591036282
1.7 -3.25740344703613
1.8 -3.2765515425493
1.9 -3.29222867720882
2 -3.3050640294747
2.1 -3.31557272710138
2.2 -3.32417652102313
2.3 -3.33122071170002
2.4 -3.33698800723773
2.5 -3.34170986945654
2.6 -3.34557580326688
2.7 -3.34874096216677
2.8 -3.35133237509649
2.9 -3.35345404455597
3 -3.35519112059032
3.1 -3.35661331816007
3.2 -3.35777771504738
3.3 -3.35873104258781
3.4 -3.35951156116291
3.5 -3.3601505957237
3.6 -3.3606737929709
3.7 -3.3611021506471
3.8 -3.36145286024993
3.9 -3.36173999698716
4 -3.36197508466428
4.1 -3.3621675581752
4.2 -3.36232514215774
4.3 -3.36245416101044
4.4 -3.36255979271287
4.5 -3.36264627663616
4.6 -3.36271708368379
4.7 -3.36277505559123
4.8 -3.36282251897466
4.9 -3.36286137870632
5 -3.36289319436369
5.1 -3.36291924282081
5.2 -3.36294056949372
5.3 -3.36295803029669
5.4 -3.36297232599306
5.5 -3.36298403031931
5.6 -3.36299361301116
5.7 -3.36300145865567
5.8 -3.36300788212611
5.9 -3.3630131412189
6 -3.3630174469999
6.1 -3.36302097227523
6.2 -3.36302385852654
6.3 -3.36302622158926
6.4 -3.36302815630138
6.5 -3.36302974030969
6.6 -3.363031037186
6.7 -3.36303209897852
6.8 -3.36303296830072
6.9 -3.36303368004153
7 -3.36303426276562
7.1 -3.36303473985975
7.2 -3.36303513047139
7.3 -3.36303545027715
7.4 -3.36303571211197
7.5 -3.36303592648418
7.6 -3.3630361019973
7.7 -3.3630362456953
7.8 -3.36303636334526
7.9 -3.3630364596689
8 -3.36303653853203
8.1 -3.3630366030997
8.2 -3.36303665596324
8.3 -3.36303669924424
8.4 -3.36303673467973
8.5 -3.36303676369186
8.6 -3.36303678744498
8.7 -3.36303680689238
8.8 -3.36303682281458
8.9 -3.36303683585056
9 -3.36303684652353
9.1 -3.36303685526181
9.2 -3.36303686241611
9.3 -3.36303686827356
9.4 -3.36303687306923
9.5 -3.3630368769956
9.6 -3.36303688021023
9.7 -3.36303688284215
9.8 -3.36303688499699
9.9 -3.36303688676121
10 -3.36303688820564
};
\addplot [thick, color5]
table {%
0 -0.168005520016769
0.1 -0.168005520016769
0.2 -0.168005520016769
0.3 -0.168005520016769
0.4 -0.168005520016769
0.5 -0.168005520016769
0.6 -0.168005520016769
0.7 -0.168005520016769
0.8 -0.168005520016769
0.9 -0.168005520016769
1 -0.168005520016769
1.1 -0.168005520016769
1.2 -0.168005520016769
1.3 -0.168005520016769
1.4 -0.168005520016769
1.5 -0.168005520016769
1.6 -0.168005520016769
1.7 -0.168005520016769
1.8 -0.168005520016769
1.9 -0.168005520016769
2 -0.168005520016769
2.1 -0.168005520016769
2.2 -0.168005520016769
2.3 -0.168005520016769
2.4 -0.168005520016769
2.5 -0.168005520016769
2.6 -0.168005520016769
2.7 -0.168005520016769
2.8 -0.168005520016769
2.9 -0.168005520016769
3 -0.168005520016769
3.1 -0.168005520016769
3.2 -0.168005520016769
3.3 -0.168005520016769
3.4 -0.168005520016769
3.5 -0.168005520016769
3.6 -0.168005520016769
3.7 -0.168005520016769
3.8 -0.168005520016769
3.9 -0.168005520016769
4 -0.168005520016769
4.1 -0.168005520016769
4.2 -0.168005520016769
4.3 -0.168005520016769
4.4 -0.168005520016769
4.5 -0.168005520016769
4.6 -0.168005520016769
4.7 -0.168005520016769
4.8 -0.168005520016769
4.9 -0.168005520016769
5 -0.168005520016769
5.1 -0.168005520016769
5.2 -0.168005520016769
5.3 -0.168005520016769
5.4 -0.168005520016769
5.5 -0.168005520016769
5.6 -0.168005520016769
5.7 -0.168005520016769
5.8 -0.168005520016769
5.9 -0.168005520016769
6 -0.168005520016769
6.1 -0.168005520016769
6.2 -0.168005520016769
6.3 -0.168005520016769
6.4 -0.168005520016769
6.5 -0.168005520016769
6.6 -0.168005520016769
6.7 -0.168005520016769
6.8 -0.168005520016769
6.9 -0.168005520016769
7 -0.168005520016769
7.1 -0.168005520016769
7.2 -0.168005520016769
7.3 -0.168005520016769
7.4 -0.168005520016769
7.5 -0.168005520016769
7.6 -0.168005520016769
7.7 -0.168005520016769
7.8 -0.168005520016769
7.9 -0.168005520016769
8 -0.168005520016769
8.1 -0.168005520016769
8.2 -0.168005520016769
8.3 -0.168005520016769
8.4 -0.168005520016769
8.5 -0.168005520016769
8.6 -0.168005520016769
8.7 -0.168005520016769
8.8 -0.168005520016769
8.9 -0.168005520016769
9 -0.168005520016769
9.1 -0.168005520016769
9.2 -0.168005520016769
9.3 -0.168005520016769
9.4 -0.168005520016769
9.5 -0.168005520016769
9.6 -0.168005520016769
9.7 -0.168005520016769
9.8 -0.168005520016769
9.9 -0.168005520016769
10 -0.168005520016769
};
\addplot [thick, color6]
table {%
0 1.44162517251417
0.1 1.37131628319639
0.2 1.30443639887336
0.3 1.24081828499809
0.4 1.1803028631487
0.5 1.12273881324948
0.6 1.06798219519192
0.7 1.01589608890942
0.8 0.966350252005832
0.9 0.919220794081628
1 0.874389866943356
1.1 0.831745369921785
1.2 0.791180669561833
1.3 0.752594332983392
1.4 0.71588987424629
1.5 0.680975513085187
1.6 0.647763945411101
1.7 0.616172125005714
1.8 0.586121055862567
1.9 0.5575355946559
2 0.530344262843196
2.1 0.504479067931588
2.2 0.479875333461221
2.3 0.456471537280406
2.4 0.434209157708197
2.5 0.413032527199708
2.6 0.392888693148253
2.7 0.373727285476251
2.8 0.355500390683788
2.9 0.338162432039919
3 0.321670055617094
3.1 0.305982021883761
3.2 0.291059102584055
3.3 0.276863982646724
3.4 0.263361166878019
3.5 0.250516891205215
3.6 0.238299038248845
3.7 0.226677057012522
3.8 0.215621886489537
3.9 0.2051058829952
4 0.195102751043235
4.1 0.185587477593363
4.2 0.176536269505674
4.3 0.167926494045385
4.4 0.159736622289214
4.5 0.151946175291857
4.6 0.144535672877958
4.7 0.137486584931523
4.8 0.130781285060981
4.9 0.124403006524021
5 0.118335800302003
5.1 0.112564495219106
5.2 0.107074660006483
5.3 0.101852567216577
5.4 0.0968851588973447
5.5 0.0921600139405815
5.6 0.0876653170226771
5.7 0.0833898290601538
5.8 0.079322859106103
5.9 0.0754542376172496
6 0.0717742910247964
6.1 0.0682738175454639
6.2 0.0649440641722384
6.3 0.0617767047872957
6.4 0.0587638193423698
6.5 0.0558978740545064
6.6 0.0531717025676815
6.7 0.0505784880331788
6.8 0.0481117460639169
6.9 0.0457653085201042
7 0.0435333080856763
7.1 0.0414101635969502
7.2 0.0393905660868073
7.3 0.0374694655095111
7.4 0.0356420581129616
7.5 0.0339037744268134
7.6 0.0322502678364198
7.7 0.0306774037140315
7.8 0.0291812490800742
7.9 0.027758062768651
8 0.0264042860726786
8.1 0.0251165338452663
8.2 0.0238915860350853
8.3 0.0227263796345635
8.4 0.0216180010207706
8.5 0.0205636786698435
8.6 0.0195607762267328
8.7 0.0186067859129423
8.8 0.0176993222557761
8.9 0.0168361161234146
9 0.0160150090509028
9.1 0.015233947842864
9.2 0.0144909794394414
9.3 0.0137842460326316
9.4 0.0131119804207964
9.5 0.0124725015897388
9.6 0.0118642105092914
9.7 0.0112855861349086
9.8 0.0107351816042624
9.9 0.0102116206193331
10 0.0097135940049479
};
\addplot [thick, color7]
table {%
0 -1.03485722757579
0.1 -0.936377541835533
0.2 -0.847269437261322
0.3 -0.766641089992315
0.4 -0.69368554442892
0.5 -0.627672636949933
0.6 -0.5679416881896
0.7 -0.513894890736461
0.8 -0.464991326075851
0.9 -0.42074155089559
1 -0.380702698572811
1.1 -0.344474046815945
1.2 -0.311693007101337
1.3 -0.282031495765438
1.4 -0.255192650433219
1.5 -0.230907858919747
1.6 -0.208934070869155
1.7 -0.189051365224989
1.8 -0.171060749186352
1.9 -0.154782166621076
2 -0.140052696003426
2.1 -0.126724919840715
2.2 -0.114665449269486
2.3 -0.103753589054935
2.4 -0.0938801296324317
2.5 -0.0849462541014907
2.6 -0.0768625492330194
2.7 -0.0695481105916671
2.8 -0.0629297328170434
2.9 -0.0569411769598664
3 -0.0515225075402942
3.1 -0.046619492693498
3.2 -0.042183061398931
3.3 -0.0381688123610611
3.4 -0.0345365696262816
3.5 -0.0312499804884638
3.6 -0.0282761516588557
3.7 -0.0255853200589922
3.8 -0.0231505549418021
3.9 -0.0209474883596398
4 -0.0189540712816748
4.1 -0.0171503529197802
4.2 -0.0155182810543394
4.3 -0.0140415213615648
4.4 -0.0127052939340951
4.5 -0.0114962253587145
4.6 -0.0104022148707388
4.7 -0.00941231324549453
4.8 -0.00851661321479893
4.9 -0.0077061503116896
5 -0.00697281315102622
5.1 -0.00630926224802175
5.2 -0.00570885656221176
5.3 -0.00516558703168933
5.4 -0.00467401643239381
5.5 -0.00422922496054486
5.6 -0.00382676099359264
5.7 -0.00346259653688309
5.8 -0.00313308691013355
5.9 -0.00283493427024751
6 -0.00256515460539241
6.1 -0.00232104787000632
6.2 -0.00210017096183438
6.3 -0.00190031327054032
6.4 -0.00171947455317517
6.5 -0.00155584491507356
6.6 -0.00140778669581954
6.7 -0.00127381807899072
6.8 -0.00115259826164149
6.9 -0.00104291403509642
7 -0.000943667642750112
7.1 -0.000853865793350091
7.2 -0.000772609719804123
7.3 -0.000699086184017048
7.4 -0.000632559337730598
7.5 -0.000572363357906691
7.6 -0.000517895782946682
7.7 -0.000468611483053188
7.8 -0.000424017204387848
7.9 -0.000383666632421126
8 -0.000347155925066484
8.1 -0.000314119670893042
8.2 -0.000284227231965165
8.3 -0.000257179434706868
8.4 -0.00023270557567211
8.5 -0.000210560712253724
8.6 -0.000190523211215472
8.7 -0.000172392530512127
8.8 -0.000155987212197279
8.9 -0.000141143066331213
9 -0.000127711527712813
9.1 -0.00011555816898909
9.2 -0.000104561355261051
9.3 -9.461102672075e-05
9.4 -8.56075971357346e-05
9.5 -7.74609571565608e-05
9.6 -7.00895724721364e-05
9.7 -6.34196677869322e-05
9.8 -5.73844884530261e-05
9.9 -5.19236323671505e-05
10 -4.69824454461408e-05
};
\addplot [thick, color8]
table {%
0 0.755850732791848
0.1 0.650566754700669
0.2 0.55994799496783
0.3 0.481951705652029
0.4 0.414819677306363
0.5 0.357038605035655
0.6 0.307305975245864
0.7 0.264500704097197
0.8 0.227657865786501
0.9 0.195946940978347
1 0.168653095056153
1.1 0.145161064163604
1.2 0.124941285791964
1.3 0.107537961265946
1.4 0.0925587810300888
1.5 0.0796660811207777
1.6 0.0685692314711791
1.7 0.0590180844645799
1.8 0.0507973360519918
1.9 0.0437216723210935
2 0.0376315921094078
2.1 0.0323898114941416
2.2 0.0278781691079117
2.3 0.0239949625193062
2.4 0.0206526556343871
2.5 0.0177759054388769
2.6 0.0152998635994189
2.7 0.0131687146382353
2.8 0.0113344177283945
2.9 0.00975562374696275
3 0.00839674317401255
3.1 0.00722714382586526
3.2 0.00622046033769334
3.3 0.00535399982968837
3.4 0.00460823035919118
3.5 0.00396634062735996
3.6 0.00341386101518748
3.7 0.00293833740617841
3.8 0.00252905044292587
3.9 0.00217677388900762
4 0.00187356664914266
4.1 0.00161259375928105
4.2 0.00138797231134647
4.3 0.00119463883943303
4.4 0.00102823517804718
4.5 0.000885010219386051
4.6 0.000761735355043264
4.7 0.00065563169601071
4.8 0.000564307430353507
4.9 0.000485703906461192
5 0.000418049226471958
5.1 0.00035981830376274
5.2 0.000309698484112324
5.3 0.000266559955562226
5.4 0.000229430279947917
5.5 0.000197472471984607
5.6 0.000169966131761525
5.7 0.000146291205329256
5.8 0.000125914007307727
5.9 0.000108375190433404
6 9.32793908526152e-05
6.1 8.02863157429164e-05
6.2 6.91030723577088e-05
6.3 5.94775655737579e-05
6.4 5.11928151076778e-05
6.5 4.40620642988316e-05
6.6 3.79245701997581e-05
6.7 3.26419800734238e-05
6.8 2.80952126155037e-05
6.9 2.41817735975219e-05
7 2.0813445419492e-05
7.1 1.79142984894443e-05
7.2 1.54189796019241e-05
7.3 1.32712387317113e-05
7.4 1.14226610334251e-05
7.5 9.8315754634687e-06
7.6 8.46211542223239e-06
7.7 7.28341024134491e-06
7.8 6.26888929030152e-06
7.9 5.39568301548814e-06
8 4.64410740969186e-06
8.1 3.9972202908965e-06
8.2 3.4404393879027e-06
8.3 2.9612136235748e-06
8.4 2.54874018570937e-06
8.5 2.19372100767514e-06
8.6 1.88815316935722e-06
8.7 1.625148493578e-06
8.8 1.39877827129776e-06
8.9 1.20393961535603e-06
9 1.03624043007105e-06
9.1 8.91900403656289e-07
9.2 7.67665791603695e-07
9.3 6.60736070061956e-07
9.4 5.68700805814072e-07
9.5 4.8948531976353e-07
9.6 4.21303919063439e-07
9.7 3.6261964363704e-07
9.8 3.12109619686812e-07
9.9 2.6863523918343e-07
10 2.31216493114032e-07
};
\addplot [thick, color9]
table {%
0 -0.791302798759913
0.1 -0.647863936341418
0.2 -0.530426128492875
0.3 -0.43427618363321
0.4 -0.35555526686985
0.5 -0.291104031405195
0.6 -0.238335822856412
0.7 -0.19513286773269
0.8 -0.159761279749052
0.9 -0.130801472881643
1 -0.107091188396097
1.1 -0.0876788493235526
1.2 -0.0717853703356831
1.3 -0.0587728903149157
1.4 -0.0481191727481005
1.5 -0.0393966465415419
1.6 -0.0322552460917036
1.7 -0.0264083619233761
1.8 -0.0216213380450817
1.9 -0.0177020543802033
2 -0.0144932163137313
2.1 -0.0118660419070633
2.2 -0.00971509342662481
2.3 -0.00795404575740348
2.4 -0.00651222187297568
2.5 -0.00533175631827229
2.6 -0.00436527286568736
2.7 -0.00357398314071509
2.8 -0.00292612990828567
2.9 -0.00239571254341474
3 -0.00196144353482832
3.1 -0.00160589414238993
3.2 -0.00131479492056242
3.3 -0.00107646303545518
3.4 -0.000881333391678828
3.5 -0.000721574751481979
3.6 -0.000590775439682898
3.7 -0.000483686020631555
3.8 -0.000396008619924965
3.9 -0.000324224435616539
4 -0.000265452516338613
4.1 -0.000217334138608358
4.2 -0.000177938142972375
4.3 -0.00014568342979707
4.4 -0.000119275504188739
4.5 -9.76545233682021e-05
4.6 -7.99527614587194e-05
4.7 -6.54597845997616e-05
4.8 -5.35939387416852e-05
4.9 -4.38790058263952e-05
5 -3.59250914845577e-05
5.1 -2.94129772055473e-05
5.2 -2.40813089777633e-05
5.3 -1.97161082344677e-05
5.4 -1.61421841425727e-05
5.5 -1.3216102579372e-05
5.6 -1.08204296175651e-05
5.7 -8.85901848941639e-06
5.8 -7.25315087937164e-06
5.9 -5.93837768165617e-06
6 -4.86193243136384e-06
6.1 -3.98061360094478e-06
6.2 -3.25905077121398e-06
6.3 -2.66828509223539e-06
6.4 -2.18460706299264e-06
6.5 -1.78860498586344e-06
6.6 -1.46438590703501e-06
6.7 -1.19893777646356e-06
6.8 -9.81607228617648e-07
6.9 -8.03672025512918e-07
7 -6.57991002675899e-07
7.1 -5.38717469139374e-07
7.2 -4.41064559204745e-07
7.3 -3.61113118713708e-07
7.4 -2.95654415630813e-07
7.5 -2.42061362360247e-07
7.6 -1.98183081496287e-07
7.7 -1.6225858356077e-07
7.8 -1.32846092312076e-07
7.9 -1.08765181202133e-07
8 -8.90493987142855e-08
8.1 -7.29074812704885e-08
8.2 -5.96915970456057e-08
8.3 -4.88713462015763e-08
8.4 -4.00124740795513e-08
8.5 -3.27594430356643e-08
8.6 -2.68211634670047e-08
8.7 -2.19593113637684e-08
8.8 -1.7978763529932e-08
8.9 -1.47197666042722e-08
9 -1.20515255970479e-08
9.1 -9.8669546278096e-09
9.2 -8.07837919301281e-09
9.3 -6.61401748034488e-09
9.4 -5.4150995125537e-09
9.5 -4.43350850190531e-09
9.6 -3.62984975454256e-09
9.7 -2.97186962309656e-09
9.8 -2.43316105456743e-09
9.9 -1.99210378256601e-09
10 -1.63099663010976e-09
};
\end{axis}

\end{tikzpicture}

%% file: txz-dmd.tex
\begin{tikzpicture}

\definecolor{color0}{rgb}{0.254627,0.013882,0.615419}
\definecolor{color1}{rgb}{0.399411,0.000859,0.656133}
\definecolor{color2}{rgb}{0.534952,0.031217,0.650165}
\definecolor{color3}{rgb}{0.650746,0.125309,0.595617}
\definecolor{color4}{rgb}{0.752312,0.227133,0.513149}
\definecolor{color5}{rgb}{0.836801,0.329105,0.430905}
\definecolor{color6}{rgb}{0.907365,0.434524,0.35297}
\definecolor{color7}{rgb}{0.959424,0.543431,0.278701}
\definecolor{color8}{rgb}{0.990681,0.669558,0.201642}
\definecolor{color9}{rgb}{0.988648,0.809579,0.145357}

\begin{axis}[
height=\figureheight,
tick align=outside,
tick pos=left,
width=.95\linewidth,
x grid style={white!69.0196078431373!black},
xlabel={time \(\displaystyle t\)},
xmin=0, xmax=10,
xtick style={color=black},
y grid style={white!69.0196078431373!black},
ymin=-2.46869853866356, ymax=1.13059137490499,
ytick style={color=black}
]
\addplot [thick, color0]
table {%
0 -0.0512243212074243
0.1 -0.0512243212074334
0.2 -0.0512243212074292
0.3 -0.0512243212074254
0.4 -0.051224321207422
0.5 -0.051224321207419
0.6 -0.0512243212074162
0.7 -0.0512243212074136
0.8 -0.0512243212074113
0.9 -0.0512243212074093
1 -0.0512243212074073
1.1 -0.0512243212074056
1.2 -0.051224321207404
1.3 -0.0512243212074026
1.4 -0.0512243212074012
1.5 -0.0512243212074
1.6 -0.0512243212073989
1.7 -0.0512243212073979
1.8 -0.0512243212073969
1.9 -0.051224321207396
2 -0.0512243212073952
2.1 -0.0512243212073945
2.2 -0.0512243212073938
2.3 -0.0512243212073932
2.4 -0.0512243212073926
2.5 -0.051224321207392
2.6 -0.0512243212073915
2.7 -0.051224321207391
2.8 -0.0512243212073906
2.9 -0.0512243212073901
3 -0.0512243212073897
3.1 -0.0512243212073894
3.2 -0.051224321207389
3.3 -0.0512243212073887
3.4 -0.0512243212073884
3.5 -0.0512243212073881
3.6 -0.0512243212073878
3.7 -0.0512243212073875
3.8 -0.0512243212073872
3.9 -0.051224321207387
4 -0.0512243212073867
4.1 -0.0512243212073865
4.2 -0.0512243212073863
4.3 -0.0512243212073861
4.4 -0.0512243212073859
4.5 -0.0512243212073857
4.6 -0.0512243212073855
4.7 -0.0512243212073853
4.8 -0.0512243212073851
4.9 -0.0512243212073849
5 -0.0512243212073848
5.1 -0.0512243212073846
5.2 -0.0512243212073844
5.3 -0.0512243212073842
5.4 -0.0512243212073841
5.5 -0.0512243212073839
5.6 -0.0512243212073837
5.7 -0.0512243212073836
5.8 -0.0512243212073834
5.9 -0.0512243212073833
6 -0.0512243212073831
6.1 -0.051224321207383
6.2 -0.0512243212073828
6.3 -0.0512243212073827
6.4 -0.0512243212073825
6.5 -0.0512243212073824
6.6 -0.0512243212073822
6.7 -0.0512243212073821
6.8 -0.0512243212073819
6.9 -0.0512243212073818
7 -0.0512243212073816
7.1 -0.0512243212073815
7.2 -0.0512243212073813
7.3 -0.0512243212073812
7.4 -0.051224321207381
7.5 -0.0512243212073809
7.6 -0.0512243212073808
7.7 -0.0512243212073806
7.8 -0.0512243212073805
7.9 -0.0512243212073803
8 -0.0512243212073802
8.1 -0.05122432120738
8.2 -0.0512243212073799
8.3 -0.0512243212073798
8.4 -0.0512243212073796
8.5 -0.0512243212073795
8.6 -0.0512243212073793
8.7 -0.0512243212073792
8.8 -0.051224321207379
8.9 -0.0512243212073789
9 -0.0512243212073788
9.1 -0.0512243212073786
9.2 -0.0512243212073785
9.3 -0.0512243212073784
9.4 -0.0512243212073782
9.5 -0.0512243212073781
9.6 -0.0512243212073779
9.7 -0.0512243212073778
9.8 -0.0512243212073776
9.9 -0.0512243212073775
10 -0.0512243212073774
};
\addplot [thick, color1]
table {%
0 -0.734977255192507
0.1 -0.77407918022318
0.2 -0.811274081866979
0.3 -0.84665496675197
0.4 -0.880310305519446
0.5 -0.91232425404661
0.6 -0.942776863880099
0.7 -0.971744282406553
0.8 -0.999298943260744
0.9 -1.02550974744739
1 -1.05044223562955
1.1 -1.07415875201444
1.2 -1.0967186002464
1.3 -1.11817819169691
1.4 -1.1385911865224
1.5 -1.15800862784258
1.6 -1.17647906937486
1.7 -1.19404869684388
1.8 -1.21076144346993
1.9 -1.22665909982485
2 -1.24178141833025
2.1 -1.25616621265926
2.2 -1.26984945229041
2.3 -1.28286535245005
2.4 -1.29524645966826
2.5 -1.30702373316212
2.6 -1.31822662224988
2.7 -1.32888313998957
2.8 -1.33901993322627
2.9 -1.34866234922311
3 -1.35783449904257
3.1 -1.36655931783678
3.2 -1.37485862219726
3.3 -1.38275316470784
3.4 -1.39026268583687
3.5 -1.39740596329872
3.6 -1.4042008590078
3.7 -1.41066436374269
3.8 -1.41681263963192
3.9 -1.4226610605677
4 -1.42822425064868
4.1 -1.4335161207478
4.2 -1.43854990329672
4.3 -1.44333818537379
4.4 -1.4478929401783
4.5 -1.45222555696975
4.6 -1.45634686954685
4.7 -1.46026718333776
4.8 -1.46399630116895
4.9 -1.4675435477774
5 -1.47091779312733
5.1 -1.47412747458966
5.2 -1.4771806180399
5.3 -1.48008485792699
5.4 -1.4828474563634
5.5 -1.4854753212842
5.6 -1.48797502372047
5.7 -1.49035281423035
5.8 -1.49261463852864
5.9 -1.49476615235423
6 -1.49681273561235
6.1 -1.49875950582716
6.2 -1.50061133093824
6.3 -1.50237284147292
6.4 -1.50404844212507
6.5 -1.50564232276912
6.6 -1.50715846893687
6.7 -1.50860067178348
6.8 -1.50997253756728
6.9 -1.51127749666729
7 -1.512518812161
7.1 -1.5136995879837
7.2 -1.51482277668998
7.3 -1.51589118683667
7.4 -1.51690749000564
7.5 -1.51787422748417
7.6 -1.51879381661952
7.7 -1.51966855686351
7.8 -1.5205006355224
7.9 -1.52129213322622
8 -1.52204502913153
8.1 -1.52276120587024
8.2 -1.52344245425725
8.3 -1.52409047776836
8.4 -1.5247068967999
8.5 -1.52529325272052
8.6 -1.52585101172545
8.7 -1.52638156850271
8.8 -1.52688624972061
8.9 -1.52736631734508
9 -1.52782297179521
9.1 -1.52825735494501
9.2 -1.52867055297861
9.3 -1.52906359910631
9.4 -1.52943747614817
9.5 -1.52979311899153
9.6 -1.53013141692875
9.7 -1.53045321588087
9.8 -1.53075932051291
9.9 -1.53105049624588
10 -1.53132747117078
};
\addplot [thick, color2]
table {%
0 0.966987287924606
0.1 0.70435392758835
0.2 0.466713435931599
0.3 0.251687427040114
0.4 0.0571238483441705
0.5 -0.118924557846898
0.6 -0.278219743154167
0.7 -0.422355987333151
0.8 -0.55277585436146
0.9 -0.670784630103945
1 -0.777563386052356
1.1 -0.874180799885805
1.2 -0.961603851156172
1.3 -1.04070749914447
1.4 -1.11228343974743
1.5 -1.17704802903611
1.6 -1.23564945278823
1.7 -1.28867421374933
1.8 -1.33665300154934
1.9 -1.38006600402279
2 -1.41934771309006
2.1 -1.45489127329853
2.2 -1.48705241654536
2.3 -1.51615302236191
2.4 -1.54248433939224
2.5 -1.56630990030744
2.6 -1.58786815932922
2.7 -1.60737487875983
2.8 -1.62502528840377
2.9 -1.64099603949327
3 -1.65544697267319
3.1 -1.66852271773992
3.2 -1.68035414114499
3.3 -1.69105965575053
3.4 -1.70074640594495
3.5 -1.70951133998002
3.6 -1.71744218026158
3.7 -1.72461830130479
3.8 -1.73111152414105
3.9 -1.73698683512694
4 -1.74230303634957
4.1 -1.74711333413761
4.2 -1.75146587156813
4.3 -1.75540421029866
4.4 -1.75896776654695
4.5 -1.76219220558167
4.6 -1.76510979867247
4.7 -1.76774974607162
4.8 -1.77013846926002
4.9 -1.77229987538221
5 -1.77425559651715
5.1 -1.77602520617928
5.2 -1.77762641521689
5.3 -1.77907524906822
5.4 -1.78038620814942
5.5 -1.7815724129796
5.6 -1.78264573549541
5.7 -1.78361691786933
5.8 -1.78449568002098
5.9 -1.78529081689736
6 -1.78601028649556
6.1 -1.78666128950915
6.2 -1.78725034139511
6.3 -1.78778333758268
6.4 -1.78826561247687
6.5 -1.78870199284691
6.6 -1.78909684613422
6.7 -1.78945412416321
6.8 -1.78977740269248
6.9 -1.79006991720221
7 -1.79033459527593
7.1 -1.79057408590077
7.2 -1.79079078597939
7.3 -1.79098686431902
7.4 -1.79116428333758
7.5 -1.79132481870425
7.6 -1.79147007711093
7.7 -1.79160151235257
7.8 -1.79172043987726
7.9 -1.79182804995163
8 -1.79192541957348
8.1 -1.79201352325071
8.2 -1.79209324275453
8.3 -1.79216537594454
8.4 -1.79223064475393
8.5 -1.79228970241491
8.6 -1.79234313999638
8.7 -1.79239149231962
8.8 -1.79243524331094
8.9 -1.79247483084496
9 -1.79251065112703
9.1 -1.79254306265857
9.2 -1.79257238982509
9.3 -1.79259892614271
9.4 -1.79262293719584
9.5 -1.79264466329515
9.6 -1.79266432188276
9.7 -1.79268210970841
9.8 -1.79269820479864
9.9 -1.79271276823853
10 -1.79272594578388
};
\addplot [thick, color3]
table {%
0 -1.31689588267591
0.1 -1.41886836462977
0.2 -1.5066368932233
0.3 -1.58217996586283
0.4 -1.64720049104733
0.5 -1.70316417570495
0.6 -1.7513325654799
0.7 -1.79279148277074
0.8 -1.82847550357691
0.9 -1.85918902491569
1 -1.88562439771607
1.1 -1.90837753394512
1.2 -1.92796133978615
1.3 -1.94481727768227
1.4 -1.95932531787959
1.5 -1.97181250379971
1.6 -1.98256032432426
1.7 -1.99181105917892
1.8 -1.99977324045612
1.9 -2.00662635339114
2 -2.01252488235766
2.1 -2.01760179328831
2.2 -2.02197153102191
2.3 -2.02573259914411
2.4 -2.02896978047677
2.5 -2.03175604827091
2.6 -2.03415421118579
2.7 -2.03621832913539
2.8 -2.03799493191889
2.9 -2.03952406810559
3 -2.04084020781852
3.1 -2.04197301976752
3.2 -2.04294804004782
3.3 -2.04378724778025
3.4 -2.04450956056942
3.5 -2.04513126094854
3.6 -2.04566636342379
3.7 -2.04612693039244
3.8 -2.04652334405604
3.9 -2.04686454045826
4 -2.04715821092318
4.1 -2.04741097543477
4.2 -2.04762853186605
4.3 -2.04781578442178
4.4 -2.0479769541901
4.5 -2.04811567429525
4.6 -2.04823507179625
4.7 -2.04833783817772
4.8 -2.04842629002196
4.9 -2.04850242122982
5 -2.04856794796768
5.1 -2.04862434735363
5.2 -2.04867289075498
5.3 -2.04871467244772
5.4 -2.04875063428393
5.5 -2.04878158692321
5.6 -2.04880822810672
5.7 -2.04883115838587
5.8 -2.04885089466004
5.9 -2.04886788182863
6 -2.04888250282014
6.1 -2.04889508722416
6.2 -2.04890591872107
6.3 -2.04891524147686
6.4 -2.04892326564713
6.5 -2.04893017211448
6.6 -2.04893611656602
6.7 -2.04894123300287
6.8 -2.04894563676088
6.9 -2.04894942711053
7 -2.0489526894947
7.1 -2.04895549745478
7.2 -2.04895791428842
7.3 -2.0489599944764
7.4 -2.0489617849108
7.5 -2.04896332595196
7.6 -2.04896465233838
7.7 -2.04896579396975
7.8 -2.04896677658097
7.9 -2.04896762232229
8 -2.04896835025859
8.1 -2.04896897679917
8.2 -2.04896951606764
8.3 -2.04896998022032
8.4 -2.04897037972022
8.5 -2.04897072357298
8.6 -2.04897101952979
8.7 -2.04897127426218
8.8 -2.04897149351237
8.9 -2.04897168222277
9 -2.0489718446473
9.1 -2.0489719844474
9.2 -2.04897210477445
9.3 -2.04897220834091
9.4 -2.04897229748138
9.5 -2.0489723742053
9.6 -2.04897244024219
9.7 -2.04897249708066
9.8 -2.04897254600199
9.9 -2.04897258810896
10 -2.04897262435077
};
\addplot [thick, color4]
table {%
0 -1.02005922093539
0.1 -1.25299658996169
0.2 -1.44370957752458
0.3 -1.59985216545369
0.4 -1.72769090405642
0.5 -1.83235641078516
0.6 -1.91804927993046
0.7 -1.9882086672192
0.8 -2.04565031520959
0.9 -2.09267955892679
1 -2.13118384705205
1.1 -2.16270849186557
1.2 -2.18851868805425
1.3 -2.2096502894169
1.4 -2.22695138131427
1.5 -2.24111631731248
1.6 -2.25271358602958
1.7 -2.26220862657998
1.8 -2.26998250828031
1.9 -2.27634722429916
2 -2.28155821303839
2.1 -2.28582460977314
2.2 -2.2893176399847
2.3 -2.29217749124034
2.4 -2.29451893941256
2.5 -2.29643595503789
2.6 -2.29800547468447
2.7 -2.29929048868669
2.8 -2.30034256916843
2.9 -2.30120393981355
3 -2.30190917045051
3.1 -2.30248656446099
3.2 -2.30295929469402
3.3 -2.3033463334737
3.4 -2.30366321402526
3.5 -2.30392265387787
3.6 -2.30413506526378
3.7 -2.30430897299772
3.8 -2.3044513566077
3.9 -2.30456793044792
4 -2.30466337303591
4.1 -2.30474151481785
4.2 -2.30480549189782
4.3 -2.30485787190069
4.4 -2.30490075701988
4.5 -2.30493586838581
4.6 -2.30496461514087
4.7 -2.3049881509933
4.8 -2.30500742051947
4.9 -2.30502319707315
5 -2.30503611382282
5.1 -2.30504668916301
5.2 -2.30505534751924
5.3 -2.30506243638176
5.4 -2.3050682402515
5.5 -2.30507299205815
5.6 -2.30507688250838
5.7 -2.30508006773963
5.8 -2.30508267558641
5.9 -2.30508481071076
6 -2.30508655880273
6.1 -2.30508799001939
6.2 -2.30508916180047
6.3 -2.30509012117368
6.4 -2.30509090664203
6.5 -2.30509154972913
6.6 -2.3050920762443
6.7 -2.30509250731847
6.8 -2.30509286025215
6.9 -2.3050931492098
7 -2.30509338578832
7.1 -2.30509357948242
7.2 -2.30509373806574
7.3 -2.30509386790278
7.4 -2.30509397420436
7.5 -2.30509406123673
7.6 -2.3050941324928
7.7 -2.30509419083234
7.8 -2.30509423859671
7.9 -2.30509427770287
8 -2.30509430972029
8.1 -2.30509433593393
8.2 -2.30509435739584
8.3 -2.30509437496737
8.4 -2.30509438935372
8.5 -2.30509440113226
8.6 -2.30509441077572
8.7 -2.30509441867111
8.8 -2.30509442513531
8.9 -2.30509443042775
9 -2.30509443476083
9.1 -2.30509443830845
9.2 -2.305094441213
9.3 -2.30509444359104
9.4 -2.30509444553802
9.5 -2.30509444713207
9.6 -2.30509444843716
9.7 -2.30509444950568
9.8 -2.30509445038051
9.9 -2.30509445109676
10 -2.30509445168317
};
\addplot [thick, color5]
table {%
0 -0.30734592724463
0.1 -0.307345927244606
0.2 -0.307345927244579
0.3 -0.307345927244555
0.4 -0.307345927244534
0.5 -0.307345927244514
0.6 -0.307345927244496
0.7 -0.30734592724448
0.8 -0.307345927244466
0.9 -0.307345927244452
1 -0.30734592724444
1.1 -0.307345927244429
1.2 -0.307345927244419
1.3 -0.30734592724441
1.4 -0.307345927244401
1.5 -0.307345927244393
1.6 -0.307345927244386
1.7 -0.30734592724438
1.8 -0.307345927244374
1.9 -0.307345927244368
2 -0.307345927244363
2.1 -0.307345927244358
2.2 -0.307345927244354
2.3 -0.30734592724435
2.4 -0.307345927244346
2.5 -0.307345927244343
2.6 -0.307345927244339
2.7 -0.307345927244336
2.8 -0.307345927244334
2.9 -0.307345927244331
3 -0.307345927244328
3.1 -0.307345927244326
3.2 -0.307345927244324
3.3 -0.307345927244322
3.4 -0.30734592724432
3.5 -0.307345927244318
3.6 -0.307345927244316
3.7 -0.307345927244314
3.8 -0.307345927244313
3.9 -0.307345927244311
4 -0.30734592724431
4.1 -0.307345927244308
4.2 -0.307345927244307
4.3 -0.307345927244306
4.4 -0.307345927244305
4.5 -0.307345927244303
4.6 -0.307345927244302
4.7 -0.307345927244301
4.8 -0.3073459272443
4.9 -0.307345927244299
5 -0.307345927244298
5.1 -0.307345927244297
5.2 -0.307345927244295
5.3 -0.307345927244294
5.4 -0.307345927244293
5.5 -0.307345927244292
5.6 -0.307345927244291
5.7 -0.30734592724429
5.8 -0.307345927244289
5.9 -0.307345927244289
6 -0.307345927244288
6.1 -0.307345927244287
6.2 -0.307345927244286
6.3 -0.307345927244285
6.4 -0.307345927244284
6.5 -0.307345927244283
6.6 -0.307345927244282
6.7 -0.307345927244281
6.8 -0.30734592724428
6.9 -0.30734592724428
7 -0.307345927244279
7.1 -0.307345927244278
7.2 -0.307345927244277
7.3 -0.307345927244276
7.4 -0.307345927244275
7.5 -0.307345927244274
7.6 -0.307345927244273
7.7 -0.307345927244273
7.8 -0.307345927244272
7.9 -0.307345927244271
8 -0.30734592724427
8.1 -0.307345927244269
8.2 -0.307345927244268
8.3 -0.307345927244267
8.4 -0.307345927244267
8.5 -0.307345927244266
8.6 -0.307345927244265
8.7 -0.307345927244264
8.8 -0.307345927244263
8.9 -0.307345927244262
9 -0.307345927244262
9.1 -0.307345927244261
9.2 -0.30734592724426
9.3 -0.307345927244259
9.4 -0.307345927244258
9.5 -0.307345927244257
9.6 -0.307345927244256
9.7 -0.307345927244256
9.8 -0.307345927244255
9.9 -0.307345927244254
10 -0.307345927244253
};
\addplot [thick, color6]
table {%
0 -0.200438095257661
0.1 -0.190662613999953
0.2 -0.18136388858897
0.3 -0.172518667367693
0.4 -0.164104832675797
0.5 -0.156101345543983
0.6 -0.148488193085589
0.7 -0.141246338453955
0.8 -0.13435767324039
0.9 -0.127804972193713
1 -0.121571850148157
1.1 -0.115642721051922
1.2 -0.11000275899392
1.3 -0.104637861131281
1.4 -0.0995346124248983
1.5 -0.0946802520948424
1.6 -0.0900626417117647
1.7 -0.0856702348445015
1.8 -0.0814920481879819
1.9 -0.0775176340992447
2 -0.0737370544728878
2.1 -0.0701408558906294
2.2 -0.0667200459828377
2.3 -0.0634660709429227
2.4 -0.0603707941383649
2.5 -0.0574264757648952
2.6 -0.0546257534929528
2.7 -0.0519616240580268
2.8 -0.0494274257488469
2.9 -0.0470168217496351
3 -0.0447237842947658
3.1 -0.042542579596212
3.2 -0.0404677535060886
3.3 -0.0384941178784413
3.4 -0.0366167375961804
3.5 -0.0348309182307165
3.6 -0.033132194303444
3.7 -0.031516318119719
3.8 -0.02997924914741
3.9 -0.0285171439134627
4 -0.0271263463932155
4.1 -0.0258033788684338
4.2 -0.0245449332312025
4.3 -0.0233478627119338
4.4 -0.0222091740108029
4.5 -0.0211260198129408
4.6 -0.0200956916686628
4.7 -0.0191156132209345
4.8 -0.0181833337631363
4.9 -0.017296522111021
5 -0.016452960773539
5.1 -0.0156505404079549
5.2 -0.0148872545453928
5.3 -0.0141611945736182
5.4 -0.013470544964514
5.5 -0.0128135787343143
5.6 -0.012188653125245
5.7 -0.0115942054977743
5.8 -0.0110287494231995
5.9 -0.0104908709668015
6 -0.00997922515227045
6.1 -0.00949253259856594
6.2 -0.00902957632079682
6.3 -0.00858919868712552
6.4 -0.00817029852408535
6.5 -0.00777182836307347
6.6 -0.00739279182113337
6.7 -0.00703224110947895
6.8 -0.00668927466352875
6.9 -0.00636303488852433
7 -0.00605270601509764
7.1 -0.00575751205942196
7.2 -0.00547671488284859
7.3 -0.00520961234617531
7.4 -0.00495553655393285
7.5 -0.00471385218429851
7.6 -0.00448395490046036
7.7 -0.00426526983946066
7.8 -0.00405725017473921
7.9 -0.00385937574878131
8 -0.003671151772454
8.1 -0.00349210758777496
8.2 -0.00332179549102249
8.3 -0.00315978961324315
8.4 -0.00300568485535735
8.5 -0.00285909587520078
8.6 -0.00271965612396829
8.7 -0.00258701692965094
8.8 -0.00246084662517412
8.9 -0.00234082971905766
9 -0.00222666610652214
9.1 -0.00211807031907105
9.2 -0.00201477081067078
9.3 -0.0019165092787439
9.4 -0.00182304001827857
9.5 -0.00173412930743769
9.6 -0.00164955482313256
9.7 -0.00156910508509944
9.8 -0.00149257892708901
9.9 -0.00141978499384551
10 -0.00135054126261927
};
\addplot [thick, color7]
table {%
0 -0.689959632546218
0.1 -0.624301292462155
0.2 -0.56489116954793
0.3 -0.511134667325026
0.4 -0.46249377265101
0.5 -0.418481671103216
0.6 -0.378657874776375
0.7 -0.342623813731606
0.8 -0.310018846974509
0.9 -0.280516653038869
1 -0.25382196405175
1.1 -0.229667610593372
1.2 -0.207811847775766
1.3 -0.188035935778678
1.4 -0.170141950627927
1.5 -0.153950803305747
1.6 -0.139300447367707
1.7 -0.126044257127424
1.8 -0.114049560177412
1.9 -0.10319630955904
2 -0.0933758822922157
2.1 -0.0844899922400924
2.2 -0.0764497064283779
2.3 -0.0691745549742353
2.4 -0.0625917257166485
2.5 -0.0566353354878422
2.6 -0.0512457707323946
2.7 -0.046369090874738
2.8 -0.041956488463748
2.9 -0.0379638006913687
3 -0.0343510673963853
3.1 -0.0310821311297002
3.2 -0.0281242752784287
3.3 -0.0254478966270417
3.4 -0.0230262090784342
3.5 -0.0208349755696624
3.6 -0.0188522654992716
3.7 -0.0170582352384652
3.8 -0.0154349295293989
3.9 -0.0139661017829244
4 -0.0126370514772648
4.1 -0.0114344770302519
4.2 -0.0103463426726207
4.3 -0.00936175798998548
4.4 -0.0084708689279121
4.5 -0.00766475916922915
4.6 -0.00693536089652891
4.7 -0.00627537404673884
4.8 -0.00567819324963725
4.9 -0.00513784171908718
5 -0.00464891143535243
5.1 -0.00420650901981828
5.2 -0.00380620676041349
5.3 -0.00344399829757971
5.4 -0.00311625852727837
5.5 -0.0028197073197313
5.6 -0.00255137669077893
5.7 -0.00230858109729756
5.8 -0.0020888905593815
5.9 -0.00189010634028652
6 -0.00171023894073444
6.1 -0.00154748818733486
6.2 -0.00140022521584535
6.3 -0.00126697616895041
6.4 -0.00114640744540231
6.5 -0.00103731235289117
6.6 -0.000938599031063081
6.7 -0.000849279523814306
6.8 -0.000768459891494988
6.9 -0.000695331264060711
7 -0.000629161745628409
7.1 -0.000569289089417474
7.2 -0.00051511406976068
7.3 -0.000466094484852381
7.4 -0.000421739730210813
7.5 -0.000381605888543107
7.6 -0.000345291286872801
7.7 -0.000312432476460428
7.8 -0.000282700595287066
7.9 -0.000255798076692848
8 -0.000231455671229397
8.1 -0.000209429751920981
8.2 -0.0001894998759642
8.3 -0.000171466578461676
8.4 -0.000155149376110747
8.5 -0.000140384960866019
8.6 -0.000127025565497041
8.7 -0.000114937484684968
8.8 -0.000103999736854044
8.9 -9.41028533474608e-05
9 -8.51477828288094e-05
9.1 -7.70448999423412e-05
9.2 -6.97131083127189e-05
9.3 -6.30790289049832e-05
9.4 -5.7076265622652e-05
9.5 -5.16447407931941e-05
9.6 -4.67300938905013e-05
9.7 -4.22831374764787e-05
9.8 -3.82593649167207e-05
9.9 -3.46185049429626e-05
10 -3.13241186049029e-05
};
\addplot [thick, color8]
table {%
0 -0.183019241405419
0.1 -0.157526120916876
0.2 -0.13558398876845
0.3 -0.116698220608528
0.4 -0.100443089312371
0.5 -0.0864521681479331
0.6 -0.074410070704168
0.7 -0.0640453413814325
0.8 -0.0551243361798681
0.9 -0.0474459558451527
1 -0.0408371126450384
1.1 -0.0351488285877579
1.2 -0.0302528771274855
1.3 -0.0260388926534394
1.4 -0.0224118826040968
1.5 -0.0192900861240541
1.6 -0.0166031309929057
1.7 -0.0142904472792304
1.8 -0.0122999019599218
1.9 -0.0105866237261575
2 -0.0091119914845211
2.1 -0.00784276375185131
2.2 -0.00675032931844249
2.3 -0.00581006228788628
2.4 -0.00500076695471705
2.5 -0.00430420000617485
2.6 -0.00370465927745059
2.7 -0.00318862979004589
2.8 -0.00274447909416611
2.9 -0.00236219504748772
3 -0.00203316011925167
3.1 -0.00174995713199649
3.2 -0.00150620206191845
3.3 -0.00129640012880833
3.4 -0.00111582193151097
3.5 -0.000960396836728883
3.6 -0.000826621217913243
3.7 -0.000711479475747429
3.8 -0.000612376059845848
3.9 -0.000527076959288386
4 -0.0004536593430568
4.1 -0.000390468215156065
4.2 -0.000336079107332632
4.3 -0.000289265968398317
4.4 -0.000248973526316143
4.5 -0.00021429350002633
4.6 -0.000184444124776056
4.7 -0.000158752529406808
4.8 -0.00013663956834542
4.9 -0.000117606766377509
5 -0.000101225081910072
5.1 -8.71252354216201e-05
5.2 -7.49893850826855e-05
5.3 -6.45439618951749e-05
5.4 -5.55535028406162e-05
5.5 -4.78153430206463e-05
5.6 -4.11550471407784e-05
5.7 -3.54224773516099e-05
5.8 -3.0488408808621e-05
5.9 -2.62416166573742e-05
6 -2.25863687786931e-05
6.1 -1.94402677737293e-05
6.2 -1.67323935440222e-05
6.3 -1.44017045954903e-05
6.4 -1.23956620268206e-05
6.5 -1.06690451869007e-05
6.6 -9.18293230051104e-06
6.7 -7.90382308546445e-06
6.8 -6.80288358123615e-06
6.9 -5.85529616847769e-06
7 -5.03970012394328e-06
7.1 -4.33771010283079e-06
7.2 -3.73350169219355e-06
7.3 -3.21345469392331e-06
7.4 -2.7658460943214e-06
7.5 -2.38058580229183e-06
7.6 -2.04898919603513e-06
7.7 -1.76358135201005e-06
7.8 -1.51792854410915e-06
7.9 -1.30649321300469e-06
8 -1.12450913683504e-06
8.1 -9.67873991009061e-07
8.2 -8.3305687162416e-07
8.3 -7.17018701551497e-07
8.4 -6.17143723125047e-07
8.5 -5.31180532448161e-07
8.6 -4.57191328573359e-07
8.7 -3.93508230667194e-07
8.8 -3.38695680252021e-07
8.9 -2.91518080924869e-07
9 -2.50911944832399e-07
9.1 -2.15961919708407e-07
9.2 -1.85880154324192e-07
9.3 -1.59988538750433e-07
9.4 -1.37703418889856e-07
9.5 -1.18522438419255e-07
9.6 -1.02013215513841e-07
9.7 -8.78035955886869e-08
9.8 -7.55732624813099e-08
9.9 -6.50465172002956e-08
10 -5.59860635074294e-08
};
\addplot [thick, color9]
table {%
0 -0.321258808349837
0.1 -0.2630244660932
0.2 -0.215346219202428
0.3 -0.176310572220108
0.4 -0.144350887569386
0.5 -0.118184510887165
0.6 -0.0967612936008083
0.7 -0.0792214467785956
0.8 -0.0648610347809727
0.9 -0.0531037238516485
1 -0.0434776518203108
1.1 -0.0355964906169105
1.2 -0.0291439415697216
1.3 -0.0238610412290438
1.4 -0.0195357682546848
1.5 -0.0159945342551197
1.6 -0.0130952170758304
1.7 -0.0107214569382189
1.8 -0.00877798651312572
1.9 -0.00718680750840431
2 -0.00588406032358685
2.1 -0.00481746113989093
2.2 -0.0039442035869913
2.3 -0.00322924077307457
2.4 -0.00264387873001382
2.5 -0.00216462482367535
2.6 -0.00177224491202327
2.7 -0.00145099141146376
2.8 -0.00118797129102172
2.9 -0.000972628629737404
3 -0.000796320970494341
3.1 -0.000651972467868772
3.2 -0.000533789909608506
3.3 -0.000437030214683443
3.4 -0.000357810076789894
3.5 -0.000292950113633283
3.6 -0.000239847267153481
3.7 -0.000196370333664458
3.8 -0.000160774431167493
3.9 -0.000131630971109686
4 -0.000107770324109274
4.1 -8.82348786216403e-05
4.2 -7.22406086258731e-05
4.3 -5.9145607907235e-05
4.4 -4.84243281073382e-05
4.5 -3.96464866228019e-05
4.6 -3.24597978537955e-05
4.7 -2.65758347458844e-05
4.8 -2.17584531993498e-05
4.9 -1.78143147780419e-05
5 -1.45851273579967e-05
5.1 -1.19412923098305e-05
5.2 -9.77670324989188e-06
5.3 -8.00448761867356e-06
5.4 -6.55352018033573e-06
5.5 -5.36556851685188e-06
5.6 -4.3929559567224e-06
5.7 -3.59664814297833e-06
5.8 -2.9446864469318e-06
5.9 -2.41090535660349e-06
6 -1.97388236244933e-06
6.1 -1.61607819740284e-06
6.2 -1.32313292383612e-06
6.3 -1.08328961934134e-06
6.4 -8.86922530113354e-07
6.5 -7.26150755331334e-07
6.6 -5.94521959051963e-07
6.7 -4.86753415573915e-07
6.8 -3.98519994809776e-07
6.9 -3.26280579721061e-07
7 -2.67135949083563e-07
7.1 -2.18712421118372e-07
7.2 -1.79066589639509e-07
7.3 -1.46607328110893e-07
7.4 -1.20031932449027e-07
7.5 -9.82738388166648e-08
7.6 -8.04598183634653e-08
7.7 -6.58749319815133e-08
7.8 -5.39338369520692e-08
7.9 -4.41572951814795e-08
8 -3.61529399195426e-08
8.1 -2.95995280785721e-08
8.2 -2.42340482126124e-08
8.3 -1.98411649054986e-08
8.4 -1.62445761564545e-08
8.5 -1.32999384283572e-08
8.6 -1.08890730599054e-08
8.7 -8.91522325408012e-09
8.8 -7.29917179476125e-09
8.9 -5.97606081287338e-09
9 -4.89278914606039e-09
9.1 -4.0058813211008e-09
9.2 -3.2797425886244e-09
9.3 -2.6852304824132e-09
9.4 -2.19848517168941e-09
9.5 -1.79997172722324e-09
9.6 -1.47369658143326e-09
9.7 -1.20656507007766e-09
9.8 -9.87856296763567e-10
9.9 -8.08792643969269e-10
10 -6.62187721323093e-10
};
\end{axis}

\end{tikzpicture}

%% file: hxz-dmd.tex
\begin{tikzpicture}

\definecolor{color0}{rgb}{0.254627,0.013882,0.615419}
\definecolor{color1}{rgb}{0.399411,0.000859,0.656133}
\definecolor{color2}{rgb}{0.534952,0.031217,0.650165}
\definecolor{color3}{rgb}{0.650746,0.125309,0.595617}
\definecolor{color4}{rgb}{0.752312,0.227133,0.513149}
\definecolor{color5}{rgb}{0.836801,0.329105,0.430905}
\definecolor{color6}{rgb}{0.907365,0.434524,0.35297}
\definecolor{color7}{rgb}{0.959424,0.543431,0.278701}
\definecolor{color8}{rgb}{0.990681,0.669558,0.201642}
\definecolor{color9}{rgb}{0.988648,0.809579,0.145357}

\begin{axis}[
height=\figureheight,
tick align=outside,
tick pos=left,
width=.95\linewidth,
x grid style={white!69.0196078431373!black},
xlabel={time \(\displaystyle t\)},
xmin=0, xmax=10,
xtick style={color=black},
y grid style={white!69.0196078431373!black},
ymin=-1.15868134758029, ymax=1.56544929251866,
ytick style={color=black}
]
\addplot [thick, color0]
table {%
0 0.594494225663158
0.1 -6.93889390390723e-18
0.2 6.90626076138649e-18
0.3 6.90626076151389e-18
0.4 6.90626076162561e-18
0.5 6.90626076172343e-18
0.6 6.90626076180902e-18
0.7 6.90626076188406e-18
0.8 6.90626076194978e-18
0.9 6.90626076200752e-18
1 6.9062607620583e-18
1.1 6.90626076210272e-18
1.2 6.9062607621419e-18
1.3 6.90626076217631e-18
1.4 6.90626076220671e-18
1.5 6.90626076223343e-18
1.6 6.90626076225697e-18
1.7 6.90626076227768e-18
1.8 6.90626076229604e-18
1.9 6.90626076231219e-18
2 6.90626076232651e-18
2.1 6.9062607623391e-18
2.2 6.90626076235027e-18
2.3 6.90626076236014e-18
2.4 6.90626076236887e-18
2.5 6.90626076237658e-18
2.6 6.90626076238342e-18
2.7 6.90626076238949e-18
2.8 6.90626076239485e-18
2.9 6.90626076239959e-18
3 6.9062607624038e-18
3.1 6.90626076240755e-18
3.2 6.90626076241085e-18
3.3 6.90626076241377e-18
3.4 6.90626076241636e-18
3.5 6.90626076241866e-18
3.6 6.9062607624207e-18
3.7 6.90626076242249e-18
3.8 6.90626076242409e-18
3.9 6.9062607624255e-18
4 6.90626076242674e-18
4.1 6.90626076242784e-18
4.2 6.90626076242882e-18
4.3 6.90626076242968e-18
4.4 6.90626076243043e-18
4.5 6.9062607624311e-18
4.6 6.90626076243169e-18
4.7 6.90626076243221e-18
4.8 6.90626076243265e-18
4.9 6.90626076243305e-18
5 6.90626076243339e-18
5.1 6.90626076243369e-18
5.2 6.90626076243395e-18
5.3 6.90626076243418e-18
5.4 6.90626076243437e-18
5.5 6.90626076243454e-18
5.6 6.90626076243468e-18
5.7 6.90626076243479e-18
5.8 6.90626076243489e-18
5.9 6.90626076243498e-18
6 6.90626076243505e-18
6.1 6.9062607624351e-18
6.2 6.90626076243515e-18
6.3 6.90626076243518e-18
6.4 6.90626076243521e-18
6.5 6.90626076243522e-18
6.6 6.90626076243523e-18
6.7 6.90626076243524e-18
6.8 6.90626076243523e-18
6.9 6.90626076243523e-18
7 6.90626076243522e-18
7.1 6.9062607624352e-18
7.2 6.90626076243519e-18
7.3 6.90626076243517e-18
7.4 6.90626076243514e-18
7.5 6.90626076243512e-18
7.6 6.90626076243509e-18
7.7 6.90626076243507e-18
7.8 6.90626076243504e-18
7.9 6.90626076243501e-18
8 6.90626076243498e-18
8.1 6.90626076243494e-18
8.2 6.90626076243491e-18
8.3 6.90626076243488e-18
8.4 6.90626076243485e-18
8.5 6.90626076243482e-18
8.6 6.90626076243479e-18
8.7 6.90626076243476e-18
8.8 6.90626076243473e-18
8.9 6.9062607624347e-18
9 6.90626076243466e-18
9.1 6.90626076243463e-18
9.2 6.9062607624346e-18
9.3 6.90626076243457e-18
9.4 6.90626076243454e-18
9.5 6.90626076243451e-18
9.6 6.90626076243448e-18
9.7 6.90626076243445e-18
9.8 6.90626076243442e-18
9.9 6.90626076243439e-18
10 6.90626076243437e-18
};
\addplot [thick, color1]
table {%
0 0.360406293128509
0.1 -3.89965837399586e-14
0.2 -3.66303796852097e-14
0.3 -3.48337903178453e-14
0.4 -3.31248216478632e-14
0.5 -3.14992003634266e-14
0.6 -2.9952861564576e-14
0.7 -2.84819385988624e-14
0.8 -2.70827533927018e-14
0.9 -2.57518072542759e-14
1 -2.44857721249796e-14
1.1 -2.32814822575414e-14
1.2 -2.21359263000063e-14
1.3 -2.10462397657868e-14
1.4 -2.00096978709551e-14
1.5 -1.90237087208635e-14
1.6 -1.8085806829058e-14
1.7 -1.71936469522777e-14
1.8 -1.63449982261253e-14
1.9 -1.55377385867442e-14
2 -1.47698494645531e-14
2.1 -1.40394107367709e-14
2.2 -1.33445959261096e-14
2.3 -1.26836676336297e-14
2.4 -1.20549731943378e-14
2.5 -1.14569405446634e-14
2.6 -1.0888074291481e-14
2.7 -1.03469519728484e-14
2.8 -9.83222050111109e-15
2.9 -9.34259277947798e-15
3 -8.87684448360933e-15
3.1 -8.43381100016801e-15
3.2 -8.01238451467959e-15
3.3 -7.61151124141908e-15
3.4 -7.23018878839777e-15
3.5 -6.86746365086111e-15
3.6 -6.52242882703018e-15
3.7 -6.19422155012478e-15
3.8 -5.88202113099712e-15
3.9 -5.58504690598143e-15
4 -5.30255628482821e-15
4.1 -5.03384289384179e-15
4.2 -4.77823480957815e-15
4.3 -4.53509287868631e-15
4.4 -4.30380911969208e-15
4.5 -4.08380520272764e-15
4.6 -3.87453100340565e-15
4.7 -3.67546322722175e-15
4.8 -3.4861041010457e-15
4.9 -3.3059801284293e-15
5 -3.13464090561862e-15
5.1 -2.97165799531002e-15
5.2 -2.81662385533372e-15
5.3 -2.66915081958611e-15
5.4 -2.52887012866253e-15
5.5 -2.39543100776673e-15
5.6 -2.26849978959114e-15
5.7 -2.1477590799748e-15
5.8 -2.03290696425264e-15
5.9 -1.92365625231157e-15
6 -1.81973376046556e-15
6.1 -1.72087962835421e-15
6.2 -1.62684666915641e-15
6.3 -1.53739975149459e-15
6.4 -1.45231521148377e-15
6.5 -1.37138029345538e-15
6.6 -1.29439261795721e-15
6.7 -1.22115967569945e-15
6.8 -1.1514983461811e-15
6.9 -1.08523443979341e-15
7 -1.02220226225507e-15
7.1 -9.6224420029026e-16
7.2 -9.05210327513291e-16
7.3 -8.50958029534608e-16
7.4 -7.99351647350505e-16
7.5 -7.50262138124956e-16
7.6 -7.03566752515316e-16
7.7 -6.59148727735019e-16
7.8 -6.16896995585799e-16
7.9 -5.76705904729338e-16
8 -5.38474956503892e-16
8.1 -5.02108553625284e-16
8.2 -4.67515761143905e-16
8.3 -4.3461007905997e-16
8.4 -4.03309226028466e-16
8.5 -3.73534933612926e-16
8.6 -3.45212750573577e-16
8.7 -3.18271856700453e-16
8.8 -2.92644885725987e-16
8.9 -2.68267756874249e-16
9 -2.45079514625631e-16
9.1 -2.23022176296295e-16
9.2 -2.02040587051263e-16
9.3 -1.82082281988601e-16
9.4 -1.63097354949836e-16
9.5 -1.45038333728563e-16
9.6 -1.27860061365206e-16
9.7 -1.11519583231093e-16
9.8 -9.59760396195146e-17
9.9 -8.11905635751712e-17
10 -6.71261837065414e-17
};
\addplot [thick, color2]
table {%
0 -0.258714306893949
0.1 3.23213678043999e-14
0.2 2.81520471882763e-14
0.3 2.54960283064681e-14
0.4 2.30927630391985e-14
0.5 2.09181986999071e-14
0.6 1.89505715177899e-14
0.7 1.71701888186659e-14
0.8 1.55592319340749e-14
0.9 1.41015778660546e-14
1 1.27826379227576e-14
1.1 1.15892117099205e-14
1.2 1.05093550168807e-14
1.3 9.53226027490178e-15
1.4 8.64814839139322e-15
1.5 7.84817087746453e-15
1.6 7.12432128927457e-15
1.7 6.46935509685045e-15
1.8 5.87671717839665e-15
1.9 5.34047621443277e-15
2 4.85526532515464e-15
2.1 4.41622835689734e-15
2.2 4.01897128011712e-15
2.3 3.65951821246684e-15
2.4 3.33427162682909e-15
2.5 3.03997634605565e-15
2.6 2.77368696406048e-15
2.7 2.53273836720561e-15
2.8 2.31471906094809e-15
2.9 2.11744703479206e-15
3 1.93894792399435e-15
3.1 1.77743524945844e-15
3.2 1.63129253805132e-15
3.3 1.49905714439693e-15
3.4 1.37940561222975e-15
3.5 1.27114042879956e-15
3.6 1.17317803976141e-15
3.7 1.08453800459951e-15
3.8 1.00433318404901e-15
3.9 9.31760861308065e-16
4 8.66094708178288e-16
4.1 8.06677515727995e-16
4.2 7.52914616724335e-16
4.3 7.04267934003742e-16
4.4 6.60250595214833e-16
4.5 6.20422060036269e-16
4.6 5.84383711101148e-16
4.7 5.5177486450042e-16
4.8 5.22269159937092e-16
4.9 4.95571294402683e-16
5 4.71414066685461e-16
5.1 4.49555703130908e-16
5.2 4.29777437889719e-16
5.3 4.11881323435654e-16
5.4 3.95688249440166e-16
5.5 3.81036150176026e-16
5.6 3.67778382509059e-16
5.7 3.55782258244362e-16
5.8 3.44927716138257e-16
5.9 3.3510612028501e-16
6 3.26219172852167e-16
6.1 3.18177930282814e-16
6.2 3.10901913118562e-16
6.3 3.04318300534077e-16
6.4 2.98361201521784e-16
6.5 2.92970995432518e-16
6.6 2.88093735272027e-16
6.7 2.83680607781319e-16
6.8 2.79687444897166e-16
6.9 2.76074281703273e-16
7 2.72804956447969e-16
7.1 2.69846748625241e-16
7.2 2.67170051496912e-16
7.3 2.64748075778451e-16
7.4 2.62556581522814e-16
7.5 2.60573635518904e-16
7.6 2.58779391776622e-16
7.7 2.57155892901529e-16
7.8 2.55686890371207e-16
7.9 2.54357681914582e-16
8 2.53154964366659e-16
8.1 2.52066700525971e-16
8.2 2.51081998682221e-16
8.3 2.50191003608387e-16
8.4 2.49384797926298e-16
8.5 2.48655312858511e-16
8.6 2.47995247473279e-16
8.7 2.47397995614371e-16
8.8 2.4685757978444e-16
8.9 2.4636859132022e-16
9 2.45926136260806e-16
9.1 2.4552578636725e-16
9.2 2.45163534803253e-16
9.3 2.44835756033408e-16
9.4 2.44539169537614e-16
9.5 2.44270806978536e-16
9.6 2.44027982493482e-16
9.7 2.43808265813391e-16
9.8 2.43609457939877e-16
9.9 2.43429569136923e-16
10 2.43266799016924e-16
};
\addplot [thick, color3]
table {%
0 0.188962683197963
0.1 -2.92266211232572e-14
0.2 -2.45227469216633e-14
0.3 -2.1068444397853e-14
0.4 -1.8095298662624e-14
0.5 -1.55362884132381e-14
0.6 -1.33337278798381e-14
0.7 -1.14379664601815e-14
0.8 -9.80626948488401e-15
0.9 -8.40185488313674e-15
1 -7.19306403320496e-15
1.1 -6.15264810683896e-15
1.2 -5.257153820216e-15
1.3 -4.48639474487649e-15
1.4 -3.8229962608297e-15
1.5 -3.2520038940623e-15
1.6 -2.76054620950772e-15
1.7 -2.33754466033617e-15
1.8 -1.97346385292402e-15
1.9 -1.66009659792107e-15
2 -1.39037890198958e-15
2.1 -1.15823072971833e-15
2.2 -9.5841894613195e-16
2.3 -7.86439350215412e-16
2.4 -6.38415140227671e-16
2.5 -5.11009521987188e-16
2.6 -4.01350490126223e-16
2.7 -3.06966086716429e-16
2.8 -2.25728677851487e-16
2.9 -1.55806992057317e-16
3 -9.56248393691796e-17
3.1 -4.38255805120603e-17
3.2 7.5845475917597e-19
3.3 3.91322895383517e-17
3.4 7.21609552188115e-17
3.5 1.00588991220665e-16
3.6 1.25057228561564e-16
3.7 1.46117235609941e-16
3.8 1.64243751660051e-16
3.9 1.79845388609181e-16
4 1.93273841976588e-16
4.1 2.0483181890097e-16
4.2 2.14779861831124e-16
4.3 2.23342221730929e-16
4.4 2.30711913193716e-16
4.5 2.37055065419529e-16
4.6 2.42514667135967e-16
4.7 2.4721378988141e-16
4.8 2.51258362310614e-16
4.9 2.5473955806166e-16
5 2.57735851012083e-16
5.1 2.6031478425422e-16
5.2 2.62534492666395e-16
5.3 2.64445013402093e-16
5.4 2.66089413838434e-16
5.5 2.6750476241043e-16
5.6 2.68722964215769e-16
5.7 2.6977148022652e-16
5.8 2.70673946320383e-16
5.9 2.71450706085825e-16
6 2.72119269411706e-16
6.1 2.72694707199038e-16
6.2 2.73189991092531e-16
6.3 2.73616285890256e-16
6.4 2.73983201222966e-16
6.5 2.74299008176502e-16
6.6 2.74570825740422e-16
6.7 2.74804781285819e-16
6.8 2.75006148689872e-16
6.9 2.75179467220733e-16
7 2.75328643862706e-16
7.1 2.75457041388349e-16
7.2 2.75567554162823e-16
7.3 2.7566267338931e-16
7.4 2.75744543266259e-16
7.5 2.75815009322377e-16
7.6 2.75875660018945e-16
7.7 2.75927862557257e-16
7.8 2.75972793698372e-16
7.9 2.76011466289919e-16
8 2.76044752097933e-16
8.1 2.76073401458392e-16
8.2 2.76098060191458e-16
8.3 2.76119284159697e-16
8.4 2.76137551798451e-16
8.5 2.76153274900837e-16
8.6 2.76166807900475e-16
8.7 2.76178455861209e-16
8.8 2.76188481353921e-16
8.9 2.76197110375465e-16
9 2.76204537443138e-16
9.1 2.76210929979524e-16
9.2 2.76216432086582e-16
9.3 2.76221167794013e-16
9.4 2.76225243855174e-16
9.5 2.76228752153526e-16
9.6 2.76231771773902e-16
9.7 2.76234370785246e-16
9.8 2.76236607775039e-16
9.9 2.76238533169998e-16
10 2.76240190372796e-16
};
\addplot [thick, color4]
table {%
0 -0.197825699689978
0.1 7.43849426498855e-15
0.2 6.4127311298384e-15
0.3 5.30663535815293e-15
0.4 4.40104073402471e-15
0.5 3.65960256542903e-15
0.6 3.05256433529409e-15
0.7 2.55556346798873e-15
0.8 2.14865357361953e-15
0.9 1.81550392936787e-15
1 1.54274407024214e-15
1.1 1.31942718537078e-15
1.2 1.1365907840451e-15
1.3 9.86896999497729e-16
1.4 8.64338094544227e-16
1.5 7.63995349995288e-16
1.6 6.81841659184872e-16
1.7 6.14579906039565e-16
1.8 5.595106402336e-16
1.9 5.14423738768863e-16
2 4.77509705978715e-16
2.1 4.47287052113316e-16
2.2 4.22542835954104e-16
2.3 4.02283985223768e-16
2.4 3.85697441108841e-16
2.5 3.72117527354681e-16
2.6 3.60999234340015e-16
2.7 3.51896345927186e-16
2.8 3.44443531241776e-16
2.9 3.38341682661846e-16
3 3.33345911578839e-16
3.1 3.29255720157849e-16
3.2 3.25906954655514e-16
3.3 3.23165217353908e-16
3.4 3.20920472708226e-16
3.5 3.19082631234001e-16
3.6 3.17577933899773e-16
3.7 3.16345991918167e-16
3.8 3.15337363131819e-16
3.9 3.14511567725997e-16
4 3.138354636315e-16
4.1 3.13281916417054e-16
4.2 3.12828710289305e-16
4.3 3.12457656495034e-16
4.4 3.12153863342617e-16
4.5 3.11905138546158e-16
4.6 3.11701499906243e-16
4.7 3.11534774689229e-16
4.8 3.11398271626744e-16
4.9 3.11286512371598e-16
5 3.11195011632467e-16
5.1 3.11120097163411e-16
5.2 3.11058762383742e-16
5.3 3.11008545713393e-16
5.4 3.10967431781059e-16
5.5 3.10933770540277e-16
5.6 3.1090621104726e-16
5.7 3.10883647242787e-16
5.8 3.10865173562156e-16
5.9 3.10850048591701e-16
6 3.10837665313248e-16
6.1 3.10827526742354e-16
6.2 3.10819225982569e-16
6.3 3.10812429895259e-16
6.4 3.10806865729576e-16
6.5 3.10802310176015e-16
6.6 3.10798580404216e-16
6.7 3.10795526725341e-16
6.8 3.10793026584535e-16
6.9 3.1079097964237e-16
7 3.10789303747868e-16
7.1 3.10787931641499e-16
7.2 3.10786808255818e-16
7.3 3.10785888505413e-16
7.4 3.10785135477469e-16
7.5 3.10784518950334e-16
7.6 3.10784014180607e-16
7.7 3.10783600910107e-16
7.8 3.10783262552839e-16
7.9 3.10782985529337e-16
8 3.10782758721677e-16
8.1 3.10782573027269e-16
8.2 3.10782420993546e-16
8.3 3.10782296518861e-16
8.4 3.10782194607608e-16
8.5 3.1078211116973e-16
8.6 3.10782042856573e-16
8.7 3.1078198692649e-16
8.8 3.1078194113481e-16
8.9 3.10781903643754e-16
9 3.10781872948672e-16
9.1 3.10781847817664e-16
9.2 3.10781827242135e-16
9.3 3.10781810396315e-16
9.4 3.10781796604125e-16
9.5 3.10781785312034e-16
9.6 3.10781776066851e-16
9.7 3.10781768497535e-16
9.8 3.10781762300303e-16
9.9 3.10781757226438e-16
10 3.10781753072309e-16
};
\addplot [thick, color5]
table {%
0 -0.168005520016769
0.1 2.27595720048157e-15
0.2 4.14375645681103e-17
0.3 4.14375645689035e-17
0.4 4.14375645695933e-17
0.5 4.1437564570198e-17
0.6 4.14375645707273e-17
0.7 4.14375645711924e-17
0.8 4.14375645715994e-17
0.9 4.14375645719566e-17
1 4.14375645722707e-17
1.1 4.14375645725466e-17
1.2 4.14375645727894e-17
1.3 4.14375645730032e-17
1.4 4.14375645731913e-17
1.5 4.14375645733572e-17
1.6 4.14375645735035e-17
1.7 4.14375645736322e-17
1.8 4.14375645737461e-17
1.9 4.14375645738468e-17
2 4.14375645739358e-17
2.1 4.14375645740141e-17
2.2 4.14375645740835e-17
2.3 4.14375645741448e-17
2.4 4.14375645741991e-17
2.5 4.14375645742471e-17
2.6 4.14375645742899e-17
2.7 4.14375645743274e-17
2.8 4.14375645743608e-17
2.9 4.14375645743904e-17
3 4.14375645744166e-17
3.1 4.14375645744399e-17
3.2 4.14375645744605e-17
3.3 4.14375645744786e-17
3.4 4.14375645744946e-17
3.5 4.14375645745089e-17
3.6 4.14375645745215e-17
3.7 4.14375645745328e-17
3.8 4.14375645745426e-17
3.9 4.14375645745514e-17
4 4.14375645745591e-17
4.1 4.14375645745659e-17
4.2 4.14375645745719e-17
4.3 4.14375645745772e-17
4.4 4.14375645745819e-17
4.5 4.1437564574586e-17
4.6 4.14375645745896e-17
4.7 4.14375645745928e-17
4.8 4.14375645745955e-17
4.9 4.14375645745979e-17
5 4.14375645746e-17
5.1 4.14375645746018e-17
5.2 4.14375645746034e-17
5.3 4.14375645746048e-17
5.4 4.14375645746059e-17
5.5 4.14375645746069e-17
5.6 4.14375645746077e-17
5.7 4.14375645746084e-17
5.8 4.1437564574609e-17
5.9 4.14375645746094e-17
6 4.14375645746099e-17
6.1 4.14375645746101e-17
6.2 4.14375645746104e-17
6.3 4.14375645746106e-17
6.4 4.14375645746107e-17
6.5 4.14375645746108e-17
6.6 4.14375645746108e-17
6.7 4.14375645746108e-17
6.8 4.14375645746107e-17
6.9 4.14375645746107e-17
7 4.14375645746106e-17
7.1 4.14375645746105e-17
7.2 4.14375645746103e-17
7.3 4.14375645746102e-17
7.4 4.143756457461e-17
7.5 4.14375645746099e-17
7.6 4.14375645746097e-17
7.7 4.14375645746095e-17
7.8 4.14375645746093e-17
7.9 4.14375645746091e-17
8 4.14375645746089e-17
8.1 4.14375645746087e-17
8.2 4.14375645746085e-17
8.3 4.14375645746083e-17
8.4 4.14375645746081e-17
8.5 4.14375645746079e-17
8.6 4.14375645746076e-17
8.7 4.14375645746074e-17
8.8 4.14375645746072e-17
8.9 4.1437564574607e-17
9 4.14375645746068e-17
9.1 4.14375645746066e-17
9.2 4.14375645746064e-17
9.3 4.14375645746062e-17
9.4 4.1437564574606e-17
9.5 4.14375645746058e-17
9.6 4.14375645746056e-17
9.7 4.14375645746055e-17
9.8 4.14375645746053e-17
9.9 4.14375645746051e-17
10 4.14375645746049e-17
};
\addplot [thick, color6]
table {%
0 1.44162517251417
0.1 7.7715611723761e-15
0.2 9.20939187701292e-15
0.3 8.76024453517275e-15
0.4 8.33300236767807e-15
0.5 7.92659704656966e-15
0.6 7.54001234685765e-15
0.7 7.17228160542979e-15
0.8 6.82248530389014e-15
0.9 6.48974876928408e-15
1 6.17323998696039e-15
1.1 5.87216752010119e-15
1.2 5.58577853071769e-15
1.3 5.31335689716308e-15
1.4 5.05422142345538e-15
1.5 4.80772413593267e-15
1.6 4.57324866298146e-15
1.7 4.35020869378654e-15
1.8 4.13804651224859e-15
1.9 3.93623160240343e-15
2 3.74425932185575e-15
2.1 3.5616496399103e-15
2.2 3.38794593724505e-15
2.3 3.22271386412515e-15
2.4 3.06554025430225e-15
2.5 2.9160320918837e-15
2.6 2.77381552858815e-15
2.7 2.63853494893006e-15
2.8 2.50985208099576e-15
2.9 2.38744515058752e-15
3 2.27100807662038e-15
3.1 2.16024970576008e-15
3.2 2.054893084388e-15
3.3 1.95467476607289e-15
3.4 1.85934415281758e-15
3.5 1.76866286843344e-15
3.6 1.68240416247572e-15
3.7 1.60035234324938e-15
3.8 1.52230223846747e-15
3.9 1.44805868221356e-15
4 1.37743602692527e-15
4.1 1.31025767917867e-15
4.2 1.24635565811276e-15
4.3 1.18557017538981e-15
4.4 1.12774923564126e-15
4.5 1.07274825640015e-15
4.6 1.02042970656966e-15
4.7 9.70662762523687e-16
4.8 9.23322980979679e-16
4.9 8.78291987825582e-16
5 8.35457182122914e-16
5.1 7.94711454545766e-16
5.2 7.55952919551693e-16
5.3 7.19084660614789e-16
5.4 6.84014487883896e-16
5.5 6.50654707659948e-16
5.6 6.18921903116051e-16
5.7 5.88736725711967e-16
5.8 5.60023696781428e-16
5.9 5.32711018796159e-16
6 5.06730395834658e-16
6.1 4.8201686280682e-16
6.2 4.5850862300737e-16
6.3 4.36146893591915e-16
6.4 4.1487575858921e-16
6.5 3.94642029082111e-16
6.6 3.7539511020757e-16
6.7 3.57086874643129e-16
6.8 3.39671542263541e-16
6.9 3.23105565666618e-16
7 3.07347521282035e-16
7.1 2.9235800579083e-16
7.2 2.78099537596588e-16
7.3 2.64536463101917e-16
7.4 2.51634867555891e-16
7.5 2.39362490249503e-16
7.6 2.27688643847093e-16
7.7 2.16584137652018e-16
7.8 2.06021204614713e-16
7.9 1.95973431900597e-16
8 1.86415694844235e-16
8.1 1.77324094124583e-16
8.2 1.68675896004238e-16
8.3 1.60449475483254e-16
8.4 1.52624262225378e-16
8.5 1.45180689121492e-16
8.6 1.38100143361655e-16
8.7 1.31364919893374e-16
8.8 1.24958177149757e-16
8.9 1.18863894936822e-16
9 1.13066834374667e-16
9.1 1.07552499792333e-16
9.2 1.02307102481074e-16
9.3 9.73175262154087e-17
9.4 9.25712944557171e-17
9.5 8.80565391503986e-17
9.6 8.3761971059559e-17
9.7 7.96768515260306e-17
9.8 7.57909656231357e-17
9.9 7.20945966120496e-17
10 6.85785016448919e-17
};
\addplot [thick, color7]
table {%
0 -1.03485722757579
0.1 -2.77555756156289e-15
0.2 -6.97758201540764e-15
0.3 -6.31357729495438e-15
0.4 -5.71276097813592e-15
0.5 -5.16911989331215e-15
0.6 -4.67721309778206e-15
0.7 -4.23211742300035e-15
0.8 -3.82937820185199e-15
0.9 -3.46496468484636e-15
1 -3.13522969902165e-15
1.1 -2.83687314581195e-15
1.2 -2.56690897255162e-15
1.3 -2.32263528705658e-15
1.4 -2.10160731617916e-15
1.5 -1.90161293769673e-15
1.6 -1.72065054064902e-15
1.7 -1.5569089925428e-15
1.8 -1.40874951292918e-15
1.9 -1.27468927193806e-15
2 -1.15338654961839e-15
2.1 -1.04362730755395e-15
2.2 -9.44313038358785e-16
2.3 -8.54449771446122e-16
2.4 -7.73138125036602e-16
2.5 -6.99564304843171e-16
2.6 -6.32991959344314e-16
2.7 -5.72754810130538e-16
2.8 -5.18249983566107e-16
2.9 -4.68931977027056e-16
3 -4.24307199327586e-16
3.1 -3.83929030693576e-16
3.2 -3.47393352841762e-16
3.3 -3.14334504428138e-16
3.4 -2.84421621386317e-16
3.5 -2.57355325528748e-16
3.6 -2.32864728269191e-16
3.7 -2.10704719478698e-16
3.8 -1.90653514341058e-16
3.9 -1.72510433655809e-16
4 -1.56093895373352e-16
4.1 -1.41239597260768e-16
4.2 -1.27798872509844e-16
4.3 -1.15637201829687e-16
4.4 -1.04632867132452e-16
4.5 -9.46757333378045e-17
4.6 -8.56661461040185e-17
4.7 -7.75139344538313e-17
4.8 -7.01375083129946e-17
4.9 -6.34630419293883e-17
5 -5.74237350000794e-17
5.1 -5.19591441114379e-17
5.2 -4.70145778011378e-17
5.3 -4.25405491876192e-17
5.4 -3.8492280688745e-17
5.5 -3.48292558727082e-17
5.6 -3.15148139559647e-17
5.7 -2.8515782889789e-17
5.8 -2.58021473632617e-17
5.9 -2.33467483999487e-17
6 -2.11250115417371e-17
6.1 -1.91147008993982e-17
6.2 -1.72956966083345e-17
6.3 -1.56497934622127e-17
6.4 -1.4160518709139e-17
6.5 -1.28129671868222e-17
6.6 -1.1593652146699e-17
6.7 -1.0490370274022e-17
6.8 -9.49207955298341e-18
6.9 -8.58878875451e-18
7 -7.77145744068401e-18
7.1 -7.03190548500208e-18
7.2 -6.3627312029197e-18
7.3 -5.75723727330453e-18
7.4 -5.20936370939537e-18
7.5 -4.7136272084179e-18
7.6 -4.26506627284752e-18
7.7 -3.85919155407446e-18
7.8 -3.49194092149399e-18
7.9 -3.15963880733802e-18
8 -2.85895942035744e-18
8.1 -2.58689346018548e-18
8.2 -2.34071799924821e-18
8.3 -2.11796923079006e-18
8.4 -1.91641781026789e-18
8.5 -1.73404654332125e-18
8.6 -1.56903019701348e-18
8.7 -1.41971723228675e-18
8.8 -1.28461327480426e-18
8.9 -1.16236615874948e-18
9 -1.05175239389625e-18
9.1 -9.51664920507326e-19
9.2 -8.61102029508454e-19
9.3 -7.79157337047273e-19
9.4 -7.05010713099037e-19
9.5 -6.37920073329725e-19
9.6 -5.77213952066566e-19
9.7 -5.22284782043898e-19
9.8 -4.72582813665809e-19
9.9 -4.27610612927359e-19
10 -3.86918082927859e-19
};
\addplot [thick, color8]
table {%
0 0.755850732791848
0.1 4.44089209850063e-15
0.2 6.1997493380301e-15
0.3 5.33617370707874e-15
0.4 4.59288727327258e-15
0.5 3.95313471092706e-15
0.6 3.40249457757787e-15
0.7 2.92855422266444e-15
0.8 2.52062997884072e-15
0.9 2.16952632840446e-15
1 1.867328615922e-15
1.1 1.60722463433092e-15
1.2 1.38335106267556e-15
1.3 1.19066129384102e-15
1.4 1.02481167282961e-15
1.5 8.82063581138017e-16
1.6 7.59199159999599e-16
1.7 6.53448772706909e-16
1.8 5.62428570854047e-16
1.9 4.84086757103466e-16
2 4.1665733312073e-16
2.1 3.58620290053039e-16
2.2 3.08667344156549e-16
2.3 2.65672445177508e-16
2.4 2.28666392680656e-16
2.5 1.96814988120609e-16
2.6 1.69400230155433e-16
2.7 1.45804129303041e-16
2.8 1.25494777086857e-16
2.9 1.08014355638359e-16
3 9.29688174663642e-17
3.1 8.00190027521195e-17
3.2 6.88729939343414e-17
3.3 5.92795352395747e-17
3.4 5.1022368819484e-17
3.5 4.3915359819042e-17
3.6 3.77983004838361e-17
3.7 3.25332987217584e-17
3.8 2.80016697092455e-17
3.9 2.41012604719759e-17
4 2.07441471301355e-17
4.1 1.78546528990501e-17
4.2 1.53676421665202e-17
4.3 1.32270521915766e-17
4.4 1.13846293258869e-17
4.5 9.79884126943937e-18
4.6 8.43394084033514e-18
4.7 7.25916015397889e-18
4.8 6.24801704668164e-18
4.9 5.37771810892343e-18
5 4.62864487132045e-18
5.1 3.98391156078887e-18
5.2 3.42898445774714e-18
5.3 2.95135427382453e-18
5.4 2.54025416474078e-18
5.5 2.18641702174301e-18
5.6 1.8818665704092e-18
5.7 1.61973756772219e-18
5.8 1.39412104425707e-18
5.9 1.19993110289722e-18
6 1.03279027142715e-18
6.1 8.88930824594439e-19
6.2 7.6510985122133e-19
6.3 6.58536151790281e-19
6.4 5.66807318612773e-19
6.5 4.87855580228619e-19
6.6 4.19901189248696e-19
6.7 3.61412302899047e-19
6.8 3.11070451885661e-19
6.9 2.67740819170324e-19
7 2.30446658676639e-19
7.1 1.98347277265565e-19
7.2 1.70719083646703e-19
7.3 1.46939277024668e-19
7.4 1.26471807787186e-19
7.5 1.08855293757195e-19
7.6 9.36926196147752e-20
7.7 8.06419850363845e-20
7.8 6.94091997573129e-20
7.9 5.97410518700941e-20
8 5.14195998662917e-20
8.1 4.42572597511361e-20
8.2 3.8092576484115e-20
8.3 3.2786584424005e-20
8.4 2.82196747350325e-20
8.5 2.42888991380642e-20
8.6 2.09056492282127e-20
8.7 1.79936590445308e-20
8.8 1.54872858661266e-20
8.9 1.33300304795675e-20
9 1.1473263561149e-20
9.1 9.87512946407089e-21
9.2 8.49960269931819e-21
9.3 7.31567584110103e-21
9.4 6.29666055070235e-21
9.5 5.41958596214251e-21
9.6 4.66468086782899e-21
9.7 4.0149280316878e-21
9.8 3.45568058289933e-21
9.9 2.97433184292431e-21
10 2.56003114297261e-21
};
\addplot [thick, color9]
table {%
0 -0.791302798759913
0.1 -3.1641356201817e-15
0.2 -1.52548734889398e-15
0.3 -1.24896340597116e-15
0.4 -1.02256474993785e-15
0.5 -8.37205207787824e-16
0.6 -6.85445650253121e-16
0.7 -5.61195433425935e-16
0.8 -4.59467959832894e-16
0.9 -3.76180548769327e-16
1 -3.07990583987325e-16
1.1 -2.52161362768984e-16
1.2 -2.06452262437123e-16
1.3 -1.69028816299893e-16
1.4 -1.38389090061177e-16
1.5 -1.13303403923643e-16
1.6 -9.27649812207749e-17
1.7 -7.59495429342161e-17
1.8 -6.21822264825201e-17
1.9 -5.09105011161549e-17
2 -4.16819929184585e-17
2.1 -3.41263294519686e-17
2.2 -2.79402754120415e-17
2.3 -2.28755627293479e-17
2.4 -1.87289267005198e-17
2.5 -1.53339482618941e-17
2.6 -1.25543750081522e-17
2.7 -1.02786529048783e-17
2.8 -8.41544923346689e-18
2.9 -6.88998708843253e-18
3 -5.64104431763466e-18
3.1 -4.6184964623465e-18
3.2 -3.78130508672682e-18
3.3 -3.09587076129396e-18
3.4 -2.53468459984538e-18
3.5 -2.07522423126456e-18
3.6 -1.69904989768566e-18
3.7 -1.39106440226522e-18
3.8 -1.13890720566141e-18
3.9 -9.32458354191061e-19
4 -7.63432330553969e-19
4.1 -6.2504552693086e-19
4.2 -5.117439949838e-19
4.3 -4.18980546407145e-19
4.4 -3.4303225829528e-19
4.5 -2.80851059173853e-19
4.6 -2.29941399189267e-19
4.7 -1.88260094930685e-19
4.8 -1.54134329305314e-19
4.9 -1.2619451551497e-19
5 -1.03319330729112e-19
5.1 -8.45907134621521e-20
5.2 -6.92570185427375e-20
5.3 -5.67028509534533e-20
5.4 -4.64243678685235e-20
5.5 -3.80090576715527e-20
5.6 -3.11191844163138e-20
5.7 -2.54782332971528e-20
5.8 -2.08598131389959e-20
5.9 -1.70785705245857e-20
6 -1.39827509110623e-20
6.1 -1.1448108187276e-20
6.2 -9.37291824103696e-21
6.3 -7.67389641337178e-21
6.4 -6.28285499266108e-21
6.5 -5.14396660254144e-21
6.6 -4.21152365306837e-21
6.7 -3.44810393468529e-21
6.8 -2.82306873353776e-21
6.9 -2.31133319249749e-21
7 -1.89235956745752e-21
7.1 -1.54933297575402e-21
7.2 -1.26848655586047e-21
7.3 -1.03854895489809e-21
7.4 -8.50291969584694e-22
7.5 -6.96160186108464e-22
7.6 -5.69967754861208e-22
7.7 -4.66650130480352e-22
7.8 -3.82060813962741e-22
7.9 -3.12804939080142e-22
8 -2.56103024381629e-22
8.1 -2.09679422999558e-22
8.2 -1.71670992789769e-22
8.3 -1.40552322027661e-22
8.4 -1.15074509202196e-22
8.5 -9.42150402560075e-23
8.6 -7.7136751476305e-23
8.7 -6.31542311747815e-23
8.8 -5.17063117480842e-23
8.9 -4.23335480009565e-23
9 -3.46597780210091e-23
9.1 -2.83770265079908e-23
9.2 -2.32331445937474e-23
9.3 -1.90216902297589e-23
9.4 -1.5573642985971e-23
9.5 -1.27506206266261e-23
9.6 -1.04393253723354e-23
9.7 -8.54699682878356e-24
9.8 -6.99768923039828e-24
9.9 -5.72922341123506e-24
10 -4.69069140374086e-24
};
\end{axis}

\end{tikzpicture}

%% file: txz-tdmd.tex
\begin{tikzpicture}

\definecolor{color0}{rgb}{0.254627,0.013882,0.615419}
\definecolor{color1}{rgb}{0.399411,0.000859,0.656133}
\definecolor{color2}{rgb}{0.534952,0.031217,0.650165}
\definecolor{color3}{rgb}{0.650746,0.125309,0.595617}
\definecolor{color4}{rgb}{0.752312,0.227133,0.513149}
\definecolor{color5}{rgb}{0.836801,0.329105,0.430905}
\definecolor{color6}{rgb}{0.907365,0.434524,0.35297}
\definecolor{color7}{rgb}{0.959424,0.543431,0.278701}
\definecolor{color8}{rgb}{0.990681,0.669558,0.201642}
\definecolor{color9}{rgb}{0.988648,0.809579,0.145357}

\begin{axis}[
height=\figureheight,
tick align=outside,
tick pos=left,
width=.95\linewidth,
x grid style={white!69.0196078431373!black},
xlabel={time \(\displaystyle t\)},
xmin=0, xmax=10,
xtick style={color=black},
y grid style={white!69.0196078431373!black},
ymin=-2.46869853867125, ymax=1.13059137490536,
ytick style={color=black}
]
\addplot [thick, color0]
table {%
0 -0.0512243212074242
0.1 -0.0512243212078897
0.2 -0.0512243212079134
0.3 -0.0512243212079293
0.4 -0.0512243212079387
0.5 -0.0512243212079428
0.6 -0.0512243212079441
0.7 -0.0512243212079425
0.8 -0.0512243212079387
0.9 -0.0512243212079327
1 -0.0512243212079267
1.1 -0.0512243212079193
1.2 -0.0512243212079106
1.3 -0.0512243212079047
1.4 -0.0512243212078955
1.5 -0.0512243212078876
1.6 -0.0512243212078808
1.7 -0.051224321207873
1.8 -0.0512243212078645
1.9 -0.0512243212078589
2 -0.0512243212078525
2.1 -0.0512243212078465
2.2 -0.0512243212078406
2.3 -0.0512243212078362
2.4 -0.0512243212078293
2.5 -0.0512243212078254
2.6 -0.0512243212078215
2.7 -0.0512243212078162
2.8 -0.0512243212078141
2.9 -0.0512243212078097
3 -0.0512243212078063
3.1 -0.0512243212078029
3.2 -0.0512243212078014
3.3 -0.0512243212077983
3.4 -0.0512243212077961
3.5 -0.0512243212077936
3.6 -0.051224321207793
3.7 -0.0512243212077892
3.8 -0.0512243212077891
3.9 -0.0512243212077864
4 -0.0512243212077858
4.1 -0.0512243212077827
4.2 -0.0512243212077832
4.3 -0.0512243212077819
4.4 -0.0512243212077813
4.5 -0.0512243212077804
4.6 -0.0512243212077782
4.7 -0.051224321207777
4.8 -0.0512243212077772
4.9 -0.051224321207778
5 -0.0512243212077766
5.1 -0.0512243212077751
5.2 -0.0512243212077763
5.3 -0.0512243212077756
5.4 -0.0512243212077741
5.5 -0.0512243212077741
5.6 -0.0512243212077745
5.7 -0.0512243212077749
5.8 -0.0512243212077741
5.9 -0.0512243212077768
6 -0.0512243212077752
6.1 -0.0512243212077748
6.2 -0.0512243212077757
6.3 -0.0512243212077761
6.4 -0.0512243212077744
6.5 -0.0512243212077759
6.6 -0.0512243212077742
6.7 -0.0512243212077743
6.8 -0.0512243212077746
6.9 -0.0512243212077759
7 -0.0512243212077751
7.1 -0.0512243212077757
7.2 -0.051224321207775
7.3 -0.0512243212077768
7.4 -0.051224321207777
7.5 -0.0512243212077758
7.6 -0.0512243212077755
7.7 -0.0512243212077779
7.8 -0.0512243212077776
7.9 -0.0512243212077773
8 -0.0512243212077771
8.1 -0.0512243212077785
8.2 -0.0512243212077783
8.3 -0.0512243212077789
8.4 -0.0512243212077805
8.5 -0.0512243212077806
8.6 -0.0512243212077804
8.7 -0.0512243212077815
8.8 -0.0512243212077804
8.9 -0.051224321207781
9 -0.051224321207783
9.1 -0.0512243212077828
9.2 -0.051224321207783
9.3 -0.051224321207783
9.4 -0.0512243212077835
9.5 -0.051224321207784
9.6 -0.0512243212077833
9.7 -0.0512243212077853
9.8 -0.0512243212077856
9.9 -0.0512243212077858
10 -0.0512243212077856
};
\addplot [thick, color1]
table {%
0 -0.734977255192508
0.1 -0.774079180223
0.2 -0.811274081866954
0.3 -0.846654966752084
0.4 -0.880310305519685
0.5 -0.912324254046962
0.6 -0.942776863880552
0.7 -0.971744282407099
0.8 -0.999298943261374
0.9 -1.0255097474481
1 -1.05044223563033
1.1 -1.07415875201528
1.2 -1.0967186002473
1.3 -1.11817819169787
1.4 -1.13859118652341
1.5 -1.15800862784365
1.6 -1.17647906937597
1.7 -1.19404869684503
1.8 -1.21076144347113
1.9 -1.22665909982609
2 -1.24178141833153
2.1 -1.25616621266058
2.2 -1.26984945229176
2.3 -1.28286535245143
2.4 -1.29524645966968
2.5 -1.30702373316357
2.6 -1.31822662225135
2.7 -1.32888313999107
2.8 -1.33901993322781
2.9 -1.34866234922467
3 -1.35783449904416
3.1 -1.36655931783839
3.2 -1.3748586221989
3.3 -1.38275316470951
3.4 -1.39026268583857
3.5 -1.39740596330044
3.6 -1.40420085900954
3.7 -1.41066436374446
3.8 -1.41681263963371
3.9 -1.42266106056952
4 -1.42822425065052
4.1 -1.43351612074967
4.2 -1.43854990329861
4.3 -1.4433381853757
4.4 -1.44789294018024
4.5 -1.45222555697171
4.6 -1.45634686954884
4.7 -1.46026718333977
4.8 -1.46399630117098
4.9 -1.46754354777945
5 -1.4709177931294
5.1 -1.47412747459176
5.2 -1.47718061804202
5.3 -1.48008485792914
5.4 -1.48284745636557
5.5 -1.48547532128639
5.6 -1.48797502372268
5.7 -1.49035281423258
5.8 -1.4926146385309
5.9 -1.4947661523565
6 -1.49681273561465
6.1 -1.49875950582948
6.2 -1.50061133094057
6.3 -1.50237284147528
6.4 -1.50404844212745
6.5 -1.50564232277152
6.6 -1.5071584689393
6.7 -1.50860067178593
6.8 -1.50997253756975
6.9 -1.51127749666978
7 -1.51251881216351
7.1 -1.51369958798622
7.2 -1.51482277669254
7.3 -1.51589118683924
7.4 -1.51690749000823
7.5 -1.51787422748678
7.6 -1.51879381662215
7.7 -1.51966855686617
7.8 -1.52050063552507
7.9 -1.52129213322892
8 -1.52204502913424
8.1 -1.52276120587297
8.2 -1.52344245426
8.3 -1.52409047777113
8.4 -1.52470689680269
8.5 -1.52529325272333
8.6 -1.52585101172828
8.7 -1.52638156850556
8.8 -1.52688624972349
8.9 -1.52736631734797
9 -1.52782297179813
9.1 -1.52825735494795
9.2 -1.52867055298156
9.3 -1.52906359910929
9.4 -1.52943747615116
9.5 -1.52979311899454
9.6 -1.53013141693178
9.7 -1.53045321588392
9.8 -1.53075932051598
9.9 -1.53105049624897
10 -1.53132747117389
};
\addplot [thick, color2]
table {%
0 0.966987287924606
0.1 0.70435392758802
0.2 0.466713435931219
0.3 0.251687427039696
0.4 0.0571238483437241
0.5 -0.118924557847364
0.6 -0.278219743154646
0.7 -0.422355987333639
0.8 -0.552775854361952
0.9 -0.670784630104437
1 -0.777563386052847
1.1 -0.874180799886293
1.2 -0.961603851156655
1.3 -1.04070749914495
1.4 -1.1122834397479
1.5 -1.17704802903657
1.6 -1.23564945278869
1.7 -1.28867421374978
1.8 -1.33665300154978
1.9 -1.38006600402323
2 -1.4193477130905
2.1 -1.45489127329896
2.2 -1.48705241654579
2.3 -1.51615302236234
2.4 -1.54248433939266
2.5 -1.56630990030786
2.6 -1.58786815932963
2.7 -1.60737487876024
2.8 -1.62502528840419
2.9 -1.64099603949369
3 -1.65544697267361
3.1 -1.66852271774034
3.2 -1.68035414114542
3.3 -1.69105965575097
3.4 -1.70074640594539
3.5 -1.70951133998047
3.6 -1.71744218026203
3.7 -1.72461830130526
3.8 -1.73111152414152
3.9 -1.73698683512742
4 -1.74230303635006
4.1 -1.74711333413811
4.2 -1.75146587156864
4.3 -1.75540421029918
4.4 -1.75896776654748
4.5 -1.76219220558222
4.6 -1.76510979867303
4.7 -1.76774974607219
4.8 -1.77013846926061
4.9 -1.77229987538281
5 -1.77425559651776
5.1 -1.7760252061799
5.2 -1.77762641521753
5.3 -1.77907524906888
5.4 -1.78038620815009
5.5 -1.78157241298029
5.6 -1.78264573549611
5.7 -1.78361691787004
5.8 -1.78449568002172
5.9 -1.78529081689811
6 -1.78601028649633
6.1 -1.78666128950994
6.2 -1.78725034139591
6.3 -1.7877833375835
6.4 -1.78826561247771
6.5 -1.78870199284776
6.6 -1.78909684613509
6.7 -1.7894541241641
6.8 -1.78977740269339
6.9 -1.79006991720314
7 -1.79033459527688
7.1 -1.79057408590174
7.2 -1.79079078598038
7.3 -1.79098686432003
7.4 -1.79116428333861
7.5 -1.79132481870529
7.6 -1.79147007711199
7.7 -1.79160151235365
7.8 -1.79172043987836
7.9 -1.79182804995275
8 -1.79192541957462
8.1 -1.79201352325187
8.2 -1.79209324275571
8.3 -1.79216537594574
8.4 -1.79223064475516
8.5 -1.79228970241615
8.6 -1.79234313999764
8.7 -1.79239149232091
8.8 -1.79243524331224
8.9 -1.79247483084629
9 -1.79251065112838
9.1 -1.79254306265994
9.2 -1.79257238982647
9.3 -1.79259892614412
9.4 -1.79262293719727
9.5 -1.7926446632966
9.6 -1.79266432188423
9.7 -1.7926821097099
9.8 -1.79269820480015
9.9 -1.79271276824007
10 -1.79272594578543
};
\addplot [thick, color3]
table {%
0 -1.31689588267591
0.1 -1.4188683646298
0.2 -1.50663689322357
0.3 -1.58217996586335
0.4 -1.6472004910481
0.5 -1.70316417570598
0.6 -1.75133256548119
0.7 -1.79279148277227
0.8 -1.82847550357868
0.9 -1.85918902491769
1 -1.88562439771829
1.1 -1.90837753394755
1.2 -1.92796133978878
1.3 -1.94481727768509
1.4 -1.95932531788259
1.5 -1.97181250380288
1.6 -1.98256032432758
1.7 -1.99181105918239
1.8 -1.99977324045973
1.9 -2.00662635339488
2 -2.01252488236152
2.1 -2.01760179329228
2.2 -2.021971531026
2.3 -2.0257325991483
2.4 -2.02896978048105
2.5 -2.03175604827529
2.6 -2.03415421119025
2.7 -2.03621832913993
2.8 -2.0379949319235
2.9 -2.03952406811027
3 -2.04084020782327
3.1 -2.04197301977234
3.2 -2.0429480400527
3.3 -2.04378724778518
3.4 -2.04450956057441
3.5 -2.04513126095358
3.6 -2.04566636342889
3.7 -2.04612693039759
3.8 -2.04652334406123
3.9 -2.0468645404635
4 -2.04715821092846
4.1 -2.04741097544009
4.2 -2.04762853187142
4.3 -2.04781578442719
4.4 -2.04797695419554
4.5 -2.04811567430073
4.6 -2.04823507180176
4.7 -2.04833783818327
4.8 -2.04842629002754
4.9 -2.04850242123544
5 -2.04856794797333
5.1 -2.04862434735931
5.2 -2.0486728907607
5.3 -2.04871467245347
5.4 -2.04875063428971
5.5 -2.04878158692902
5.6 -2.04880822811256
5.7 -2.04883115839174
5.8 -2.04885089466593
5.9 -2.04886788183456
6 -2.0488825028261
6.1 -2.04889508723014
6.2 -2.04890591872708
6.3 -2.0489152414829
6.4 -2.04892326565319
6.5 -2.04893017212058
6.6 -2.04893611657214
6.7 -2.04894123300903
6.8 -2.04894563676706
6.9 -2.04894942711673
7 -2.04895268950094
7.1 -2.04895549746104
7.2 -2.0489579142947
7.3 -2.04895999448272
7.4 -2.04896178491714
7.5 -2.04896332595832
7.6 -2.04896465234477
7.7 -2.04896579397617
7.8 -2.04896677658742
7.9 -2.04896762232876
8 -2.04896835026509
8.1 -2.04896897680569
8.2 -2.04896951607419
8.3 -2.04896998022689
8.4 -2.04897037972682
8.5 -2.0489707235796
8.6 -2.04897101953644
8.7 -2.04897127426885
8.8 -2.04897149351907
8.9 -2.04897168222949
9 -2.04897184465405
9.1 -2.04897198445417
9.2 -2.04897210478125
9.3 -2.04897220834773
9.4 -2.04897229748823
9.5 -2.04897237421218
9.6 -2.04897244024909
9.7 -2.04897249708759
9.8 -2.04897254600894
9.9 -2.04897258811594
10 -2.04897262435777
};
\addplot [thick, color4]
table {%
0 -1.02005922093539
0.1 -1.25299658996221
0.2 -1.44370957752548
0.3 -1.59985216545493
0.4 -1.72769090405798
0.5 -1.83235641078701
0.6 -1.91804927993257
0.7 -1.98820866722155
0.8 -2.04565031521216
0.9 -2.09267955892957
1 -2.13118384705502
1.1 -2.16270849186872
1.2 -2.18851868805757
1.3 -2.20965028942037
1.4 -2.22695138131789
1.5 -2.24111631731622
1.6 -2.25271358603345
1.7 -2.26220862658397
1.8 -2.26998250828441
1.9 -2.27634722430335
2 -2.28155821304268
2.1 -2.28582460977752
2.2 -2.28931763998917
2.3 -2.29217749124489
2.4 -2.29451893941718
2.5 -2.29643595504258
2.6 -2.29800547468924
2.7 -2.29929048869152
2.8 -2.30034256917333
2.9 -2.3012039398185
3 -2.30190917045552
3.1 -2.30248656446606
3.2 -2.30295929469913
3.3 -2.30334633347887
3.4 -2.30366321403048
3.5 -2.30392265388314
3.6 -2.30413506526909
3.7 -2.30430897300308
3.8 -2.3044513566131
3.9 -2.30456793045336
4 -2.30466337304139
4.1 -2.30474151482337
4.2 -2.30480549190338
4.3 -2.30485787190628
4.4 -2.30490075702551
4.5 -2.30493586839148
4.6 -2.30496461514658
4.7 -2.30498815099904
4.8 -2.30500742052525
4.9 -2.30502319707896
5 -2.30503611382867
5.1 -2.30504668916889
5.2 -2.30505534752515
5.3 -2.3050624363877
5.4 -2.30506824025748
5.5 -2.30507299206416
5.6 -2.30507688251443
5.7 -2.3050800677457
5.8 -2.30508267559251
5.9 -2.3050848107169
6 -2.3050865588089
6.1 -2.30508799002558
6.2 -2.3050891618067
6.3 -2.30509012117994
6.4 -2.30509090664832
6.5 -2.30509154973544
6.6 -2.30509207625065
6.7 -2.30509250732485
6.8 -2.30509286025856
6.9 -2.30509314921624
7 -2.30509338579478
7.1 -2.30509357948892
7.2 -2.30509373807227
7.3 -2.30509386790934
7.4 -2.30509397421094
7.5 -2.30509406124334
7.6 -2.30509413249944
7.7 -2.30509419083901
7.8 -2.30509423860341
7.9 -2.3050942777096
8 -2.30509430972705
8.1 -2.30509433594072
8.2 -2.30509435740266
8.3 -2.30509437497421
8.4 -2.30509438936059
8.5 -2.30509440113916
8.6 -2.30509441078265
8.7 -2.30509441867807
8.8 -2.3050944251423
8.9 -2.30509443043476
9 -2.30509443476787
9.1 -2.30509443831552
9.2 -2.3050944412201
9.3 -2.30509444359817
9.4 -2.30509444554517
9.5 -2.30509444713925
9.6 -2.30509444844437
9.7 -2.30509444951292
9.8 -2.30509445038778
9.9 -2.30509445110405
10 -2.30509445169049
};
\addplot [thick, color5]
table {%
0 -0.30734592724463
0.1 -0.30734592724462
0.2 -0.30734592724463
0.3 -0.307345927244641
0.4 -0.307345927244651
0.5 -0.307345927244661
0.6 -0.307345927244671
0.7 -0.30734592724468
0.8 -0.30734592724469
0.9 -0.307345927244699
1 -0.307345927244707
1.1 -0.307345927244715
1.2 -0.307345927244723
1.3 -0.307345927244731
1.4 -0.307345927244738
1.5 -0.307345927244745
1.6 -0.307345927244752
1.7 -0.307345927244759
1.8 -0.307345927244765
1.9 -0.307345927244771
2 -0.307345927244777
2.1 -0.307345927244783
2.2 -0.307345927244788
2.3 -0.307345927244793
2.4 -0.307345927244798
2.5 -0.307345927244803
2.6 -0.307345927244808
2.7 -0.307345927244813
2.8 -0.307345927244817
2.9 -0.307345927244822
3 -0.307345927244826
3.1 -0.30734592724483
3.2 -0.307345927244834
3.3 -0.307345927244838
3.4 -0.307345927244842
3.5 -0.307345927244846
3.6 -0.30734592724485
3.7 -0.307345927244853
3.8 -0.307345927244857
3.9 -0.307345927244861
4 -0.307345927244864
4.1 -0.307345927244868
4.2 -0.307345927244871
4.3 -0.307345927244875
4.4 -0.307345927244878
4.5 -0.307345927244882
4.6 -0.307345927244885
4.7 -0.307345927244888
4.8 -0.307345927244892
4.9 -0.307345927244895
5 -0.307345927244898
5.1 -0.307345927244901
5.2 -0.307345927244905
5.3 -0.307345927244908
5.4 -0.307345927244911
5.5 -0.307345927244914
5.6 -0.307345927244917
5.7 -0.30734592724492
5.8 -0.307345927244923
5.9 -0.307345927244926
6 -0.307345927244929
6.1 -0.307345927244932
6.2 -0.307345927244935
6.3 -0.307345927244939
6.4 -0.307345927244942
6.5 -0.307345927244945
6.6 -0.307345927244948
6.7 -0.307345927244951
6.8 -0.307345927244953
6.9 -0.307345927244957
7 -0.30734592724496
7.1 -0.307345927244963
7.2 -0.307345927244965
7.3 -0.307345927244968
7.4 -0.307345927244971
7.5 -0.307345927244974
7.6 -0.307345927244977
7.7 -0.30734592724498
7.8 -0.307345927244983
7.9 -0.307345927244986
8 -0.307345927244989
8.1 -0.307345927244992
8.2 -0.307345927244995
8.3 -0.307345927244998
8.4 -0.307345927245001
8.5 -0.307345927245004
8.6 -0.307345927245007
8.7 -0.30734592724501
8.8 -0.307345927245012
8.9 -0.307345927245015
9 -0.307345927245018
9.1 -0.307345927245021
9.2 -0.307345927245024
9.3 -0.307345927245027
9.4 -0.30734592724503
9.5 -0.307345927245033
9.6 -0.307345927245036
9.7 -0.307345927245039
9.8 -0.307345927245041
9.9 -0.307345927245044
10 -0.307345927245047
};
\addplot [thick, color6]
table {%
0 -0.200438095257661
0.1 -0.190662613999948
0.2 -0.181363888588972
0.3 -0.172518667367702
0.4 -0.164104832675816
0.5 -0.156101345544011
0.6 -0.148488193085628
0.7 -0.141246338454005
0.8 -0.134357673240451
0.9 -0.127804972193785
1 -0.12157185014824
1.1 -0.115642721052014
1.2 -0.110002758994022
1.3 -0.104637861131393
1.4 -0.0995346124250191
1.5 -0.0946802520949716
1.6 -0.0900626417119016
1.7 -0.0856702348446459
1.8 -0.0814920481881332
1.9 -0.0775176340994022
2 -0.0737370544730512
2.1 -0.0701408558907984
2.2 -0.0667200459830116
2.3 -0.0634660709431014
2.4 -0.0603707941385479
2.5 -0.0574264757650822
2.6 -0.0546257534931434
2.7 -0.0519616240582206
2.8 -0.0494274257490438
2.9 -0.0470168217498348
3 -0.0447237842949681
3.1 -0.0425425795964165
3.2 -0.0404677535062951
3.3 -0.0384941178786498
3.4 -0.0366167375963907
3.5 -0.0348309182309282
3.6 -0.0331321943036571
3.7 -0.0315163181199334
3.8 -0.0299792491476254
3.9 -0.0285171439136789
4 -0.0271263463934324
4.1 -0.0258033788686513
4.2 -0.0245449332314206
4.3 -0.0233478627121522
4.4 -0.0222091740110218
4.5 -0.0211260198131599
4.6 -0.0200956916688824
4.7 -0.0191156132211541
4.8 -0.0181833337633559
4.9 -0.0172965221112408
5 -0.0164529607737588
5.1 -0.0156505404081746
5.2 -0.0148872545456122
5.3 -0.0141611945738378
5.4 -0.0134705449647335
5.5 -0.0128135787345335
5.6 -0.0121886531254639
5.7 -0.0115942054979929
5.8 -0.0110287494234179
5.9 -0.0104908709670195
6 -0.00997922515248845
6.1 -0.00949253259878367
6.2 -0.00902957632101409
6.3 -0.00858919868734253
6.4 -0.00817029852430204
6.5 -0.00777182836328986
6.6 -0.00739279182134939
6.7 -0.00703224110969466
6.8 -0.00668927466374407
6.9 -0.00636303488873936
7 -0.00605270601531244
7.1 -0.00575751205963643
7.2 -0.00547671488306282
7.3 -0.00520961234638903
7.4 -0.00495553655414631
7.5 -0.00471385218451167
7.6 -0.00448395490067325
7.7 -0.00426526983967342
7.8 -0.00405725017495148
7.9 -0.00385937574899341
8 -0.00367115177266583
8.1 -0.00349210758798635
8.2 -0.00332179549123367
8.3 -0.00315978961345401
8.4 -0.00300568485556794
8.5 -0.00285909587541111
8.6 -0.00271965612417848
8.7 -0.00258701692986082
8.8 -0.0024608466253839
8.9 -0.00234082971926714
9 -0.00222666610673138
9.1 -0.00211807031927999
9.2 -0.00201477081087951
9.3 -0.00191650927895243
9.4 -0.00182304001848699
9.5 -0.00173412930764581
9.6 -0.00164955482334048
9.7 -0.00156910508530712
9.8 -0.00149257892729671
9.9 -0.00141978499405303
10 -0.0013505412628264
};
\addplot [thick, color7]
table {%
0 -0.689959632546218
0.1 -0.624301292462183
0.2 -0.564891169548003
0.3 -0.511134667325142
0.4 -0.462493772651166
0.5 -0.41848167110341
0.6 -0.378657874776604
0.7 -0.342623813731869
0.8 -0.310018846974803
0.9 -0.280516653039193
1 -0.253821964052101
1.1 -0.22966761059375
1.2 -0.207811847776168
1.3 -0.188035935779103
1.4 -0.170141950628374
1.5 -0.153950803306215
1.6 -0.139300447368194
1.7 -0.126044257127929
1.8 -0.114049560177934
1.9 -0.103196309559578
2 -0.093375882292769
2.1 -0.0844899922406598
2.2 -0.0764497064289588
2.3 -0.0691745549748287
2.4 -0.0625917257172536
2.5 -0.0566353354884582
2.6 -0.051245770733021
2.7 -0.0463690908753742
2.8 -0.0419564884643934
2.9 -0.0379638006920227
3 -0.0343510673970474
3.1 -0.0310821311303699
3.2 -0.0281242752791055
3.3 -0.0254478966277252
3.4 -0.0230262090791239
3.5 -0.0208349755703581
3.6 -0.0188522654999727
3.7 -0.0170582352391715
3.8 -0.0154349295301101
3.9 -0.0139661017836401
4 -0.0126370514779849
4.1 -0.0114344770309762
4.2 -0.0103463426733489
4.3 -0.00936175799071749
4.4 -0.00847086892864761
4.5 -0.0076647591699679
4.6 -0.00693536089727053
4.7 -0.00627537404748365
4.8 -0.00567819325038486
4.9 -0.00513784171983736
5 -0.00464891143610512
5.1 -0.00420650902057344
5.2 -0.00380620676117101
5.3 -0.00344399829833921
5.4 -0.0031162585280399
5.5 -0.00281970732049475
5.6 -0.00255137669154423
5.7 -0.00230858109806482
5.8 -0.00208889056015031
5.9 -0.00189010634105675
6 -0.00171023894150599
6.1 -0.00154748818810776
6.2 -0.0014002252166196
6.3 -0.00126697616972596
6.4 -0.00114640744617903
6.5 -0.00103731235366901
6.6 -0.000938599031841937
6.7 -0.000849279524594244
6.8 -0.000768459892275981
6.9 -0.000695331264842558
7 -0.000629161746411103
7.1 -0.000569289090201118
7.2 -0.000515114070544942
7.3 -0.00046609448563753
7.4 -0.000421739730996573
7.5 -0.000381605889329582
7.6 -0.000345291287659859
7.7 -0.000312432477247881
7.8 -0.000282700596075186
7.9 -0.000255798077481412
8 -0.000231455672018294
8.1 -0.000209429752710627
8.2 -0.000189499876754262
8.3 -0.000171466579252065
8.4 -0.000155149376901559
8.5 -0.00014038496165733
8.6 -0.000127025566288665
8.7 -0.000114937485477022
8.8 -0.000103999737646326
8.9 -9.41028541401809e-05
9 -8.51477836217168e-05
9.1 -7.70449007355817e-05
9.2 -6.9713109106119e-05
9.3 -6.30790296985428e-05
9.4 -5.7076266416406e-05
9.5 -5.16447415871979e-05
9.6 -4.673009468463e-05
9.7 -4.22831382709613e-05
9.8 -3.82593657111616e-05
9.9 -3.461850573757e-05
10 -3.13241193996699e-05
};
\addplot [thick, color8]
table {%
0 -0.183019241405419
0.1 -0.157526120916913
0.2 -0.135583988768498
0.3 -0.116698220608578
0.4 -0.100443089312416
0.5 -0.0864521681479691
0.6 -0.0744100707041915
0.7 -0.0640453413814411
0.8 -0.0551243361798603
0.9 -0.0474459558451278
1 -0.040837112644996
1.1 -0.0351488285876981
1.2 -0.0302528771274085
1.3 -0.0260388926533456
1.4 -0.022411882603987
1.5 -0.019290086123929
1.6 -0.0166031309927662
1.7 -0.0142904472790772
1.8 -0.0122999019597556
1.9 -0.0105866237259791
2 -0.00911199148433148
2.1 -0.00784276375165125
2.2 -0.00675032931823268
2.3 -0.00581006228766737
2.4 -0.0050007669544898
2.5 -0.00430420000593992
2.6 -0.00370465927720841
2.7 -0.00318862978979707
2.8 -0.00274447909391136
2.9 -0.00236219504722728
3 -0.00203316011898619
3.1 -0.00174995713172608
3.2 -0.00150620206164391
3.3 -0.00129640012852988
3.4 -0.00111582193122908
3.5 -0.000960396836443403
3.6 -0.00082662121762464
3.7 -0.000711479475456106
3.8 -0.000612376059551847
3.9 -0.000527076958992095
4 -0.00045365934275815
4.1 -0.000390468214855458
4.2 -0.000336079107030138
4.3 -0.000289265968094046
4.4 -0.000248973526010235
4.5 -0.000214293499718979
4.6 -0.000184444124467539
4.7 -0.000158752529096889
4.8 -0.000136639568034391
4.9 -0.000117606766065551
5 -0.000101225081597114
5.1 -8.71252351076629e-05
5.2 -7.49893847677707e-05
5.3 -6.4543961579705e-05
5.4 -5.55535025245357e-05
5.5 -4.78153427038858e-05
5.6 -4.1155046823324e-05
5.7 -3.54224770335171e-05
5.8 -3.04884084899731e-05
5.9 -2.62416163385876e-05
6 -2.25863684594901e-05
6.1 -1.94402674540128e-05
6.2 -1.67323932240004e-05
6.3 -1.44017042748718e-05
6.4 -1.23956617057996e-05
6.5 -1.06690448656716e-05
6.6 -9.18293197905984e-06
6.7 -7.90382276361079e-06
6.8 -6.80288325911882e-06
6.9 -5.85529584597178e-06
7 -5.03969980128471e-06
7.1 -4.33770977983916e-06
7.2 -3.73350136918804e-06
7.3 -3.21345437058473e-06
7.4 -2.76584577081629e-06
7.5 -2.38058547877285e-06
7.6 -2.04898887225247e-06
7.7 -1.76358102824126e-06
7.8 -1.51792822021546e-06
7.9 -1.3064928889861e-06
8 -1.1245088129691e-06
8.1 -9.67873666865571e-07
8.2 -8.33056547383526e-07
8.3 -7.17018377310863e-07
8.4 -6.17143398745634e-07
8.5 -5.31180207943849e-07
8.6 -4.57191004096802e-07
8.7 -3.93507906079615e-07
8.8 -3.38695355650565e-07
8.9 -2.91517756156878e-07
9 -2.50911620133798e-07
9.1 -2.15961594954295e-07
9.2 -1.85879829486812e-07
9.3 -1.59988214107343e-07
9.4 -1.37703094177377e-07
9.5 -1.18522113540243e-07
9.6 -1.02012890620951e-07
9.7 -8.7803270598652e-08
9.8 -7.55729376022973e-08
9.9 -6.50461923351608e-08
10 -5.59857386700502e-08
};
\addplot [thick, color9]
table {%
0 -0.321258808349837
0.1 -0.263024466093193
0.2 -0.215346219202408
0.3 -0.176310572220075
0.4 -0.144350887569343
0.5 -0.118184510887112
0.6 -0.0967612936007456
0.7 -0.0792214467785242
0.8 -0.0648610347808933
0.9 -0.0531037238515619
1 -0.0434776518202174
1.1 -0.0355964906168108
1.2 -0.029143941569616
1.3 -0.0238610412289328
1.4 -0.0195357682545685
1.5 -0.0159945342549988
1.6 -0.0130952170757051
1.7 -0.0107214569380895
1.8 -0.00877798651299257
1.9 -0.00718680750826765
2 -0.00588406032344693
2.1 -0.00481746113974795
2.2 -0.0039442035868455
2.3 -0.00322924077292636
2.4 -0.00264387872986326
2.5 -0.0021646248235225
2.6 -0.00177224491186832
2.7 -0.0014509914113068
2.8 -0.00118797129086291
2.9 -0.000972628629577102
3 -0.000796320970332387
3.1 -0.000651972467705528
3.2 -0.000533789909443846
3.3 -0.000437030214517437
3.4 -0.000357810076622583
3.5 -0.000292950113464988
3.6 -0.000239847266984256
3.7 -0.000196370333494233
3.8 -0.000160774430996394
3.9 -0.000131630970937713
4 -0.000107770323936524
4.1 -8.8234878448154e-05
4.2 -7.22406084517208e-05
4.3 -5.91456077326108e-05
4.4 -4.84243279321589e-05
4.5 -3.96464864471785e-05
4.6 -3.24597976775892e-05
4.7 -2.65758345691924e-05
4.8 -2.17584530221859e-05
4.9 -1.7814314600198e-05
5 -1.45851271798753e-05
5.1 -1.19412921313206e-05
5.2 -9.77670307095169e-06
5.3 -8.00448743934479e-06
5.4 -6.55352000053511e-06
5.5 -5.36556833670432e-06
5.6 -4.39295577638055e-06
5.7 -3.59664796241443e-06
5.8 -2.94468626620137e-06
5.9 -2.41090517542897e-06
6 -1.97388218113603e-06
6.1 -1.61607801590913e-06
6.2 -1.32313274220364e-06
6.3 -1.08328943775049e-06
6.4 -8.86922348480867e-07
6.5 -7.26150573324147e-07
6.6 -5.94521776822732e-07
6.7 -4.86753233136517e-07
6.8 -3.98519812275233e-07
6.9 -3.26280397103251e-07
7 -2.67135766257587e-07
7.1 -2.18712238236884e-07
7.2 -1.79066406605366e-07
7.3 -1.46607144937971e-07
7.4 -1.20031749220595e-07
7.5 -9.82736554078212e-08
7.6 -8.04596349407438e-08
7.7 -6.58747484616473e-08
7.8 -5.3933653446081e-08
7.9 -4.41571116338579e-08
8 -3.61527562192654e-08
8.1 -2.95993443505393e-08
8.2 -2.42338643874351e-08
8.3 -1.98409809970546e-08
8.4 -1.62443922480104e-08
8.5 -1.32997544505242e-08
8.6 -1.08888889710501e-08
8.7 -8.9150391097137e-09
8.8 -7.29898760876146e-09
8.9 -5.9758765158513e-09
9 -4.89260477964937e-09
9.1 -4.0056969269342e-09
9.2 -3.2795581944578e-09
9.3 -2.6850459772243e-09
9.4 -2.19830056935599e-09
9.5 -1.79978720815654e-09
9.6 -1.4735120207332e-09
9.7 -1.20638055101097e-09
9.8 -9.87671638918997e-10
9.9 -8.08607972246911e-10
10 -6.62003132867461e-10
};
\end{axis}

\end{tikzpicture}

%% file: hxz-tdmd.tex
\begin{tikzpicture}

\definecolor{color0}{rgb}{0.254627,0.013882,0.615419}
\definecolor{color1}{rgb}{0.399411,0.000859,0.656133}
\definecolor{color2}{rgb}{0.534952,0.031217,0.650165}
\definecolor{color3}{rgb}{0.650746,0.125309,0.595617}
\definecolor{color4}{rgb}{0.752312,0.227133,0.513149}
\definecolor{color5}{rgb}{0.836801,0.329105,0.430905}
\definecolor{color6}{rgb}{0.907365,0.434524,0.35297}
\definecolor{color7}{rgb}{0.959424,0.543431,0.278701}
\definecolor{color8}{rgb}{0.990681,0.669558,0.201642}
\definecolor{color9}{rgb}{0.988648,0.809579,0.145357}

\begin{axis}[
height=\figureheight,
tick align=outside,
tick pos=left,
width=.95\linewidth,
x grid style={white!69.0196078431373!black},
xlabel={time \(\displaystyle t\)},
xmin=0, xmax=10,
xtick style={color=black},
y grid style={white!69.0196078431373!black},
ymin=-1.15868134758029, ymax=1.56544929251866,
ytick style={color=black}
]
\addplot [thick, color0]
table {%
0 0.594494225663158
0.1 0.0075229794182554
0.2 0.00752297941834273
0.3 0.00752297941835765
0.4 0.00752297941836946
0.5 0.00752297941837879
0.6 0.00752297941838596
0.7 0.00752297941839164
0.8 0.00752297941839619
0.9 0.00752297941839956
1 0.0075229794184022
1.1 0.0075229794184042
1.2 0.0075229794184059
1.3 0.00752297941840701
1.4 0.00752297941840785
1.5 0.00752297941840868
1.6 0.00752297941840924
1.7 0.00752297941840974
1.8 0.00752297941840994
1.9 0.00752297941841026
2 0.0075229794184107
2.1 0.00752297941841055
2.2 0.007522979418411
2.3 0.00752297941841142
2.4 0.0075229794184116
2.5 0.00752297941841156
2.6 0.00752297941841206
2.7 0.00752297941841232
2.8 0.00752297941841228
2.9 0.00752297941841261
3 0.00752297941841294
3.1 0.00752297941841333
3.2 0.00752297941841337
3.3 0.00752297941841378
3.4 0.00752297941841387
3.5 0.00752297941841435
3.6 0.00752297941841476
3.7 0.00752297941841489
3.8 0.00752297941841533
3.9 0.00752297941841549
4 0.00752297941841571
4.1 0.00752297941841587
4.2 0.00752297941841629
4.3 0.00752297941841645
4.4 0.00752297941841676
4.5 0.00752297941841718
4.6 0.00752297941841738
4.7 0.0075229794184178
4.8 0.00752297941841768
4.9 0.00752297941841807
5 0.00752297941841816
5.1 0.00752297941841858
5.2 0.00752297941841894
5.3 0.00752297941841901
5.4 0.00752297941841949
5.5 0.00752297941841943
5.6 0.00752297941841967
5.7 0.00752297941842005
5.8 0.00752297941842014
5.9 0.00752297941842046
6 0.00752297941842054
6.1 0.00752297941842063
6.2 0.00752297941842093
6.3 0.00752297941842123
6.4 0.00752297941842128
6.5 0.00752297941842169
6.6 0.0075229794184216
6.7 0.00752297941842177
6.8 0.00752297941842206
6.9 0.00752297941842199
7 0.0075229794184221
7.1 0.00752297941842238
7.2 0.00752297941842236
7.3 0.00752297941842266
7.4 0.00752297941842302
7.5 0.00752297941842291
7.6 0.0075229794184233
7.7 0.00752297941842337
7.8 0.00752297941842338
7.9 0.00752297941842359
8 0.00752297941842374
8.1 0.00752297941842377
8.2 0.00752297941842382
8.3 0.00752297941842392
8.4 0.00752297941842412
8.5 0.00752297941842426
8.6 0.00752297941842441
8.7 0.00752297941842454
8.8 0.00752297941842436
8.9 0.00752297941842461
9 0.00752297941842481
9.1 0.00752297941842478
9.2 0.00752297941842504
9.3 0.00752297941842512
9.4 0.00752297941842512
9.5 0.00752297941842522
9.6 0.00752297941842528
9.7 0.00752297941842548
9.8 0.00752297941842547
9.9 0.0075229794184257
10 0.00752297941842572
};
\addplot [thick, color1]
table {%
0 0.360406293128509
0.1 0.520364566830696
0.2 0.50599308850803
0.3 0.492322515453984
0.4 0.479318664115183
0.5 0.466949018089877
0.6 0.455182646819944
0.7 0.44399012824838
0.8 0.433343475248834
0.9 0.423216065643215
1 0.413582575632376
1.1 0.40441891647343
1.2 0.395702174245343
1.3 0.387410552552197
1.4 0.379523318020847
1.5 0.372020748456687
1.6 0.364884083527893
1.7 0.358095477854822
1.8 0.351637956387263
1.9 0.345495371957974
2 0.339652364906355
2.1 0.334094324671289
2.2 0.328807353257134
2.3 0.323778230481496
2.4 0.318994380917881
2.5 0.314443842450586
2.6 0.310115236363171
2.7 0.30599773888575
2.8 0.302081054129919
2.9 0.298355388343678
3 0.294811425421951
3.1 0.291440303611464
3.2 0.288233593351752
3.3 0.285183276196866
3.4 0.282281724765079
3.5 0.279521683666461
3.6 0.276896251360624
3.7 0.274398862899277
3.8 0.272023273510435
3.9 0.269763542983237
4 0.267614020814323
4.1 0.265569332078636
4.2 0.263624363989305
4.3 0.261774253113019
4.4 0.260014373208906
4.5 0.258340323660527
4.6 0.256747918472036
4.7 0.255233175801016
4.8 0.253792308001796
4.9 0.252421712154361
5 0.251117961055183
5.1 0.24987779464742
5.2 0.248698111869078
5.3 0.247575962898743
5.4 0.246508541779486
5.5 0.245493179402517
5.6 0.244527336833012
5.7 0.243608598961464
5.8 0.242734668464644
5.9 0.2419033600611
6 0.241112595046815
6.1 0.240360396097362
6.2 0.239644882323563
6.3 0.238964264568289
6.4 0.238316840932636
6.5 0.237700992520286
6.6 0.237115179389426
6.7 0.236557936702093
6.8 0.236027871061315
6.9 0.23552365702689
7 0.235044033801098
7.1 0.234587802076051
7.2 0.234153821034796
7.3 0.233741005498679
7.4 0.233348323213833
7.5 0.232974792270008
7.6 0.23261947864528
7.7 0.232281493870513
7.8 0.231959992807721
7.9 0.231654171536785
8 0.231363265345233
8.1 0.231086546816059
8.2 0.230823324008805
8.3 0.230572938729345
8.4 0.23033476488406
8.5 0.23010820691428
8.6 0.229892698307069
8.7 0.229687700178658
8.8 0.229492699926945
8.9 0.229307209949731
9 0.229130766425455
9.1 0.228962928153402
9.2 0.228803275450467
9.3 0.228651409101734
9.4 0.228506949362229
9.5 0.228369535007355
9.6 0.228238822429651
9.7 0.228114484779586
9.8 0.227996211148272
9.9 0.227883705790023
10 0.227776687382842
};
\addplot [thick, color2]
table {%
0 -0.258714306893949
0.1 -0.360685140610963
0.2 -0.301304696261305
0.3 -0.247575048314176
0.4 -0.198958452393717
0.5 -0.154968337267353
0.6 -0.11516443507731
0.7 -0.0791483749919174
0.8 -0.0465596961764231
0.9 -0.0170722401798088
1 0.00960911336861638
1.1 0.0337514004230777
1.2 0.0555962451069195
1.3 0.0753622779680432
1.4 0.0932473241069165
1.5 0.109430383076669
1.6 0.124073420370783
1.7 0.137322988428194
1.8 0.149311693379354
1.9 0.160159522212957
2 0.16997504364605
2.1 0.178856494716247
2.2 0.186892763971018
2.3 0.194164281094146
2.4 0.200743821873043
2.5 0.206697236563282
2.6 0.212084108940096
2.7 0.216958352632822
2.8 0.221368750710626
2.9 0.225359443919858
3 0.228970372459474
3.1 0.232237675715972
3.2 0.235194053958523
3.3 0.237869095614251
3.4 0.240289573399159
3.5 0.242479712268469
3.6 0.244461431868116
3.7 0.246254565913932
3.8 0.247877060694141
3.9 0.249345154681843
4 0.250673541055109
4.1 0.25187551475125
4.2 0.252963105527014
4.3 0.253947198356436
4.4 0.254837642371319
4.5 0.255643349434651
4.6 0.25637238333353
4.7 0.257032040484253
4.8 0.257628922957302
4.9 0.258169004553087
5 0.258657690589747
5.1 0.259099872001388
5.2 0.259499974288202
5.3 0.259862001808353
5.4 0.260189577854944
5.5 0.260485980919153
5.6 0.26075417750247
5.7 0.260996851806445
5.8 0.261216432597077
5.9 0.261415117512724
6 0.2615948950588
6.1 0.261757564509414
6.2 0.2619047539151
6.3 0.262037936396904
6.4 0.262158444889867
6.5 0.262267485483492
6.6 0.262366149492689
6.7 0.262455424380024
6.8 0.262536203638576
6.9 0.262609295734315
7 0.262675432197502
7.1 0.262735274944092
7.2 0.262789422900404
7.3 0.262838417997385
7.4 0.262882750594435
7.5 0.262922864387085
7.6 0.262959160847654
7.7 0.262992003243319
7.8 0.263021720271815
7.9 0.263048609351151
8 0.263072939596272
8.1 0.263094954512447
8.2 0.263114874432357
8.3 0.263132898721257
8.4 0.263149207772287
8.5 0.263163964811912
8.6 0.263177317533544
8.7 0.26318939957571
8.8 0.263200331859548
8.9 0.263210223799029
9 0.263219174396009
9.1 0.263227273231071
9.2 0.263234601360078
9.3 0.263241232125407
9.4 0.263247231889988
9.5 0.26325266070148
9.6 0.263257572893254
9.7 0.263262017628176
9.8 0.263266039390646
9.9 0.263269678431817
10 0.263272971172434
};
\addplot [thick, color3]
table {%
0 0.188962683197963
0.1 0.36821212056007
0.2 0.358838750242328
0.3 0.350771015643899
0.4 0.343827052123359
0.5 0.337850327333233
0.6 0.33270611263348
0.7 0.328278446008965
0.8 0.324467518028299
0.9 0.321187421917762
1 0.318364217031985
1.1 0.315934262067718
1.2 0.313842780447622
1.3 0.312042625534662
1.4 0.310493217842278
1.5 0.30915963028271
1.6 0.308011800832931
1.7 0.307023854869931
1.8 0.306173521899302
1.9 0.305441633528865
2 0.304811691370578
2.1 0.304269495130255
2.2 0.303802822501422
2.3 0.303401153647407
2.4 0.303055434060876
2.5 0.302757870455142
2.6 0.302501755086194
2.7 0.302281314545256
2.8 0.302091579613343
2.9 0.30192827324404
3 0.30178771414938
3.1 0.301666733815448
3.2 0.301562605077042
3.3 0.301472980641322
3.4 0.301395840174615
3.5 0.301329444759615
3.6 0.301272297696327
3.7 0.301223110763127
3.8 0.301180775177385
3.9 0.301144336601051
4 0.301112973627751
4.1 0.301085979266468
4.2 0.301062745004393
4.3 0.3010427470897
4.4 0.301025534725011
4.5 0.301010719905431
4.6 0.300997968672049
4.7 0.300986993583769
4.8 0.300977547237744
4.9 0.300969416692373
5 0.30096241866712
5.1 0.300956395410965
5.2 0.300951211146349
5.3 0.300946749008443
5.4 0.300942908410755
5.5 0.300939602777692
5.6 0.300936757592947
5.7 0.300934308719743
5.8 0.300932200955043
5.9 0.300930386785154
6 0.300928825314661
6.1 0.300927481344552
6.2 0.30092632457876
6.3 0.300925328941216
6.4 0.300924471988041
6.5 0.300923734401608
6.6 0.300923099555082
6.7 0.300922553137614
6.8 0.30092208283174
6.9 0.300921678035725
7 0.300921329624565
7.1 0.300921029744302
7.2 0.300920771634967
7.3 0.300920549478205
7.4 0.300920358266107
7.5 0.30092019368833
7.6 0.300920052034925
7.7 0.30091993011271
7.8 0.300919825173287
7.9 0.300919734851089
8 0.300919657110054
8.1 0.300919590197725
8.2 0.30091953260575
8.3 0.300919483035878
8.4 0.300919440370694
8.5 0.30091940364843
8.6 0.300919372041285
8.7 0.300919344836764
8.8 0.300919321421616
8.9 0.300919301268011
9 0.300919283921644
9.1 0.300919268991487
9.2 0.300919256140983
9.3 0.300919245080451
9.4 0.300919235560564
9.5 0.300919227366722
9.6 0.300919220314216
9.7 0.300919214244069
9.8 0.300919209019446
9.9 0.300919204522571
10 0.300919200652075
};
\addplot [thick, color4]
table {%
0 -0.197825699689978
0.1 -0.152396311886751
0.2 -0.0634057305772353
0.3 0.00945359507514175
0.4 0.0691057656352688
0.5 0.1179448321607
0.6 0.157930877876694
0.7 0.190668683198362
0.8 0.217472131203494
0.9 0.239416938373823
1 0.25738382687454
1.1 0.2720938710272
1.2 0.28413743655412
1.3 0.29399787402772
1.4 0.302070917426159
1.5 0.308680566327396
1.6 0.314092089149887
1.7 0.318522669305645
1.8 0.322150121533141
1.9 0.325120028227115
2 0.327551582171243
2.1 0.32954237016307
2.2 0.331172289514838
2.3 0.332506754613167
2.4 0.333599322228079
2.5 0.334493840934225
2.6 0.335226210908151
2.7 0.335825824728435
2.8 0.336316747003073
2.9 0.336718680166691
3 0.337047755208427
3.1 0.337317179065167
3.2 0.337537764662293
3.3 0.337718364874347
3.4 0.337866227821968
3.5 0.337987287764427
3.6 0.338086403262284
3.7 0.338167552168487
3.8 0.338233991273574
3.9 0.338288387012116
4 0.338332922476098
4.1 0.338369385030062
4.2 0.338399238044329
4.3 0.338423679625182
4.4 0.338443690699081
4.5 0.338460074380684
4.6 0.338473488204662
4.7 0.338484470514869
4.8 0.338493462069976
4.9 0.338500823732661
5 0.338506850952295
5.1 0.338511785622365
5.2 0.338515825788509
5.3 0.338519133596779
5.4 0.338521841801135
5.5 0.338524059091327
5.6 0.338525874454997
5.7 0.338527360749061
5.8 0.338528577623721
5.9 0.338529573916428
6 0.338530389611906
6.1 0.33853105744688
6.2 0.338531604223912
6.3 0.338532051887084
6.4 0.338532418402689
6.5 0.338532718480288
6.6 0.338532964163047
6.7 0.338533165311078
6.8 0.338533329997157
6.9 0.338533464830715
7 0.338533575223096
7.1 0.338533665604734
7.2 0.338533739602961
7.3 0.338533800187585
7.4 0.338533849790081
7.5 0.33853389040117
7.6 0.338533923650719
7.7 0.338533950873147
7.8 0.338533973160987
7.9 0.338533991408727
8 0.338534006348714
8.1 0.338534018580541
8.2 0.338534028595114
8.3 0.338534036794354
8.4 0.338534043507325
8.5 0.338534049003441
8.6 0.33853405350328
8.7 0.338534057187438
8.8 0.338534060203772
8.9 0.338534062673337
9 0.338534064695247
9.1 0.338534066350648
9.2 0.338534067705975
9.3 0.338534068815625
9.4 0.338534069724129
9.5 0.33853407046795
9.6 0.33853407107694
9.7 0.338534071575539
9.8 0.338534071983758
9.9 0.33853407231798
10 0.338534072591619
};
\addplot [thick, color5]
table {%
0 -0.168005520016769
0.1 0.045137876510257
0.2 0.0451378765102576
0.3 0.0451378765102599
0.4 0.045137876510262
0.5 0.045137876510264
0.6 0.0451378765102659
0.7 0.0451378765102676
0.8 0.0451378765102692
0.9 0.0451378765102707
1 0.045137876510272
1.1 0.0451378765102732
1.2 0.0451378765102744
1.3 0.0451378765102754
1.4 0.0451378765102763
1.5 0.0451378765102772
1.6 0.0451378765102781
1.7 0.0451378765102788
1.8 0.0451378765102795
1.9 0.0451378765102802
2 0.0451378765102808
2.1 0.0451378765102814
2.2 0.045137876510282
2.3 0.0451378765102825
2.4 0.045137876510283
2.5 0.0451378765102835
2.6 0.0451378765102839
2.7 0.0451378765102843
2.8 0.0451378765102848
2.9 0.0451378765102852
3 0.0451378765102856
3.1 0.045137876510286
3.2 0.0451378765102864
3.3 0.0451378765102867
3.4 0.0451378765102871
3.5 0.0451378765102874
3.6 0.0451378765102878
3.7 0.0451378765102882
3.8 0.0451378765102885
3.9 0.0451378765102889
4 0.0451378765102892
4.1 0.0451378765102896
4.2 0.0451378765102899
4.3 0.0451378765102903
4.4 0.0451378765102906
4.5 0.045137876510291
4.6 0.0451378765102913
4.7 0.0451378765102917
4.8 0.0451378765102921
4.9 0.0451378765102924
5 0.0451378765102928
5.1 0.0451378765102931
5.2 0.0451378765102935
5.3 0.0451378765102939
5.4 0.0451378765102942
5.5 0.0451378765102946
5.6 0.045137876510295
5.7 0.0451378765102953
5.8 0.0451378765102957
5.9 0.0451378765102961
6 0.0451378765102965
6.1 0.0451378765102969
6.2 0.0451378765102972
6.3 0.0451378765102976
6.4 0.045137876510298
6.5 0.0451378765102984
6.6 0.0451378765102988
6.7 0.0451378765102992
6.8 0.0451378765102995
6.9 0.0451378765102999
7 0.0451378765103003
7.1 0.0451378765103007
7.2 0.0451378765103011
7.3 0.0451378765103015
7.4 0.0451378765103019
7.5 0.0451378765103023
7.6 0.0451378765103027
7.7 0.0451378765103031
7.8 0.0451378765103035
7.9 0.0451378765103039
8 0.0451378765103043
8.1 0.0451378765103047
8.2 0.0451378765103051
8.3 0.0451378765103055
8.4 0.0451378765103059
8.5 0.0451378765103063
8.6 0.0451378765103067
8.7 0.0451378765103072
8.8 0.0451378765103076
8.9 0.045137876510308
9 0.0451378765103084
9.1 0.0451378765103088
9.2 0.0451378765103092
9.3 0.0451378765103096
9.4 0.04513787651031
9.5 0.0451378765103105
9.6 0.0451378765103109
9.7 0.0451378765103113
9.8 0.0451378765103117
9.9 0.0451378765103121
10 0.0451378765103125
};
\addplot [thick, color6]
table {%
0 1.44162517251417
0.1 -0.0736687960698213
0.2 -0.0700759264891747
0.3 -0.0666582832256625
0.4 -0.0634073203909612
0.5 -0.0603149088846337
0.6 -0.0573733160671491
0.7 -0.0545751864242566
0.8 -0.0519135231743688
0.9 -0.0493816707729625
1 -0.0469732982702514
1.1 -0.0446823834805136
1.2 -0.0425031979234907
1.3 -0.0404302925002031
1.4 -0.0384584838673644
1.5 -0.0365828414763234
1.6 -0.0347986752441241
1.7 -0.0331015238258553
1.8 -0.0314871434589646
1.9 -0.0299514973516418
2 -0.0284907455887362
2.1 -0.0271012355299689
2.2 -0.0257794926764296
2.3 -0.0245222119825194
2.4 -0.0233262495916152
2.5 -0.0221886149747908
2.6 -0.0211064634529367
2.7 -0.0200770890835808
2.8 -0.0190979178946226
2.9 -0.018166501448062
3 -0.0172805107176297
3.1 -0.0164377302650075
3.2 -0.0156360527000792
3.3 -0.0148734734113573
3.4 -0.0141480855534101
3.5 -0.0134580752787552
3.6 -0.0128017172022955
3.7 -0.0121773700869584
3.8 -0.0115834727397475
3.9 -0.0110185401079476
4 -0.0104811595657189
4.1 -0.0099699873817967
4.2 -0.00948374535946362
4.3 -0.00902121764039163
4.4 -0.00858124766436311
4.5 -0.00816273527726789
4.6 -0.00776463398014487
4.7 -0.00738594831238954
4.8 -0.00702573136258393
4.9 -0.00668308240072496
5 -0.00635714462593004
5.1 -0.00604710302398881
5.2 -0.00575218232940294
5.3 -0.00547164508681868
5.4 -0.00520478980700421
5.5 -0.00495094921276133
5.6 -0.00470948857038473
5.7 -0.00447980410249725
5.8 -0.00426132147829189
5.9 -0.00405349437740556
6 -0.00385580312383382
6.1 -0.00366775338646995
6.2 -0.0034888749430198
6.3 -0.00331872050420108
6.4 -0.00315686459528735
6.5 -0.00300290249219924
6.6 -0.00285644920948386
6.7 -0.00271713853765026
6.8 -0.00258462212745514
6.9 -0.00245856861884835
7 -0.0023386628124
7.1 -0.00222460488113785
7.2 -0.00211610962082361
7.3 -0.0020129057367938
7.4 -0.00191473516558194
7.5 -0.00182135242962513
7.6 -0.00173252402344263
7.7 -0.00164802782975036
7.8 -0.0015676525640519
7.9 -0.0014911972463175
8 -0.00141847069842897
8.1 -0.00134929106613504
8.2 -0.0012834853643209
8.3 -0.00122088904445538
8.4 -0.00116134558313385
8.5 -0.00110470609068818
8.6 -0.00105082893888507
8.7 -0.000999579406781667
8.8 -0.000950829343853
8.9 -0.000904456849549006
9 -0.000860345968479515
9.1 -0.000818386400465605
9.2 -0.000778473224731413
9.3 -0.000740506637547769
9.4 -0.00070439170267081
9.5 -0.000670038113951972
9.6 -0.000637359969525361
9.7 -0.000606275557008687
9.8 -0.000576707149179511
9.9 -0.000548580809616782
10 -0.000521826207821204
};
\addplot [thick, color7]
table {%
0 -1.03485722757579
0.1 0.15599735506349
0.2 0.141152243976107
0.3 0.127719831989332
0.4 0.115565683009223
0.5 0.104568154227637
0.6 0.0946171786801306
0.7 0.0856131636587862
0.8 0.077465993954916
0.9 0.0700941299557654
1 0.0634237915686617
1.1 0.0573882198050487
1.2 0.0519270086340904
1.3 0.0469855004188113
1.4 0.0425142388840947
1.5 0.0384684741416582
1.6 0.0348077148181309
1.7 0.0314953228037794
1.8 0.0284981465659907
1.9 0.0257861893575911
2 0.0233323089993188
2.1 0.0211119462317705
2.2 0.0191028789180788
2.3 0.0172849996372975
2.4 0.0156401144425743
2.5 0.0141517607700153
2.6 0.0128050426758126
2.7 0.0115864817526318
2.8 0.0104838822331814
2.9 0.00948620893087413
3 0.00858347679597094
3.1 0.00776665098184699
3.2 0.0070275564212099
3.3 0.00635879600727856
3.4 0.00575367656105222
3.5 0.00520614184372538
3.6 0.00471071194381427
3.7 0.00426242843236083
3.8 0.00385680473730915
3.9 0.00348978124038433
4 0.00315768464706832
4.1 0.00285719122303369
4.2 0.00258529352909335
4.3 0.00233927032173841
4.4 0.00211665931801838
4.5 0.00191523255218594
4.6 0.00173297407746675
4.7 0.00156805978978662
4.8 0.00141883917152486
4.9 0.00128381877257915
5 0.0011616472634149
5.1 0.00105110191050512
5.2 0.000951076338802309
5.3 0.000860569458765167
5.4 0.000778675447117836
5.5 0.000704574681066273
5.6 0.000637525535237653
5.7 0.000576856959244519
5.8 0.000521961761586983
5.9 0.000472290532675918
6 0.000427346146157324
6.1 0.00038667878350464
6.2 0.000349881432083623
6.3 0.000316585811633236
6.4 0.000286458688393002
6.5 0.000259198539987425
6.6 0.000234532537688869
6.7 0.000212213815855749
6.8 0.000192019001218352
6.9 0.000173745977284163
7 0.000157211861487812
7.1 0.000142251174841136
7.2 0.000128714185763702
7.3 0.000116465411518873
7.4 0.000105382262257082
7.5 9.53538140952993e-05
7.6 8.62796989536778e-05
7.7 7.80691000379814e-05
7.8 7.06398429145506e-05
7.9 6.39175730810619e-05
8 5.78350118015727e-05
8.1 5.23312827583923e-05
8.2 4.73513027813852e-05
8.3 4.28452305571182e-05
8.4 3.87679678002066e-05
8.5 3.50787078945893e-05
8.6 3.17405274871405e-05
8.7 2.87200169463214e-05
8.8 2.59869459874017e-05
8.9 2.35139611176353e-05
9 2.12763118732058e-05
9.1 1.9251603108437e-05
9.2 1.7419570857412e-05
9.3 1.57618795256246e-05
9.4 1.42619383810471e-05
9.5 1.29047355086247e-05
9.6 1.16766875657608e-05
9.7 1.05655038359041e-05
9.8 9.5600632188152e-06
9.9 8.65030292687696e-06
10 7.82711777326875e-06
};
\addplot [thick, color8]
table {%
0 0.755850732791848
0.1 -0.0168232359562398
0.2 -0.0144798933768164
0.3 -0.0124629597272034
0.4 -0.0107269688470632
0.5 -0.00923278764952698
0.6 -0.00794673397458453
0.7 -0.00683981731845192
0.8 -0.0058870853232818
0.9 -0.00506706129564435
1 -0.00436126007419732
1.1 -0.0037537713331279
1.2 -0.00323090092810156
1.3 -0.00278086219985922
1.4 -0.00239351027676124
1.5 -0.00206011338686754
1.6 -0.00177315602442099
1.7 -0.00152616953366962
1.8 -0.00131358629101089
1.9 -0.0011306141984004
2 -0.000973128658827515
2.1 -0.00083757959874569
2.2 -0.000720911441536441
2.3 -0.000620494228031735
2.4 -0.000534064331397929
2.5 -0.000459673429963718
2.6 -0.000395644587725886
2.7 -0.00034053445249059
2.8 -0.00029310071951172
2.9 -0.000252274127185276
3 -0.000217134353519652
3.1 -0.000186889270035915
3.2 -0.000160857085433834
3.3 -0.000138450976503085
3.4 -0.00011916585982576
3.5 -0.000102567006075305
3.6 -8.8280240252793e-05
3.7 -7.59835069520202e-05
3.8 -6.5399610516095e-05
3.9 -5.62899664320327e-05
4 -4.84492231065045e-05
4.1 -4.17006327851303e-05
4.2 -3.58920672658813e-05
4.3 -3.08925885918881e-05
4.4 -2.65894974191983e-05
4.5 -2.28857925235939e-05
4.6 -1.96979841776117e-05
4.7 -1.69542121069184e-05
4.8 -1.45926256001593e-05
4.9 -1.25599892567812e-05
5 -1.08104829428662e-05
5.1 -9.30466890366501e-06
5.2 -8.00860274909228e-06
5.3 -6.89306827188034e-06
5.4 -5.93291884935787e-06
5.5 -5.10651058287931e-06
5.6 -4.3952143961145e-06
5.7 -3.78299609452362e-06
5.8 -3.25605491904857e-06
5.9 -2.80251244619625e-06
6 -2.41214482215105e-06
6.1 -2.07615229441509e-06
6.2 -1.7869608457538e-06
6.3 -1.5380514591691e-06
6.4 -1.32381316471884e-06
6.5 -1.13941655583392e-06
6.6 -9.80704923761108e-07
6.7 -8.44100556068525e-07
6.8 -7.26524087149918e-07
6.9 -6.25325082536168e-07
7 -5.38222292054136e-07
7.1 -4.63252225519142e-07
7.2 -3.98724891255073e-07
7.3 -3.4318569994661e-07
7.4 -2.9538267498283e-07
7.5 -2.54238230081374e-07
7.6 -2.18824878175927e-07
7.7 -1.88344323722237e-07
7.8 -1.62109467379337e-07
7.9 -1.39528917237827e-07
8 -1.20093657636489e-07
8.1 -1.03365574669056e-07
8.2 -8.89675802299983e-08
8.3 -7.65751115745783e-08
8.4 -6.59088149369824e-08
8.5 -5.67282483505266e-08
8.6 -4.88264614555839e-08
8.7 -4.20253304534574e-08
8.8 -3.61715427350112e-08
8.9 -3.11331409767052e-08
9 -2.67965483795318e-08
9.1 -2.30640085623762e-08
9.2 -1.98513817519647e-08
9.3 -1.70862482270356e-08
9.4 -1.47062757110961e-08
9.5 -1.2657814419878e-08
9.6 -1.08946874317201e-08
9.7 -9.37714995929984e-09
9.8 -8.07099335628503e-09
9.9 -6.94677396682886e-09
10 -5.97914934577504e-09
};
\addplot [thick, color9]
table {%
0 -0.791302798759913
0.1 0.122732596428407
0.2 0.100484951101044
0.3 0.0822701196879582
0.4 0.0673570770479337
0.5 0.0551473104165821
0.6 0.0451507989875892
0.7 0.0369663476571771
0.8 0.0302654856558985
0.9 0.0247792838633201
1 0.0202875617381445
1.1 0.0166100506999824
1.2 0.0135991593182553
1.3 0.0111340499498578
1.4 0.00911578910025014
1.5 0.00746337687494302
1.6 0.00611049616932203
1.7 0.00500285113038427
1.8 0.00409598807351167
1.9 0.00335351140001976
2 0.00274562291398889
2.1 0.00224792591603342
2.2 0.00184044607809267
2.3 0.00150682980351136
2.4 0.00123368789978433
2.5 0.00101005822324879
2.6 0.000826965729768241
2.7 0.00067706227469801
2.8 0.000554331706039362
2.9 0.000453848415135814
3 0.000371579654702595
3.1 0.000304223690518328
3.2 0.000249077291237472
3.3 0.000203927238224731
3.4 0.000166961501320121
3.5 0.000136696515706126
3.6 0.00011191764124251
3.7 9.16304146924909e-05
3.8 7.50206384213612e-05
3.9 6.14217037864413e-05
4 5.0287837791773e-05
4.1 4.11721993012811e-05
4.2 3.37089457351902e-05
4.3 2.75985505226346e-05
4.4 2.2595782048657e-05
4.5 1.84998616485025e-05
4.6 1.51464056547344e-05
4.7 1.24008281035815e-05
4.8 1.01529393274832e-05
4.9 8.31252365700134e-06
5 6.80571874914909e-06
5.1 5.57205123221001e-06
5.2 4.56200969703273e-06
5.3 3.73505763029587e-06
5.4 3.05800654194777e-06
5.5 2.50368399454942e-06
5.6 2.04984307787016e-06
5.7 1.67826956237044e-06
5.8 1.37405089820422e-06
5.9 1.12497772219305e-06
6 9.21053853245729e-07
6.1 7.54095110458083e-07
6.2 6.17400853251265e-07
6.3 5.05485061093325e-07
6.4 4.13856160319179e-07
6.5 3.38836761386643e-07
6.6 2.77416072428163e-07
6.7 2.27129065494616e-07
6.8 1.85957546442947e-07
6.9 1.52249157650899e-07
7 1.24651063120734e-07
7.1 1.02055654401315e-07
7.2 8.35560984100564e-08
7.3 6.84099430037916e-08
7.4 5.60093197862765e-08
7.5 4.58565482096968e-08
7.6 3.75441618939709e-08
7.7 3.07385555949752e-08
7.8 2.51665964162012e-08
7.9 2.06046620685602e-08
8 1.68696661397016e-08
8.1 1.38117101158852e-08
8.2 1.13080674789268e-08
8.3 9.25825826766924e-09
8.4 7.58001642703343e-09
8.5 6.20598823893037e-09
8.6 5.08102908185204e-09
8.7 4.15999043372473e-09
8.8 3.40590776588823e-09
8.9 2.78851710992789e-09
9 2.28304038174532e-09
9.1 1.86919106061334e-09
9.2 1.53035988915973e-09
9.3 1.2529483926077e-09
9.4 1.02582305497312e-09
9.5 8.39868568788815e-10
9.6 6.87621902126234e-10
9.7 5.62972876649859e-10
9.8 4.60918903738561e-10
9.9 3.77364179834916e-10
10 3.08955343947259e-10
};
\end{axis}

\end{tikzpicture}

%% file: HeiU21-ArXiv.bbl
\begin{thebibliography}{10}

\bibitem{Ant05}
{\sc A.~C. Antoulas}, {\em Approximation of large-scale dynamical systems},
  Advances in Design and Control, SIAM, Philadelphia, 2005.

\bibitem{AntBG20}
{\sc A.~C. Antoulas, C.~A. Beattie, and S.~Güğercin}, {\em Interpolatory
  Methods for Model Reduction}, SIAM, Philadelphia, PA, USA, 2020.

\bibitem{BauBF14}
{\sc U.~Baur, P.~Benner, and L.~Feng}, {\em {Model Order Reduction for Linear
  and Nonlinear Systems: A System-Theoretic Perspective}}, Arch. Comput.
  Methods Eng., 21 (2014), pp.~331--358.

\bibitem{BeaG12}
{\sc C.~Beattie and S.~Gugercin}, {\em {Realization-independent
  H2-approximation}}, in Proc. IEEE Conf. Decis. Control, Maui, HI, USA, 2012,
  pp.~4953--4958.

\bibitem{BenCOW17}
{\sc P.~Benner, A.~Cohen, M.~Ohlberger, and K.~Willcox}, {\em Model Reduction
  and Approximation}, SIAM, Philadelphia, PA, 2017.

\bibitem{BenGW15}
{\sc P.~Benner, S.~Gugercin, and K.~Willcox}, {\em A survey of projection-based
  model reduction methods for parametric dynamical systems}, SIAM review, 57
  (2015), pp.~483--531.

\bibitem{DrmGB15a}
{\sc Z.~Drma{\v{c}}, S.~Gugercin, and C.~Beattie}, {\em Quadrature-based vector
  fitting for discretized $\mathcal{H}_2$ approximation}, SIAM J. Sci. Comput.,
  37 (2015), pp.~A625--A652.

\bibitem{DrmGB15b}
\leavevmode\vrule height 2pt depth -1.6pt width 23pt, {\em Vector fitting for
  matrix-valued rational approximation}, SIAM J. Sci. Comput., 37 (2015),
  pp.~A2345--A2379.

\bibitem{GolV96}
{\sc G.~H. Golub and C.~F. {Van~Loan}}, {\em Matrix Computations}, Johns
  Hopkins University Press, Baltimore, third~ed., 1996.

\bibitem{GusS99}
{\sc B.~Gustavsen and A.~Semlyen}, {\em Rational approximation of frequency
  domain responses by vector fitting}, IEEE Trans. Power Deliv., 14 (1999),
  pp.~1052--1061.

\bibitem{HaiNW08}
{\sc E.~Hairer, S.~N{\o}rsett, and G.~Wanner}, {\em Solving Ordinary
  Differential Equations I: Nonstiff Problems}, Springer Series in
  Computational Mathematics, Springer Berlin Heidelberg, 2008.

\bibitem{HesRS16}
{\sc J.~S. Hesthaven, G.~Rozza, and B.~Stamm}, {\em Certified reduced basis
  methods for parametrized partial differential equations}, Springer, 2016.

\bibitem{Hig08}
{\sc N.~Higham}, {\em Functions of Matrices: Theory and Computation}, Other
  Titles in Applied Mathematics, SIAM, 2008.

\bibitem{KunM06}
{\sc P.~Kunkel and V.~Mehrmann}, {\em Differential-Algebraic Equations.
  Analysis and Numerical Solution}, European Mathematical Society, 2006.

\bibitem{KutBBP16}
{\sc J.~Kutz, S.~Brunton, B.~Brunton, and J.~Proctor}, {\em Dynamic Mode
  Decomposition}, SIAM, Philadelphia, PA, 2016.

\bibitem{MayA07}
{\sc A.~J. {Mayo} and A.~C. {Antoulas}}, {\em A framework for the solution of
  the generalized realization problem}, Linear Algebra Appl., 425 (2007),
  pp.~634--662.

\bibitem{Mez05}
{\sc I.~{Mezi\'c}}, {\em Spectral properties of dynamical systems, model
  reduction and decompositions}, {Nonlinear Dyn.}, 41 (2005), pp.~309--325.

\bibitem{PehW16}
{\sc B.~Peherstorfer and K.~Willcox}, {\em Data-driven operator inference for
  nonintrusive projection-based model reduction}, Comput. Methods Appl. Mech.
  Engrg., 306 (2016), pp.~196--215.

\bibitem{QuaMN16}
{\sc A.~Quarteroni, A.~Manzoni, and F.~Negri}, {\em Reduced Basis Methods for
  Partial Differential Equations: An Introduction}, UNITEXT, Springer Cham,
  2016.

\bibitem{TuRLBK14}
{\sc J.~H. Tu, C.~W. Rowley, D.~M. Luchtenburg, S.~L. Brunton, and J.~N. Kutz},
  {\em On dynamic mode decomposition: Theory and applications}, J. Comput.
  Dyn., 1 (2014), pp.~391--421.

\end{thebibliography}
